\documentclass{amsart}
\usepackage{amssymb, enumerate}

\usepackage[dvipsnames]{xcolor}

\usepackage{amssymb}
\usepackage{amsthm}
\usepackage{amsmath}
\usepackage{paralist}
\usepackage{graphics}
\usepackage{epsfig}
\usepackage[colorlinks=true]{hyperref}
\hypersetup{urlcolor=blue, citecolor=red}

\usepackage{graphicx}

\newcommand{\dvol}{{\,d\operatorname{vol}}}
\newcommand\Symb{{\operatorname{Symb}}}
\newcommand\ADN{{Douglis-Nirenberg}}

\newtheorem{theorem}{Theorem}[section]
\newtheorem{lemma}[theorem]{Lemma}
\newtheorem{corollary}[theorem]{Corollary}
\newtheorem{proposition}[theorem]{Proposition}
\newtheorem{problem}[theorem]{Problem}

\theoremstyle{definition}
\newtheorem{definition}[theorem]{Definition}

\newtheorem{notation}[theorem]{Notation}

\newtheorem{assumption}[theorem]{Assumption}

\theoremstyle{remark}
\newtheorem{remark}[theorem]{Remark}

\usepackage[normalem]{ulem}
\renewcommand\sout{\bgroup\markoverwith
{\textcolor{red}{\rule[0.7ex]{3pt}{1.4pt}}}\ULon}

\newcommand\seq{\, = \,}
\newcommand\define{\mathrel{\ := \ }}
\newcommand\ede{\define}

\newcommand\Hom{\operatorname{Hom}}
\newcommand\End{\operatorname{End}}
\newcommand\supp{\operatorname{supp}}

\newcommand\Def{{\operatorname{Def}\,}}

\newcommand{\CC}{\mathbb C}

\newcommand{\II}{\mathbb{I}}

\newcommand{\RR}{\mathbb R}

\newcommand{\ZZ}{\mathbb Z}

\newcommand{\maA}{\mathcal A}
\newcommand{\maB}{\mathcal B}
\newcommand{\maC}{\mathcal C}
\newcommand{\maD}{\mathcal D}

\newcommand{\maF}{\mathcal F}

\newcommand{\maH}{\mathcal H}

\newcommand{\maL}{\mathcal L}

\newcommand{\maN}{\mathcal N}

\newcommand{\maP}{\mathcal P}
\newcommand{\maQ}{\mathcal Q}

\newcommand{\maS}{\mathcal S}

\newcommand{\maV}{\mathcal V}
\newcommand{\maW}{\mathcal W}

\newcommand{\CI}{\maC^{\infty}}
\newcommand{\CIc}{\maC^{\infty}_{\text{c}}}

\newcommand\pa{\partial}

\newcommand\cl{\operatorname{cl}}

\newcommand\bop{{\boldsymbol T_{\bsnu}}}
\newcommand\bopstar{{\boldsymbol T_{\bsnu}^{*}}}
\newcommand\bsu{\boldsymbol u}

\newcommand\bsw{\boldsymbol w}
\newcommand\bsnu{\boldsymbol \nu}
\newcommand\bsf{{\boldsymbol f}}
\newcommand\bsh{\boldsymbol h}
\newcommand\Dnu{\boldsymbol D_{\bsnu}}
\newcommand\bsL{\boldsymbol L}
\newcommand\bsXi{{\boldsymbol \Xi}}
\newcommand\bsK{{\boldsymbol K}}
\newcommand\bsS{{\boldsymbol S}}
\newcommand\bsP{{\boldsymbol P}}
\newcommand\Defstar{{\operatorname{Def}^{*}\,}}

\newcommand\tbop{{\tilde{\boldsymbol T}_{\bsnu}}}
\newcommand\tbopstar{{\tilde{\boldsymbol T}_{\bsnu}^{*}}}

\newcommand\loc{\operatorname{loc}}
\newcommand\comp{\operatorname{comp}}
\newcommand\JC{{\operatorname{\mathfrak L}}}
\newcommand\psdinv{{\bsXi^{(-1)}}}
\newcommand\Smbmn{{S^{m}(\RR^{n} \times \RR^{n})}}
\newcommand\pv{{\operatorname{p.\!\!v.}}}

\newcommand\bss{{\boldsymbol s}}
\newcommand\bst{{\boldsymbol t}}

\newcommand\jap[1]{\langle #1 \rangle}

\newcommand\cvector[2]{{\left(
  \begin{array}{c} {#1} \\ {#2}
  \end{array}
\right)}}

\numberwithin{equation}{section}

\begin{document}

\title[Well-posedness for a generalized Stokes operator]{Well-posedness
of a generalized Stokes operator on smooth bounded domains via layer-potentials}

\author[M. Kohr]{Mirela Kohr}
\address{Faculty of Mathematics and Computer Science,
Babe\c{s}-Bolyai University, 1 M. Kog\u{a}l\-niceanu Str., 400084
Cluj-Napoca, Romania} \email{mkohr@math.ubbcluj.ro}


\author[V. Nistor]{Victor Nistor}
\address{Universit\'{e} de Lorraine, CNRS, IECL, F-57000 Metz, France}
\email{victor.nistor@univ-lorraine.fr}

\author[W.L. Wendland]{Wolfgang L. Wendland}
\address{Institut f\"ur Angewandte Analysis und Numerische Simulation,
Universit\"at Stuttgart, Pfaffenwaldring 57, 70569 Stuttgart,
Germany}
\email{wendland@mathematik.uni-stuttgart.de}

\thanks{M.K. has been partially supported by
  AGC35124/31.10.2018. V.N. has been partially supported byANR grant
  OpART ANR-23-CE40-0016 https://victor-nistor.apps.math.cnrs.fr/}

\date\today


\date{\today}


\dedicatory{In memory of Professor Gabriela Kohr, with deep respect}

\begin{abstract}
  We prove the invertibility of the relevant single and double
  layer potentials associated to some generalizations of the
  Stokes operator on bounded domains. In order to do that, we first
  develop an ``algebra tool kit'' to deal with limit and jump relations of
  layer operators. We do that first on $\RR^{n}$ for operators acting
  on a distribution supported on $\{x_{n} = 0\}$ and then in general
  on (possibly non-compact) manifolds. We use these results to study
  the limit and jump relations of the layer potential operators
  associated to our generalized Stokes operators.
\end{abstract}

\maketitle
\tableofcontents

\section{Introduction}

  Let $\Omega$ be a smooth domain in a Riemannian manifold $M$
  which is on one side of its boundary $\Gamma := \pa \Omega$.
  By a \emph{domain,} we shall always mean an open and connected subset.
  The condition that $\Omega$ be on one side of its boundary $\Gamma$ means
  that $\Gamma = \pa \Omega_{-}$, where $\Omega_{-} := M \smallsetminus \overline{\Omega}$.
  Let $TM$ is the tangent bundle to $M$,
  let $\Def : \CI(M; TM) \to \CI(M; T^{*}M \otimes T^{*}M)$
  be the deformation operator, and let $\bsL := 2 \Defstar \Def$ be the
  ``Deformation Laplacian,'' as usual. (These operators are reviewed in
  Section \ref{sec.Def}.) For suitable ``potentials''
  $V, V_{0} \in \CI(M)$, $V, V_{0} \ge 0$, we study the operator
  \begin{equation} \label{eq.def.bsXi}
    \bsXi \ede \bsXi _{V,V_0} \ede \left(\begin{array}{ccc}  \bsL + V & \nabla \\
    \nabla^* &  -V_0
    \end{array}
  \right) : \CI(M; TM \oplus \CC) \to \CI(M; TM \oplus \CC)\,.
  \end{equation}
  using the method of layer potentials. It follows from the definition that,
  if $V$ and $V_{0}$ vanish on $\Omega$, $\bsXi$ coincides with the Stokes operator
  $\bsXi_{0, 0}$ in $\Omega$. (This is permited by the assumptions of most of our results.)

  We define our layer potential operators under the natural (and not too restricting) assumption
  that $\bsXi_{V, V_{0}}$ has a {\it Moore-Penrose pseudoinverse} $\psdinv$, see Equation
  \eqref{eq.def.psdinv}. We then prove the usual jump relations for these layer potential operators
  for general (possibly non-compact) manifolds $M$. Let $\frac12 + \bsK$ and $\bsS$ be defined
  as the boundary limits of the {\it double} and {\it single} layer potentials
  $\maD_{\rm{ST}}$ and $\maS_{\rm{ST}}$ (see Subsections \ref{ssec.5.2} and \ref{ssec.Jump}).
  We prove, in general, that they are elliptic with self-adjoint principal symbols.
  If $\Gamma$ is {\it compact,} we prove that the (classical) limit operators
  $\frac12 + \bsK$ and $\bsS$ are Fredholm. In case $M$ itself is compact and connected,
  then we prove the invertibility of the single and double layer potential operators $\bsS$ and
  $\frac12 + \bsK$ on the space $L^2(\Gamma ;TM)$ or on the orthogonal complement $\{\bsnu\}^\perp$
  of $\CC \bsnu$, $\bsnu $ being the outer unit normal to $\Gamma $, under some additional
  assumptions on $V$ and $V_{0}$
  (these assumptions are such that they still allow $V$ and $V_{0}$ to vanish on
  $\Omega$, so our results can be used to study the ``true Stokes system'' formulated
  next). See Theorems \ref{thm.K2}, \ref{thm.K2bis} and \ref{thm.S}.

  Let
  \begin{equation*}
    U \ede \cvector{\bsu}{p} \in L^2(\Omega ; TM \oplus \CC)\,,
  \end{equation*}
  so that $\bsu$ is the ``vector part'' of $U$. We then consider
  the Dirichlet boundary value problem
  \begin{equation}\label{eq.bvp.Brinkman}
    \begin{cases}
      \ \bsXi U \seq 0 & \mbox{ in } \Omega\\
      \ \bsu \seq \bsh & \mbox{ on } \pa \Omega\,.
    \end{cases}
  \end{equation}
  When $V$ and $V_{0}$ vanish identically on $\Omega$, this is nothing but the classical
  Dirichlet problem for the Stokes operator. In view of the applications
  that we have in mind, we will consider also the case when $V$ and $V_{0}$ do not
  vanish identically on $\Omega$.

  When $\bsXi$ has a Moore-Penrose pseudo-inverse,
  we define the single and double layer potential
  operators $\maS_{\rm{ST}}$, $\bsS$, $\maD_{\rm{ST}}$, and $\bsK$
  associated to the operator $\bsXi$. We
  prove that they satisfy the usual mapping, symbol, and  ``jump'' properties
  (Theorem \ref{thm.jump.rel} and Theorem \ref{jump-conormal-dl}). This is the
  case when $M$ is compact, in which case we prove that
  $\bsXi$ is Fredholm (and hence it has a Moore-Penrose pseudo-inverse $\psdinv$).
  Then, using the jump relations
  and symbol properties of the layer potential operators, we prove that $S$
  and $\frac12 + K$ are invertible. As it is well known, this implies
  the solvability of the Dirichlet problem \eqref{eq.bvp.Brinkman} (see e.g.
  \cite{D-M, KMNW-2025, M-T}).

The paper is organized as follows. The second section is devoted to some preliminary results
most of them related to pseudodifferential operators and the Fourier transform.
We also develop
the standard machinery needed to obtain the limit and jump relations of
the layer potential operators on a flat space. These results are, for the most part known,
but they are spread out through literature. Our presentation is much more concise
and some of our statements are more general than the ones that one can find in the
literature. Other than a few basic properties of pseudodifferential
operators and a few technical points, our presentation is complete and can
be regarded as an introduction to the subject of limit and jump relations
for potential operators. (Complete proofs can be found in \cite{KMNW-2025}.)
The third section extends the results of the second section to the case
of Riemannian manifolds.
The fourth section introduces the deformation operator $\Def$,
the Stokes operator $\bsXi = \bsXi_{V, V_{0}}$ (see Equation \eqref{eq.def.bsXi}),
and also a few other basic differential operators. If $V$ and $V_{0}$ vanish, then
$\bsXi_{0,0}$ becomes the usual Stokes operator. We obtain some properties of these operators,
by using various integration by parts formulas that are also obtained.
We also discuss Green's formulas and some of their useful consequences for the
Stokes operator $\bsXi$.
In the fifth section we construct the corresponding
single and double layer potentials (Definition \ref{def.lp}) and obtain some mapping properties
as well as representation formulas of them. We also prove the jump relations of the single
and double layer potentials
(Theorem \ref{thm.K1} and Theorem \ref{thm.jump.rel}).
In the sixth section we obtain Fredholm and invertibility properties of the Stokes layer
potential operators in $L^2$-based Sobolev spaces (Theorem \ref{thm.K2} and Theorem \ref{thm.S}).
The last section is devoted to the well-posedness of the Dirichlet problem for the generalized
Stokes system in $L^2$-based Sobolev spaces on a smooth domain of a compact manifold. We use the
invertibility results of the layer potential operators established in the previous section. A
consequence of this result related to the jump of the conormal derivative of a double layer
potential across the boundary is also established.

\subsection*{Short overview of the main connected results}
The method of layer potentials has a main role in the analysis of elliptic boundary value
problems in Euclidean setting or on manifolds as they provide the explicit form of the corresponding
solutions in various function spaces. There is an extensive list of literature devoted to this subject
and let us mention the following monographs
\cite{Massimo-book, H-W, D-I-M-GHA, M-W, Varnhorn} among many other very valuable publications.
Let us also mention a few of the most relevant references related especially
to the Stokes and Navier-Stokes equations on various domains in the
Euclidean spaces and on Riemannian manifolds. Fabes, Kenig and Verchota \cite{Fa-Ke-Ve} studied
$L^2$-boundary value problems for the constant coefficient Stokes system in Lipschitz domains of
Euclidean spaces. They obtained mapping properties of the Stokes layer potential operators and
well-posedness results for the corresponding Dirichlet and Neumann problems by using Rellich
formulas and layer potential methods.
Further extensions of these results to $L^p$, Sobolev, Bessel potential, and Besov spaces, and
well-posedness results for the main boundary value problems for the Stokes system with constant
coefficients in arbitrary Lipschitz domains in $\mathbb R^n$, together with optimal ranges of $p$,
have been obtained by Mitrea and Wright \cite{M-W} using layer potential methods. Other boundary
valued problems for the Stokes system in Sobolev spaces on Lipschitz domains via layer potential
theoretical methods have also been studied in \cite{B-H, Med-CVEE-16, Russo, Sa-Se, Varnhorn-2004}.
Mapping properties for the constant-coefficient Stokes and Brinkman layer potentials in standard
and weighted Sobolev spaces on $\mathbb R^3$ have been obtained in \cite{K-L-M-W}.

Dahlberg, Kenig and Verchota \cite{Da-Ke-Ve} studied the Lam\'{e} system of the linear
elasticity by using a layer potential theoretical method.
Costabel \cite{Costabel88} used a layer potential approach in the analysis of elliptic boundary
value problems on Lipschitz domains in the Euclidean spaces. Let us also mention the contributions of Chandler-Wilde et al. in \cite{Chandler-Wilde-Perfekt, Chandler-Wilde-1} related to boundary integral equations on on locally-dilation-invariant Lipschitz domains and fractal screens.
M. Dalla Riva, M. Lanza de Cristoforis, and P. Musolino \cite{Massimo-book} have studied singularly
perturbed boundary value problems by using a functional analytic approach proposed by the second
named author. This method, which is based also on a layer potential analysis, has been used in the
study of various linear and nonlinear elliptic problems.
In their recent book \cite{D-I-M-GHA}, D. Mitrea, I. Mitrea and M. Mitrea have made a rigorous
interplay between Harmonic Analysis, Geometric Measure Theory, Function Space Theory, and Partial
Differential Equations, with many applications in the study of boundary problems for complex
coefficient elliptic systems in various geometric settings, including the
class of Lipschitz domains. The theory of Fredholm and layer potential operators in Euclidean
spaces and in Riemannian manifolds plays an important role in their book \cite{D-I-M-GHA}, and
also in our recent book \cite{KMNW-2025}.

Next we provide a brief overview of some of the significant contributions to elliptic boundary
value problems on compact manifolds, especially on those related to the Stokes and Navier-Stokes
systems that use layer potentials. Dindo\u{s} and Mitrea \cite{D-M} used the
mapping properties of Stokes layer potentials in Sobolev and Besov spaces to obtain the
well-posedness of the Poisson problem for the Stokes and Navier-Stokes systems with Dirichlet
boundary condition on $C^1$ and Lipschitz domains in compact Riemannian manifolds.
Well-posedness results for boundary value problems for the Laplace-Beltrami operator and
the Hodge-Laplacian on compact manifolds and various properties of the corresponding
boundary integral operators have been obtained by Mitrea and Taylor \cite{M-T1}. The authors
in \cite{K-W} and \cite{Khavinson-Putinar} have developed a variational approach
in order to construct the layer potentials for the Stokes system with $L^\infty $ coefficients
in Lipschitz domains on a compact Riemannian manifold.
Benavides, Nochetto, and Shakipov \cite{BNS} used a variational approach and studied the
$L^p$-based ($p\in (1,\infty )$) well-posedness and Sobolev regularity for the weak formulations
of the (stationary) tangent Stokes and tangent Navier-Stokes systems on a
compact and connected $d$-dimensional manifold without boundary of class $C^m$, $m\geq 2$,
embedded in $\mathbb R^{d+1}$, in terms of the regularity of the source terms and the manifold.
Perfekt and Putinar have studied layer potentials on polyhedral domains 
\cite{PutinarPerfekt1, PutinarPerfekt2}.

Gro\ss e, Kohr, and Nistor \cite{Grosse-Kohr-Nistor-23} proved the $L^2$-unique continuation
property for the deformation operator on manifolds with bounded geometry. Amann \cite{AmannFunctSp}
and Ammann, Gro{\ss}e, Nistor \cite{AGN-1, AGN-2} studied function spaces on manifolds with bounded 
geometry, and Gro{\ss}e and Nistor \cite{GN17} used uniform Shapiro-Lopatinski conditions to obtain 
well-posedness and regularity results for boundary value problems on manifolds with bounded geometry.

Lewis and Parenti \cite{LewisParenti} obtained the first results for the layer potentials
for the Laplace operator on manifolds with cylindrical ends. Mitrea and Nistor
\cite{Mitrea-Nistor} expanded these results and used a layer potential approach to
obtain the well-posedness of the Dirichlet problem for the Laplace operator on a
manifold with boundary and cylindrical ends. Mitrea and Nistor \cite{Mitrea-Nistor} and
Kohr, Nistor and Wendland \cite{KMNW-2025, KNW-22, KohrNistor-Stokes} developed an
essentially translation invariant pseudodifferential calculus on manifolds with
cylindrical ends, which is very useful to provide the invertibility and structure
of the Stokes operator and, as a consequence, the construction and the invertibility
of the Stokes layer potential operators on manifolds with cylindrical ends. This is
the purpose of a forthcoming paper of us.

An important operator in the structure of the Stokes and Navier-Stokes equations
on Riemannian manifolds is the deformation Laplacian $2\Def^*\Def$. In their seminal
paper \cite{Ebin-Marsd}, Ebin and Marsden mentioned that the convenient Laplace type
operator to describe the Navier-Stokes equation on closed Riemannian manifolds is the
deformation Laplacian. This is the choice that we also consider in the description of
the Stokes operator $\boldsymbol\Xi $ of Equation \eqref{eq.def.bsXi}.

In this paper, we allow our ambient manifold $M$ to be arbitrary, as long as this
is possible. Once this is not possible anymore, we assume that $M$ is compact (without
boundary).

\subsection*{Acknowledgements}
  We thank Massimo Lanza de Cristoforis, Dorina, Irina, and Marius Mitrea, Sergey
  Mikhailov, and Mihai Putinar for useful discussions.

\section{Normal lateral limits at $x_{n} = 0$ of pseudodifferential
operators on $\RR^{n}$}
\label{sec.sec2}


We shall need several results on the ``lateral normal limits'' at the boundary of
the values of a pseudodifferential operator. They are closely related to the
``jump relations'' that play such an important role in the study of layer potentials.
These results are developped in this section, in Sections \ref{sec.sec3}, and
in Subsection \ref{ssec.Jump}. The contents of this section are classical, see
\cite{KMNW-2025}, which we follow closely. See also \cite{H-W}.

\subsection{Motation: traces, normal lateral limits, and more}
\label{ssec.problem}
In the following, we let $\Gamma \subset \RR^{n}$ be the hyperplane
\begin{equation}\label{eq.def.Gamma}
  \Gamma \ede \{x = (x', x_{n})
  \in \RR^{n-1} \times \RR \mid x_{n} = 0 \}\,,
\end{equation}
with the induced Euclidean measure. Let $e_{n} := (0, \ldots, 0, 1) \in \RR^{n}$.
We parameterize $\Gamma$ and its translations
$\Gamma + \epsilon e_{n}$ via the diffeomorphism
$\RR^{n} \ni x' \mapsto (x', \epsilon) \in \Gamma + \epsilon e_{n}$.

We shall repeatedly use the {\it Fourier transform,} which in this paper is defined
for $f \in L^{1}(\RR^{n})$ by
\begin{equation}\label{eq.def.inv.F}
  \begin{gathered}
    \hat f(x) \seq \maF f(x) \ede \int_{\RR^{n}}
    e^{-\imath x \cdot \tau} f(\tau)
    \, d\tau\qquad \mbox{and hence}\\
    \maF^{-1} f(x) \seq \frac1{(2\pi)^{n}} \int_{\RR^{n}}
    e^{\imath x \cdot \tau} f(\tau)
    \, d\tau\,,
  \end{gathered}
\end{equation}
where $x \cdot \xi := \sum_{j=1}^{n} x_{j}  \xi_{j}$ and $\imath ^2=-1$.
We have that $\maF f, \maF^{-1}f \in \maC_{0}(\RR^{n})$.
Let $\maS(\RR^{n})$ denote the space of Schwartz functions (i.e. smooth rapidly decaying
functions at infinity) and let $\maS'(\RR^{n})$ be the dual of $\maS(\RR^{n})$.
Then we have isomorphisms $\maF : \maS(\RR^{n}) \to \maS(\RR^{n})$ and
$\maF : \maS'(\RR^{n}) \to \maS'(\RR^{n})$.
We shall use any of the following equivalent notation: $\maF (u) = \maF u = \hat u$.

For $x \in \RR^{n}$, we let, as usual
\begin{equation}\label{eq.def.jap}
    \jap{x} \ede \sqrt{1 + |x|^{2}} \seq \big ( 1 + x_{1}^{2} + \ldots
    + x_{n}^{2} \big)^{1/2}\,.
\end{equation}
We define then the Sobolev space $H^{s}(\RR^{n})$ as the space of tempered
distributions $u$ on $\RR^{n}$ such that $\jap{\xi}^{s} \hat u(\xi)$ is
square integrable. We let $\CIc(\RR^{n})$ denotes the set of smooth
functions with compact support in $\RR^{n}$. Let us record, for further use, the following
simple, well-known lemma (a variant of Lemma A.5.1 in \cite{KMNW-2025}).

\begin{lemma}\label{lemma.def.hdelta}
  Let $h$ be a locally integrable function on $\Gamma := \{x_{n} = 0\}
  \subset \RR^{n}$. Then we define the distribution $h \delta_{\Gamma}
  := \bsh \delta_{\Gamma}$ on $\RR^{n}$ by the formula
  \begin{equation*} 
    \langle h \delta_{\Gamma}, \phi \rangle \ede
    \langle h \delta_{\Gamma}, \phi \rangle \ede \int_{\RR^{n-1}}
    h(x') \phi(x') \, dx'\,, \qquad \forall\, \phi \in \CIc(\RR^{n})\,.
  \end{equation*}
  If $h \in L^{2}({\Gamma})$, then $h \delta_{\Gamma} \in H^{s'}(\RR^{n})$
  for all $s' < -\frac12$. This definition extends by duality to $h \in H^{s}(\Gamma; E)$,
  in which case $h \delta_{\Gamma} \in H^{s'}(\RR^{n}; E)$,
  where $s' = s -1/2$ if $s < 0$ and, otherwise, is arbitrary such that
  $s' < -1/2$.
\end{lemma}

\begin{proof}
  First of all, given $\phi \in \CIc(\RR^{n})$, let $R$
  be such that the support of $\phi$ is contained in the ball $B_{R}(0)$.
  Then the restriction of $h$
  to $\Gamma \cap B_{R}(0)$ is integrable, because the latter set is
  relatively compact. Then we have
  \begin{equation*} 
    |\langle h \delta_{\Gamma}, \phi \rangle|
    \ede \left | \int_{\RR^{n-1} \cap B_{R}(0)}
    h(x') \phi(x') \, dx' \right |\, \le \,
    \|\phi\|_{\infty }\int_{\RR^{n-1} \cap B_{R}(0)}
    |h(x')| \, dx' \,,
  \end{equation*}
  and hence the map $\CIc(\RR^{n}) \ni \phi \mapsto
  \langle h \delta_{\Gamma}, \phi \rangle$ is continuous, and, thus, it defines
  a distribution on $\RR^{n}$. The fact that $h \delta_{\Gamma} \in H^{s}(\RR^{n})$
  for all $s < -\frac12$ is easily seen as follows. Let $u \in H^{-s}(\RR^{n})$,
  then the restriction (or trace) $u\vert_{\Gamma}$ is defined, since
  $-s > 1/2$. Consequently, $\int_{\RR^{n-1}} u\vert_{\Gamma}(x', 0)h(x')\, dx'$
  is also defined and depends continuously on $u$, hence it defines an element
  in $H^{-s}(\RR^{n})^{*} \simeq H^{s}(\RR^{n})$, as claimed.

  Finally, to prove the last part, let $\phi \in H^{-s'}(\RR^{n})$
  with compact support, where $-s' > 1/2$. Then
  $\phi\vert_{\Gamma} \in H^{-s'-1/2}(\Gamma )$ and the pairing
  $\langle h , \phi\vert_{\Gamma}\rangle$ is defined if $-s' -1/2 \ge -s$.
  So we need both conditions $s' \le s - 1/2$ and $s' < -1/2$. If $s \ge 0$,
  we retain the condition $s' < -1/2$. If $s < 0$, we retain the other condition
  and let $s' = s-1/2$.
\end{proof}

We let $\Smbmn$ be the set of functions $a : \RR^{n} \times \RR^{n} \to \CC$
such that, for every pair
of multi-indices $\alpha =(\alpha _1,\ldots,\alpha _n),\,
\beta =(\beta _1,\ldots,\beta_n)\in \ZZ_{+}^{n}$, with
$\alpha _i,\,\beta _i\geq 0$ for $i\in \{1,\ldots,n\}$
and $|\beta |:=\beta _1+\ldots +\beta _n$, there exists
$C_{\alpha\beta} \ge 0$ such that
\begin{equation}\label{eq.est.Sm}
  |\pa_{x}^{\alpha} \pa_{y}^{\beta} a(x, \xi)|
  \le C_{\alpha \beta}\jap{\xi}^{m-|\beta|}\,.
\end{equation}
Then we define $a(x, D): \CIc(\RR^{n}) \to \CI(\RR^{n})$ as usual,
by the formula
\begin{equation}\label{eq.def.pseudos}
  a(x, D) u (x) \ede \frac1{(2 \pi)^{n}}
  \,\int_{\RR^{n}} e^{\imath x \cdot  \xi} a(x, \xi) \hat u( \xi) \, d \xi\,.
\end{equation}
It is a particular case of a \emph{pseudodifferential operator} on $\RR^{n}$.

\begin{remark}\label{rem.mpr}
  The last lemma and the mapping properties of the pseudodifferential operators of the
  form $a(x, D)$ with $a \in \Smbmn$
  show that (if $a$ has order $m$) then
  \begin{equation}\label{eq.def.maS_a}
    \maS_{a}h \ede a(x, D)(h \delta_{\Gamma}) \in H^{s - m}(\RR^{n})\,,
    \qquad \mbox{for all } s < -1/2\,.
  \end{equation}
  Thus, if $m < -1$, we can choose $s$ close to $-1/2$ such that
  $s-m > 1/2$, and hence the trace $(\maS_{a}h)\vert_{\Gamma}$ is defined.
  (It is known that this is not the case for $m = -1$, though. In this paper,
  this phenomenon is discussed in Theorem \ref{thm.main.jump0}.)
\end{remark}

This allows us to introduce the following definition.

\begin{definition}\label{def.limit.values0}
  Let $e_{n} = (0, 0, \ldots, 0, 1) \in \RR^{n}$ and $\Gamma + \epsilon e_{n}
  \ede \{x_{n} = \epsilon\} \subset \RR^{n}$. By
  $\tau_{-\epsilon} : H^{s}(\Gamma + \epsilon e_{n}) \to H^{s}(\Gamma)$, we denote
  the natural isometry induced by translation.
  For $u : \RR^{n} \to \CC$ smooth enough,
  we define
  \begin{equation*} 
  \begin{gathered}
    u_{\epsilon} \ede u \vert_{\Gamma + \epsilon e_{n}}\,, \quad
    a_{\epsilon}(x', D') h \ede
    \tau_{-\epsilon} \big[ a(x, D)(h \delta_{\Gamma}) \big]_{\epsilon}
    \,, \mbox{ and}\\
    \quad u_{\pm} \ede \lim_{\epsilon \searrow 0} \tau_{\mp \epsilon}
    u_{\pm \epsilon}
  \end{gathered}
\end{equation*}
whenever these definitions make sense. The limits $u_{\pm \epsilon}$ are called the {\em normal
lateral limits} of $u$.
\end{definition}

As we will see shortly, for suitable $a$ and $\epsilon$,
$a_{\epsilon}(x', D')$ is defined and is again a pseudodifferential
operator.

We let $\RR_{\pm}^{n} := \{x = (x', x_{n})
  \in \RR^{n-1} \times \RR \mid  \pm x_{n} > 0 \}$.
Let $s > 1/2$. If $u \in H_{\loc}^{s}(\RR_{\pm}^{n})$ we shall write
\begin{equation}
  \gamma_{\pm}(u) \ede \mbox{ the trace of } u \mbox{ at } \Gamma \,, \ \mbox{ and } \gamma_{\pm}(u)\in
  H^{s-1/2}(\Gamma)\,,
\end{equation}
using, of course, that $\Gamma = \pa \RR_{+}^{n} = \pa \RR_{-}^{n}$. If, furthermore,
$u \in H_{\loc}^{s}(\RR^{n})$, we shall write $u\vert_{\Gamma} =
\gamma_{+}(u) = \gamma_{-}(u)$. We have the following simple lemma relating
the concepts introduced so far.

\begin{lemma}\label{lemma.limits=traces}
  Let $u \in H_{\loc}^{s}(\RR_{+}^{n})$, $s > 1/2$. Then
  \begin{equation*}
    u_{+} \ede \lim_{t \searrow 0} \tau_{-t} u_{t} \seq \gamma_{+}(u)
    \in H_{\loc}^{s-1/2}(\Gamma)\,.
  \end{equation*}
  Similarly, $u_{-} \seq \gamma_{-}(u)$, if
  $u \in H_{\loc}^{s}(\RR_{-}^{n})$. In particular, if
  $a \in \Smbmn$ for some $m < -1$ and $h \in L^{2}(\RR^{n})$, then
  \begin{equation*}
    \big[ a(x, D)(h \delta_{\Gamma}) \big]_{+}
    \seq \big[ a(x, D)(h \delta_{\Gamma}) \big]_{-}
    \seq \big[ a(x, D)(h \delta_{\Gamma}) \big]_{\Gamma}\,.
  \end{equation*}
\end{lemma}

Typically, in this paper, we shall work with functions (or sections of
vector bundles) to which the above lemma applies, in that case, we will not have to
distinguish between the limits $u_{\pm}$ and the traces $\gamma_{\pm}(u)$.
Of course, there exist important situations when the limits $u_{\pm}$ exist
but the traces $\gamma_{\pm}(u)$ do not exist. (The opposite arise, however, in
the case of an embedded hypersurface $\Gamma$ in a manifold, when the definition
of $u_{\epsilon}$ and $u_{\pm}$ requiers a \emph{tubular neighborhood} of $\Gamma$.)
We can now formulate the problem that we will deal with
in Sections \ref{sec.sec2} and \ref{sec.sec3}.

\begin{problem}\label{problem}
  Let $a \in \Smbmn$, we want to study the existence and the properties of the
  \emph{normal lateral limits} of Definition \eqref{def.limit.values0}:
  \begin{equation*}
    a_{\pm}(x', D') h \ede \big[ a(x, D)(h \delta_{\Gamma}) \big]_{\pm}\,,
  \end{equation*}
  and their relations to the traces
  $\gamma_{\pm} \big[a(x, D)(h \delta_{\Gamma})\big]$.
\end{problem}

Often in the literature, the \emph{non-tangential limits} at the boundary
are studied. Those are \emph{more general} than the \emph{normal lateral limits}
that we study in this paper, but for functions that are smooth enough,
they are the same. (They are the same for most of the applications in
this paper.) See \cite{KMNW-2025}. An extension of the above discussions to
manifolds is contained in Subsection \ref{ssec.ntn}.

A word now about the notation, we often parametrize $\Gamma$ with $\RR^{n-1}$.
Thus $\Gamma \subset \RR^{n}$, but $\Gamma \simeq
\RR^{n-1} \not \subset \RR^{n}$. This is the reason we need the isometries
$\tau_{t}$. We distinguish $\Gamma$ from
$\RR^{n-1}$ to make it easier to transition to an arbitrary
open domain with boundary $\Gamma$. However, in this paper, there will
be no situation when confusions can arise if we omit the identification
$\tau_{t}$ from the notation, so we shall do that from now on.

\subsection{Lateral limits of pseudodifferential operators of arbitrary orders}

We now begin our study of the \emph{normal lateral limits} $[a(x, D)u]_{\pm}$
of the values of a pseudodifferential operator $a(x, D)$, which
were defined in Definition \ref{def.limit.values0}
(see Subsection \ref{ssec.problem}, especially Problem \ref{problem}).
Recall that if $a(x, D)u$ is smooth enough on either of the half-spaces
$\RR_{\pm}^{n}$, then the corresponding normal lateral limit
$[a(x, D)u]_{\pm}$ coincides with the trace of $a(x, D)u$ on the boundary
of that half-space.

A complete and general treatment of normal lateral limits
is contained in the book by Hsiao and Wendland \cite{H-W}.
Here we only deal with the results (and calculations) needed to treat
our generalized Stokes operator. See also \cite{Hormander1, Taylor1} for
general results on distributions, Fourier transforms, and pseudodifferential
operators. Further background on pseudodifferential operators (including the
results not proved here) can be found in
one of the following books \cite{HAbelsBook, Hormander1, H-W, RT-book2010, Taylor,
Taylor2, Wong}. A quick introduction to some basic facts and definitions
geared towards our applications can be found in \cite{KNW-22}.
Our presentation is a complete and concise introduction to the subject of limit
and jump relations for potential operators on a half-space, the missing proofs can
be found in \cite{KMNW-2025}.

As recalled in the previous subsection, $\Smbmn$ denotes the set of \emph{order $m$ symbols}
on $\RR^{n}$. Similarly, $S^{m}_{\cl}(\RR^{n} \times \RR^{n})$ denotes the set of order $m$,
\emph{classical symbols} on $\RR^{n}$ (they consist of symbols that have expansions
in terms of homogeneous functions). The resulting pseudodifferential operator $a(x, D)$
is given by the usual formula \eqref{eq.def.pseudos}. If $b \in S^{m}(\RR^{n-1} \times \RR^{n-1})$,
we shall denote by $b(x', D')$ its associated operator. In general, symbols with a prime
(i.e. ${}'$) will refer to objects on $\RR^{n-1}$. The most often used example is that, if
$x \in \RR^{n}$, then $x = (x', x_{n})$ with $x' \in \RR^{n-1}$ and $x_{n} \in \RR$ its projections.

Let $\maF_{\xi_{n}}^{-1}$ denote the one-dimensional inverse Fourier transform in the variable
$\xi_{n} \in \RR$.

\begin{lemma} \label{lemma.Fsymbols}
  Let $a \in \Smbmn$, $m \in \RR$.
  \begin{enumerate}[\rm (1)]
    \item For any fixed $x \in \RR^{n}$, $  \xi' \in \RR^{n-1}$, the function
    $\phi_{x, \xi'}(  \xi_{n}) := a(x, \xi', \xi_{n})$ defines a tempered
    distribution on $\RR$ such that its inverse Fourier transform in $\xi_{n}$,
    $\maF_{\xi_{n}}^{-1} \phi_{x, \xi'}$, coincides
    with a smooth function outside 0.

    \item For $m < -1$, we have
    $\maF_{\xi_{n}}^{-1} \phi_{x, \xi'} \in \maC_{0}(\RR)$.

    \item Let $m =-1$ and $e_{n} = (0, \ldots, 0, 1) \in \RR^{n}$ and assume
    that $a$ is \emph{classical}. Then, for all $x \in \RR^{n}$ and $ \xi' \in \RR^{n-1}$,
    the following limits exist
    \begin{equation*}
      \JC_{\pm}(x) \seq \JC_{\pm}(a; x) \ede \lim_{\tau \to \pm \infty}
      \tau a(x,   \xi', \tau)  \seq  \pm \sigma_{-1}(a; x, \pm e_{n}) \in \CC\,.
    \end{equation*}

    \item Let us also assume that, for all $x \in \RR^{n}$,
    we have $\JC_{+}(x) = \JC_{-}(x)$. Then
    $\maF_{\xi_{n}}^{-1} \phi_{x,   \xi'} \in \maS'(\RR)$ is a function that is
    continuous everywhere, except maybe at $0 \in \RR$,
    with one-sided limits in $0$ given by
    \begin{equation*} 
      \maF_{\xi_{n}}^{-1} \phi_{x,   \xi'} (0\pm) \seq
      \pm \frac{\imath \JC_{+}(x)}2
      + \frac1{2 \pi}\, \int_{\RR} \frac{a(x,   \xi', \tau) +
      a(x, \xi', -\tau)}{2} \, d\tau\,.
    \end{equation*}
    Clearly, the condition $\JC_{+} = \JC_{-}$ is satisfied if
    $\sigma_{-1}(a)$ is odd.
  \end{enumerate}
\end{lemma}

The critical case $m =-1$ requires some further discussion.

\begin{remark}
  We first notice that in the last point the function  $a(x, \xi', \tau) + a(x, \xi', -\tau)$
  is integrable in $\tau$ for all fixed $(x, \xi') \in \RR^{n} \times \RR^{n-1}$.
  Then, we notice that the assumption $\JC_{+} = \JC_{-}$ is can be written
  explicitly as
  \begin{equation*}
    \sigma_{-1}(a; x, -e_{n}) \seq -\JC_{-}(x) \seq -\JC_{+}(x) \seq -\sigma_{-1}(a; x, e_{n})\,,
  \end{equation*}
  for all $x \in \RR^{n}$.
\end{remark}

Later on, we will want to show the dependence of $\JC_{\pm}$ on the symbol $a$ (or
the operator $a(x, D)$), so we will write
\begin{equation}\label{eq.full.notation}
  \begin{gathered}
    \JC_{+}(a; x) \seq \JC_{+}(a(x, D); x) \ede \JC_{+}(x) \qquad \mbox{and}\\
    \JC_{-}(a; x) \seq \JC_{-}(a(x, D); x) \ede \JC_{-}(x)
  \end{gathered}
\end{equation}

We now return to the general case $m \in \RR$.
Recall that if $\xi \in \RR^{n}$, then we write $\xi = (\xi', \xi_{n})$, where
$\xi' \in \RR^{n-1}$ and $\xi_{n} \in \RR$. The result of Lemma \ref{lemma.Fsymbols}
justifies the following definition that will play a central role in this section.

\begin{definition}\label{def.aepsilon}
  Let $t \in \RR$ and $a \in \Smbmn$ with $m \in \RR$.
  \begin{enumerate}[(a)]
    \item For $m \ge -1$, we also assume $t \neq 0$. Then we define:
    \begin{equation*} 
      a_{s, t}(x',   \xi') \ede
      \maF_{\xi_{n}}^{-1} a(x', s, \xi', t) \in \CC\,.
    \end{equation*}

    \item If $m=-1$ and $\sigma_{-1}(a; x, e_{n}) = -\sigma_{-1}(a; x, -e_{n})$, we define also
    \begin{equation*} 
      \begin{gathered}
      a_{s, 0}(x',   \xi') \ede \frac{1}{2\pi}
      \int_{\RR}  \frac{a(x', s, \xi', \xi_{n})
      + a(x', s, \xi', -\xi_{n})}{2} \, d\xi_{n}  \in \CC\,, \qquad \mbox{ and}\\
      a_{s, 0\pm}(x', \xi') \ede \pm \frac{\imath \sigma_{-1}(a; x', s, e_{n})}2
      + a_{s, 0}(x', \xi')  \in \CC \,.
      \end{gathered}
    \end{equation*}
  \end{enumerate}
\end{definition}

Definition \ref{def.aepsilon} is motivated by the lateral limit Problem \ref{problem}
formulated in Subsection \ref{ssec.problem}. The case $m < -1$ is simpler.

\begin{remark}
  Assume $m < -1$. Then the definition of $a_{s, t}$ above
  above extends to all $t \in \RR$ (not just for $t \neq 0$), and
  we have the following explicit formula
  \begin{equation}\label{eq.def.aepsilon.explicit}
    a_{s, t}(x', \xi') \ede
    \maF_{  \xi_{n}}^{-1} a(x', s,   \xi', t)
    \ede \frac{1}{2\pi}
    \int_{\RR} e^{\imath t \xi_{n}} a(x', s ,   \xi',   \xi_{n}) \, d  \xi_{n}\,,
  \end{equation}
  because we take the Fourier transform of an integrable function, instead of
  a temperate distribution.
\end{remark}

The following result on the lateral limit Problem \ref{problem} justifies
the last definition. To state it, recall that
$\Gamma := \{(x', s) \in \RR^{n} \mid s = 0 \}$ and
that $h \delta_{\Gamma}$ is the distribution
$\langle h \delta_{\Gamma}, \phi \rangle := \int_{\Gamma}
h(x) \phi(x) \, dx$, see Equation \eqref{eq.def.Gamma} and
Lemma \eqref{lemma.def.hdelta}.
If $u : \RR^{n} \to \CC$ is continuous enough, recall the notation
introduced in Definition \eqref{def.limit.values0}, for instance
$u_{\epsilon} : \RR^{n-1} \to \CC$ is given by
  $u_{\epsilon} (x') \ede u(x', \epsilon)$, $x' \in \RR^{n-1}$
Note that, we identify $\Gamma + \epsilon e_{n}$ with $\Gamma$
and their associated function spaces with the translation $\tau_{-\epsilon}$,
as explained in Subsection \ref{ssec.problem}. We continue to denote
  $\maS_{a}(h) := a(x, D)(h \delta_{\Gamma}) \in
  \CI(\RR^{n} \smallsetminus \Gamma)$, as in Equation \eqref{eq.def.maS_a}.

\begin{proposition}\label{prop.side.limits}
  Let $h \in L^{2}(\RR^{n-1})$, $a \in \Smbmn$, $m \in \RR$,
  and $a_{s, t}$ be as in Definition \ref{def.aepsilon}.
  Then, for any $\epsilon \neq 0$, $a(x, D) (h \delta_{\Gamma})$
  is smooth on $\Gamma + \epsilon e_{n} := \{x_{n} = \epsilon\} \simeq \RR^{n-1}$ and
  its restriction to this set satisfies
  \begin{equation*}
    \big[\maS_{a}(h)\big]_{\epsilon} \ede
    \big[a(x, D)(h \delta_{\Gamma})\big]_{\epsilon}
    \seq a_{\epsilon, \epsilon}(x', D') h \,.
  \end{equation*}
  In particular, i terms of the notation introduced in Definition \eqref{def.limit.values0} and
  in Proposition \ref{prop.side.limits}, we have $a_{\epsilon} = a_{\epsilon, \epsilon}$.
\end{proposition}

This justifies the study of the symbols $a_{s, t}$, which we do in the next subsections,
according to the values of $m$.

\subsection{Lateral limits of pseudodifferential operators of orders $<-1$}
Let $\Gamma := \{x_{n} = 0\} \subset \RR^{n}$, as always in this section.
We state next the needed results for order $m < -1$ operators.
Recall the symbols $a_{s, t}$, $s, t \in \RR$, of Definition
\ref{def.aepsilon} and their associated operators. For simplicity, we formulate and
proved our results for \emph{scalar} symbols.
The statements and proofs extend, however, immediately to the vector valued case.

\begin{proposition} \label{prop.symbols.zero}
  Let $a \in \Smbmn$, $m < -1$.
  \begin{enumerate}[\rm (1)]
    \item For any $(s, t) \in \RR^{2}$, the map
    $(x',   \xi') \to a_{s, t}(x',   \xi')$ defines a symbol in
    $S^{m+1}(\RR^{n-1} \times \RR^{n-1})$.

    \item If $h \in H^{s}(\Gamma)$, then the function
    \begin{equation*}
      \RR^{2} \ni (s, t) \to a_{s,t}(x', D')h
      \in H^{s-m-1}(\Gamma) \simeq H^{s-m-1}(\RR^{n-1})
    \end{equation*}
    is continuous.

    \item If $a \in S_{\cl}^{m}(\RR^{n} \times \RR^{n})$, then,
    for all $s \in \RR$,
    $a_{s, 0} \in S_{\cl}^{m+1}(\RR^{n-1} \times \RR^{n-1})$.
  \end{enumerate}
\end{proposition}

Recall that, if $u : \RR^{n} \to \CC$ is continuous enough, then $u_{\epsilon}$ is the restriction
of $u$ to $\Gamma + \epsilon e_{n} := \{x_{n} = \epsilon\}$ (see Definition \eqref{def.limit.values0})
and $u_{\pm} \ede \lim_{\epsilon \to 0\pm} u_{\epsilon}\,.$
Also, recall that $a_{\epsilon} = a_{\epsilon, \epsilon}$
(see Definition \eqref{def.limit.values0} and Proposition \ref{prop.side.limits}).
Recall also the distribution $h \delta_{\Gamma}$ introduced in
Lemma \eqref{lemma.def.hdelta}.
We are now ready to formulate our main result concerning the limit/jump
values of classical \emph{matrix valued} pseudodifferential operators of order $< -1$.

\begin{theorem}\label{thm.main.jump-2}
  Let $a \in S^{m}(\RR^{n} \times \RR^{n})$ for some $m < -1$. We use
  $a_{\epsilon} = a_{\epsilon, \epsilon}$ and the notation of Definition \ref{def.aepsilon}.
  \begin{enumerate}[\rm (1)]
    \item $a_{0}$ is an order $m+1$ symbol given by
    \begin{equation*}
      a_{0}(x', \xi') \seq \frac1{2\pi}
      \int_{\RR} a( x', 0, \xi', \xi_{n}) \, d\xi_{n}\,.
    \end{equation*}
    If $a$ is classical, then $a_{0}$ is also classical.

    \item For all $s \in \RR$ and all $h \in H^{s}(\Gamma)
    = H^{s}(\Gamma)$, we have
    \begin{multline*}
      [\maS_{a}h]_{\pm} \ede \big[ a(x, D)(h \delta_{\Gamma})]_{\pm}
      \ede \lim_{\epsilon \to 0\pm} \big[ a(x, D)(h \delta_{\Gamma})]_{\epsilon}\\
      \seq \lim_{\epsilon \to 0\pm} a_{\epsilon}(x', D')h \seq
      a_{0} (x', D') h \in H^{s - m - 1 }(\Gamma)\,.
    \end{multline*}

    \item If $h \in L^{2}(\Gamma)$, then
    $\maS_{a}h \ede a(x, D)(h \delta_{\Gamma})
    \in H^{s'}(\RR^{n})$ for $s' \in (1/2, -m - 1/2)$, and hence we have the equality
    of traces (i.e. restrictions)
    \begin{equation*}
      [\maS_{a}h]_{+} \seq [\maS_{a}h]_{-} \seq [\maS_{a}h]\vert_{\Gamma}
      \seq a_{0}(x', D')h \in H^{s'-1/2}(\Gamma)\,.
    \end{equation*}

    \item Let $k_{a_{0}(x', D')}$ be the distribution kernel of $a_{0}(x', D')$
    and $k_{a(x, D)}$ be the distribution kernel of $a(x, D)$. Then
    \begin{equation*}
      k_{a_{0}(x', D')}(x', y') \seq k_{a(x, D)}(x', 0, y', 0)\,,
      \quad x', y' \in \RR^{n-1}\,, \ x' \neq y'\,.
    \end{equation*}
  \end{enumerate}
\end{theorem}

The operator $a_{0}(x, D)$ will be called the {\it restriction at $\Gamma$ operator}
associated to $a(x, D)$.

\subsection{Lateral limits of pseudodifferential operators of orders $-1$}

We now consider symbols of order $-1$. For simplicity, we consider only
\emph{classical symbols.}

\begin{proposition} \label{prop.symbols.m=-1}
  Let $a \in S^{-1}_{\cl}(\RR^{n} \times \RR^{n})$ and
  assume that, for all $x \in \RR^{n}$, we have
  \begin{equation*}
    \JC_{+}(x) \ede \sigma_{-1}(a; x, e_{n}) \seq
    -\sigma_{-1}(a; x, -e_{n}) \, =:\, \JC_{-}(x) \,.
  \end{equation*}
  \begin{enumerate}[\rm (1)]
    \item For all multi-indices $\alpha \in \ZZ_{+}^{n}$ and
    $\beta \in \ZZ_{+}^{n-1}$,
    there exist constants $C_{\alpha, \beta} > 0$ such that, for all $(x,   \xi', t)
    \in \RR^{n} \times \RR^{n-1} \times \RR$, we have
    \begin{equation*}
      |\pa_{x}^{\alpha}\pa_{  \xi'}^{\beta} a_{x_n,t}(x',   \xi')|
      \le C_{\alpha, \beta}\jap{  \xi'}^{-|\beta|}\,.
    \end{equation*}

    \item The set $\{ a_{x_{n}, t}, a_{x_{n}, 0\pm} \mid x_{n}, t \in \RR \}$ is
    a bounded subset of $S^{0}(\RR^{n-1} \times \RR^{n-1})$.

    \item The function $(x,   \xi', t) \mapsto a_{x_{n}, t}(x',   \xi')$
    of Definition \ref{def.aepsilon} is continuous except at $t = 0$,
    where it has lateral limits $a_{s, 0\pm}(x',   \xi')$.
  \end{enumerate}
\end{proposition}

We next formulate our main result on side (or boundary) limits of
pseudodifferential operators of order $= -1$, Theorem \ref{thm.main.jump0}
next. Recall that $u_{\epsilon}$
and $u_{\pm} := \lim_{\epsilon \to 0\pm} u_{\epsilon}$ were introduced in Definition
\eqref{def.limit.values0}. Also, recall that the distribution
$\langle h \delta_{\Gamma}, \phi \rangle
:= \int_{\Gamma} h \phi dx'$ was introduced in Lemma \eqref{lemma.def.hdelta}.
We shall also write $a_{t} = a_{t, t}$, $t \in \RR$, and $a_{0\pm} := a_{0, 0\pm}$,
see Definition \ref{def.aepsilon}.

\begin{theorem}\label{thm.main.jump0}
  We use the notation in Definition \ref{def.aepsilon}.
  Let $a \in S^{-1}_{\cl}(\RR^{n} \times \RR^{n})$.
  \begin{enumerate}[\rm (1)]
    \item Let $k_{a_{0}(x', D')}$ be the distribution kernel of $a_{0}(x', D')$
    and $k_{a(x, D)}$ be the distribution kernel of $a(x, D)$. Then both
    $k_{a_{0}(x', D')}(x', y')$ and $k_{a(x, D)}(x, y)$ are smooth for
    $x' \neq y'$ and they coincide on $\RR^{n-1} \times \RR^{n-1}$:
    \begin{equation*}
      k_{a_{0}(x', D')}(x', y') \seq k_{a(x, D)}(x', 0, y', 0)\,,
      \quad x', y' \in \RR^{n-1}\,.
    \end{equation*}

    \item If $\sigma_{-1}(a)$ is odd in the sense that
    $\sigma_{-1}(a; x, -\xi) = - \sigma_{-1}(a; x, \xi)$ for
    all $\xi \in \RR^{n}$, then the condition
    $\JC_{+}(x) \ede \sigma_{-1}(a; x, e_{n}) \seq
    -\sigma_{-1}(a; x, -e_{n}) \, =:\, \JC_{-}(x)$
    is satisfied, $\sigma_{0}(a_{0})$ is also odd, and
    \begin{align*}
      a_{0}(x', D')h_{1}(x) &\seq \pv \,
      \int_{\RR^{n-1}} k_{a(x, D)}(x', y') h(y')\,dy'\\
      & \ede \lim_{\epsilon \to 0}\int_{\RR^{n-1}\setminus B(x',\epsilon)}
      k_{a(x, D)}(x', y') h(y')\,dy' \,,
    \end{align*}
    where $B(x',\epsilon)$ is the open ball of radius $\epsilon $ and center at $x' \in \RR^{n-1}$.

    Let us assume for the next two points that $\sigma_{-1}(a; x, e_{n}) \seq
    -\sigma_{-1}(a; x, -e_{n})$.

    \item The three operators $a_{0} = a_{0, 0}$ and $a_{0\pm} := a_{0, 0\pm}$ are
    order zero classical, with principal symbols
    \begin{equation*}
      \sigma_{0}(a_{0}; x',  \xi') \seq \frac1{2\pi}
      \int_{\RR} \frac{\sigma_{-1}(a; x', 0,  \xi',  \xi_{n})
      + \sigma_{-1}(a; x', 0,  \xi', - \xi_{n})}2 \, d \xi_{n}
    \end{equation*}
    and $\sigma_{0}(a_{0\pm}; x',  \xi') = \pm \frac{\imath}2 \sigma_{-1}(a; x', 0, e_{n})
    + \sigma_{0}(a_{0}; x',  \xi')$.

    \item For all $s \in \RR$ and for all
    $h \in H^{s}(\RR^{n-1})$, we have
    \begin{equation*}
      [\maS_{a}h]_{\pm} \ede
      \lim_{\epsilon \to \pm 0} \big[ a(x, D)(h \delta_{\Gamma})]_{\epsilon}
      \seq \lim_{\epsilon \to \pm 0} a_{\epsilon}(x', D')h \seq
      a_{0\pm} (x', D') h \in H^{s}(\RR^{n-1})\,.
    \end{equation*}
  \end{enumerate}
\end{theorem}

Most of the relations of Theorem \ref{thm.main.jump0} (dealing with the critical case
$m = -1$) have been written in a compact form. The expanded form of these relations
amounts to the following five relations:

\begin{remark}
  Recall that $a_{\epsilon} = a_{\epsilon, \epsilon}$ and $a_{0, 0\pm} = a_{0\pm}$.
  Then
  \begin{equation*}
    \begin{gathered}
    \big[ a(x, D)(h \delta_{\Gamma})]_{+} \ede
    \lim_{\epsilon \searrow 0} \big[ a(x, D)(h \delta_{\Gamma})]_{\epsilon}
    \seq \lim_{\epsilon \searrow 0} a_{\epsilon}(x', D')h \seq
    a_{0+} (x', D') h \\
    \big[ a(x, D)(h \delta_{\Gamma})]_{-} \ede
    \lim_{\epsilon \nearrow 0} \big[ a(x, D)(h \delta_{\Gamma})]_{\epsilon}
    \seq \lim_{\epsilon \nearrow 0} a_{\epsilon}(x', D')h \seq
    a_{0-} (x', D') h \\
    \sigma_{0}(a_{s, 0}; x',  \xi') \seq \frac1{2\pi}
    \int_{\RR} \frac{\sigma_{-1}(a; x', s,  \xi',  \xi_{n})
    + \sigma_{-1}(a; x', s,  \xi', - \xi_{n})}2 \, d \xi_{n}\\
    \sigma_{0}(a_{s, 0+}; x',  \xi') \seq \frac{\imath}2
    \sigma_{-1}(a; x', 0,  e_{n})
    + \sigma_{0}(a_{s, 0}; x',  \xi') \qquad \mbox{and} \\
    \sigma_{0}(a_{s, 0-}; x',  \xi') \seq - \frac{\imath}2
    \sigma_{-1}(a; x', 0, e_{n}) + \sigma_{0}(a_{s, 0}; x',  \xi')
    \,.
  \end{gathered}
\end{equation*}
\end{remark}

These calculations easily allow us to recover the usual jump relations
for the Laplacian, as in the conclusing Example A.5.16 of \cite{KMNW-2025}.
In case $a$ has order $\le -2$, then the last theorem still applies
and yields the same operator $a_{0}(x, D)$ as Theorem \ref{thm.main.jump-2}.
Thus, as it was the case for the operator obtained in Theorem \ref{thm.main.jump-2},
the operator $a_{0}(x, D)$ will be called the {\it restriction at $\Gamma$ operator}
associated to $a(x, D)$.

\section{Normal lateral limits and abstract jump relations on manifolds}
\label{sec.sec3}

In this section, we adapt the results from the previous section
on normal lateral limits on half-spaces to smooth, open domains in
general (smooth) Riemannian manifolds (possibly non-complete).
Note that this is necessary even if we are working
on $\RR^{n}$, but on other open subsets than the half-spaces.

\subsection{Normal tubular neighborhoods}
\label{ssec.ntn}

Let $M$ be a smooth Riemannian manifold and let $\Omega \subset M$ be an open subset with
smooth boundary $\Gamma := \pa \Omega = \pa \Omega_{-} \neq \emptyset$ (so the assumption
is that $\Gamma$ is also a smooth manifold). For simplicity, we assume that $\Omega$ is on
one side of its boundary. Equivalent ways of expressing this are saying that
\begin{enumerate}
  \item $\Omega$ is the interior of $\overline{\Omega}$ or that
  \item $\Gamma$ is the boundary of $\Omega_{-} := M \smallsetminus \overline{\Omega}$.
\end{enumerate}
(We are thus excluding the case of domains with cracks, whose study via the method of layer
potentials is, anyway, not very convenient.) This assumption will remain in place
{\it throughout this paper} (and will be reminded occasionally). We let $dS_{\Gamma}$ denote
the conditional ($n-1$-dimensional) measure on $\Gamma$. We will consider also a smooth,
hermitian vector bundle $E \to M$. Let $s \in \RR$. We let $H_{\loc}^{s}(M; E)$ denote the
space of distributions with values in $E$ whose restriction to any compact coordinate
chart is in $H^{s}$ of that chart. Similarly,
we let $H_{\comp}^{s}(M)$ denote the space of distributions in $H_{\loc}^{s}(M)$ that
have \emph{compact support}. We define $L_{\comp}^{p}(M; E)$ similarly.
The \emph{global} spaces $H^{s}(M)$ are defined using
the metric. Their variants with values in smooth vector bundles are defined similarly.

Distributions of the form
\begin{equation*}   
  \langle h \delta_{\Gamma}, \phi \rangle \ede \int_{\Gamma}
  h(x) \cdot \phi(x) \, dS_{\Gamma}(x)\,,
\end{equation*}
will continue to play an important role in what follows. Conditions for this formula to be
defined and to define a distribution are contained in the following analogous
version of Lemma \ref{lemma.def.hdelta}.
Recall that a continuous map $\phi : X \to Y$ between locally compact topological
spaces is \emph{proper} if, for every compact $K \subset Y$, the inverse image
$\phi^{-1}(K) \subset X$ is compact. We can now formulate the following analog of Lemma
\ref{lemma.def.hdelta}. We define our (local or compactly supported)
Sobolev spaces $H_{\loc}^{s}(\Gamma; E)$ and $H_{\comp}^{s}(\Gamma; E)$ on $M$
using partitions of unity, as in, for instance, \cite{Grosse-Kohr-Nistor-23,
GrosseSchneider, KMNW-2025, KohrNistor1, SkAtomic, TriebelBG}. We note, however, that all our
function spaces will be considered to be \emph{complex.} Thus, if $E$ is a real
vector bundle, by $H^{s}(M; E)$ we shall actually mean its complexification
$H^{s}(M; E \otimes \CC) = H^{s}(M; E) \otimes \CC$.

\begin{lemma}\label{lemma.def.hdelta2}
  Let $E$ be a smooth Hermitian vector bundle on $M$ with inner product
  denoted $\cdot$ and $s \in \RR$. Let $s' \in \RR$ be such that
  \begin{equation*}
    \begin{cases}
      \ s' \ede s -1/2 & \mbox{ if } s < 0\\
      \ s' < -1/2 & \mbox{ is arbitrary, if }\ s \ge 0\,.
    \end{cases}
  \end{equation*}
  \begin{enumerate}[\rm (1)]
    \item If $h \in H_{\comp}^{s}(\Gamma; E)$, then
    $h \delta_{\Gamma} \in H_{\comp}^{s'}(M; E)$.

    \item If $h \in H_{\loc}^{s}(\Gamma; E)$ and $\Gamma \to M$ is proper, then
    $h \delta_{\Gamma} \in H_{\loc}^{s'}(M; E)$.

    \item If the maps $H^{r}(M; E) \to H^{r-1/2}(\Gamma; E)$ are continuous for
    all $r > 1/2$, then, for all $h \in H^{s}(\Gamma; E)$, we obtain
    $h \delta_{\Gamma} \in \big(H^{-s'}(M; E)\big)^{*}$.
  \end{enumerate}
\end{lemma}

\begin{proof}
  The proof is similar to that of Lemma \ref{lemma.def.hdelta}. We assume $E = \CC$,
  to simplify the notation. We first notice that $H_{\comp}^{r}(M)^{*} \simeq
  H_{\loc}^{-r}(M)^{*}$ and $H_{\loc}^{r}(M)^{*} \simeq H_{\comp}^{-r}(M)^{*}$. Let
   $\phi \in \CIc(M)$.

  For the first point, we use the continuity of the map $H_{\loc}^{r}(M)
  \to H_{\loc}^{r-1/2}(\Gamma)$ for $r = -s'$ to conclude that the restriction
  $\phi\vert_{\Gamma} \in H_{\loc}^{-s'-1/2}(M)$ and that this restriction depends
  continuously on $\phi \in H_{\loc}^{-s'}(M)$. The composite map
  \begin{equation*}
    H_{\loc}^{-s'}(M) \ni \phi \,\mapsto\, \phi\vert_{\Gamma}
    \,\mapsto\, \langle \phi\vert_{\Gamma}, h \rangle
    \, =: \, \langle \phi , h \delta_{\Gamma} \rangle \in \CC
  \end{equation*}
  is hence continuous, and, consequently, $h \delta_{\Gamma}$ defines an
  element in $H_{\loc}^{-s'}(M)^{*} \simeq H_{\comp}^{s'}(M)$.

  The second part is similar. We first notice that the fact that the inclusion $\Gamma
  \to M$ is proper guarantees that we have a continuous map $H_{\comp}^{r}(M)
  \to H_{\comp}^{r-1/2}(\Gamma)$ for all $r > 1/2$. The rest is as in the first part
  {\it mutatis mutandis.}
  The last part is proved in the same way.
\end{proof}

We let $\Psi^{m}(M; E, F)$ denote the set of order $m$ pseudodifferential operators on
$M$ acting from sections of a smooth vector bundle $E \to M$ to sections of a vector bundle
$F \to M$. These operators are defined by the requirement that, in any coordinate
neighborhood of $U \subset M$ and for any $\phi \in \CIc(U)$, the operator $\phi P \phi$
be given by Equation \eqref{eq.def.pseudos}. If the resulting pseudodifferential operators
in local coordinates are classical, we shall say that $P$ is {\it classical}.
The set of {\it classical pseudodifferential} operators is denoted $\Psi_{\cl}^{m}(M; E, F)$.

Let $P$ be an order $m$  pseudodifferential operator acting on the sections of $E$ with values
sections of $F$, that is $P \in \Psi^{m}(M; E, F)$. We are interested in studying
\begin{equation}\label{eq.def.layerPot-0}
  \maS_{P} h \ede P(h \delta_{\Gamma})\,,
\end{equation}
provided that the latter is defined. For the convenience of the notation, we also let
\begin{equation*}
  \Omega_{+} \ede \Omega \ \mbox{ and }\
  \Omega_{-} \ede M \smallsetminus \overline{\Omega}
  \quad \Rightarrow \quad \Gamma \seq \pa \Omega_{-} \seq \pa \Omega_{+}\,.
\end{equation*}
We are especially interested in the following two restrictions and their traces
\begin{equation}\label{eq.def.layerPot}
  \maS_{P} h \vert_{\Omega_{\pm}}\quad \mbox{ and }\quad
  [\maS_{P} h]_{\pm} \ede [\maS_{P} h\vert_{\Omega_{\pm}}]\vert_{\pa \Omega_{\pm}}\,.
\end{equation}

\begin{notation}\label{not.bsnu}
We let $\bsnu$ be the \emph{outer} unit normal vector to $\Gamma := \pa \Omega$.
We extend this vector field to a global (smooth) vector field on $M$ (not necessarily
unit everywhere), still denoted $\bsnu$. Also, let $\sharp : TM \to T^{*}M$
be the isomorphism defined by the metric of $M$. We shall write
$v^{\sharp} :=\sharp v$ (in particular, $\bsnu^{\sharp}:=\sharp \bsnu$).
\end{notation}

If $v \in TM$ is a tangent vector to $M$ in $x$, we let $\exp(tv)$
denote the image of $tv$ under the exponential map, which is defined for $|t|$
small (depending on $v$).

\begin{definition}\label{def.e.nbhd}
  If $\epsilon > 0$ is such that the \emph{normal exponential map}
  \begin{equation*}
    \exp^{\perp} : \Gamma \times (-\epsilon, \epsilon) \ni (x, t)\, \mapsto\,
    \exp(t \bsnu(x)) \in M
  \end{equation*}
  ($\Gamma :=\partial \Omega $) is well defined and is a diffeomorphism onto its image,
  then we shall say that $\Gamma$ has \emph{an $\epsilon$-normal
  tubular neighborhood}.
\end{definition}

If $\Gamma$ has an $\epsilon$-normal
tubular neighborhood, then the inclusion $\Gamma \to M$ is proper.
Also, it is well-known that if $\Gamma$ is compact or that $\Gamma$ and
$M$ have cylindrical ends, then $\Gamma$ will have an
$\epsilon$-normal tubular neighborhood, for some $\epsilon > 0$ small enough,
see \cite[Corollary 5.5.3]{petersen:98}.

The curves $t \mapsto \exp(t\bsnu(x))$,
$x \in \pa \Omega$, will be called the \emph{normal geodesics} to $\pa \Omega$.
If $u$ is a section of $E$ over $M$, $\Gamma$ has an $\epsilon$-normal
tubular neighborhood, and $t \in (-\epsilon, \epsilon)$, we let
\begin{equation}\label{eq.p.transport}
  u_{t} \ede u\vert_{\exp^{\perp}(\Gamma \times \{t\})} \in
  \CI(\Gamma \times \{t\}; E) \simeq \CI(\Gamma; E)\,,
\end{equation}
where the last isomorphism is obtained via parallel transport
along the normal geodesics $(-\epsilon, \epsilon) \ni t \to \exp(t\bsnu(x)) \in M$,
$x \in \pa \Omega$.
We will use that in this case, the inclusion $\Gamma \to M$ is proper.

It will be important for us to study the limits
$u_{\pm} := \lim_{t \to \pm 0} u_{t}$ in some function space on $\Gamma := \pa \Omega$,
for suitable $u$. When they exist, we call these limits, the \emph{normal lateral limits}
of $u$. In case $u$ is smooth enough on $\Omega_{+} := \Omega$ and on
$\Omega_{-} := M \smallsetminus \overline{\Omega}$, then $u_{+}$ is the trace of
$u\vert_{\Omega_{+}} := u\vert_{\Omega}$ at the boundary and, similarly,
$u_{-}$ is the trace of $u\vert_{\Omega_{-}} :=
u\vert_{M \smallsetminus \overline{\Omega}}$ at the boundary, see Lemma
\ref{lemma.limits=traces}.

\subsection{Lateral limits on manifolds for operators of order $m < -1$}
We now turn to the study of normal lateral limits of pseudodifferential
operators at $\Gamma$ on general (possibly non-compact) Riemannian manifolds.
As usual, the case of pseudodifferential operators of order $m < -1$ is easier.

We begin with the case of operators with compactly supported distribution
kernels in $M \times M$, which will then be used to deal with the general case.
Let $F \to M$ be a \emph{second} hermitian vector bundle (in addition to $E$).
We have the following simple calculation that will be used repeatedly, so we
formulate it as a lemma.

\begin{lemma}\label{lemma.enough.reg}
  Let $P \in \Psi^{m}(M; E, F)$, $m < -1$, $\maS_{P}h := P(h \delta_{\Gamma})$,
  and $s' \in (1/2, -m-1/2) \not = \emptyset$.
  \begin{enumerate}[\rm (1)]
    \item If $h \in L_{\comp}^{2}(\Gamma; E)$, then
    $\maS_{P}h \in H_{\loc}^{s'}(M; F)$, and hence
    \begin{equation*}
      [\maS_{P}h]_{+} \seq [\maS_{P}h]_{-} \seq [\maS_{P}h]\vert_{\Gamma}
      \in H_{\loc}^{s'-1/2}(\Gamma; F)\,.
    \end{equation*}

    \item These relations remain true if $\Gamma \to M$ is proper,
    $h \in L_{\loc}^{2}(\Gamma; E)$, and $P$ is properly supported.

    \item Let us assume the following:
    \begin{enumerate}[\rm (i)]
      \item $M$ and $E$ have bounded geometry,
      \item the maps $H^{r}(M; E) \to H^{r-1/2}(\Gamma; E)$ are continuous for
      all $r > 1/2$, and
      \item $P$ maps $H^{r}(M; E) \to H^{r-m}(M; F)$
      continuously for all $r \in \RR$.
    \end{enumerate}
    Then, for all $h \in L^{2}(\Gamma; E)$, we obtain $\maS_{P}h \in H^{s'}(M; F)$ and
    \begin{equation*}
      [\maS_{P}h]_{+} \seq [\maS_{P}h]_{-} \seq [\maS_{P}h]\vert_{\Gamma}
      \in H^{s'-1/2}(\Gamma; F)\,.
    \end{equation*}
  \end{enumerate}
  In particular, in all three cases above,
  the trace (or restriction) $\maS_{P}h\vert_{\Gamma}$
  of $\maS_{P}h := P(h \delta_{\Gamma})$ at $\Gamma$ is defined
  and it coincides with the lateral traces $[\maS_{P}h]_{\pm}$ associated to the
  domains $\Omega_{+}$ and $\Omega_{-}$ with common boundary $\Gamma$.
\end{lemma}

\begin{proof}
  Let us notice first that $-m -1/2 > 1/2$, so the set $(1/2, -m-1/2)$
  is non-empty. Let us prove (i). Because $s'+ m < -1/2$, Lemma \ref{lemma.def.hdelta2}(i)
  shows that $h \delta_{\Gamma} \in H_{\comp}^{s'+m}(M; E)$, and
  therefore $\maS_{P}h := P(h \delta_{\Gamma}) \in H_{\loc}^{s'}(M; E)$, by
  the standard mapping properties of pseudodifferential operators.
  Since $s' > 1/2$, the trace $\maS_{P}h \in H_{\loc}^{s'-1/2}(\Gamma; E)$ is well defined
  and it coincides with the traces from the two domains with boundary $\Gamma$,
  see Lemma \ref{lemma.limits=traces}.
  The proofs of (ii) and (iii) are the same, using the corresponding points in
  Lemma \ref{lemma.def.hdelta2}, Lemma \ref{lemma.limits=traces}, and the corresponding
  mapping properties of the respective pseudodifferential operators.
  (For (iii), we also use $H^{-r}(M; E)^{*} \simeq H^{r}(M; E)$, since $M$ has bounded
  geometry.)
\end{proof}

Because the trace of $\maS_{P}h := P(h \delta_{\Gamma})$
at $\Gamma$ is defined and it coincide with the traces associated to
the domains $\Omega_{\pm}$ with boundary $\Gamma$, we shall concentrate
on the restriction (or trace) $\maS_{P}h\vert_{\Gamma}$
of $\maS_{P}h$ to $\Gamma$. The behavior of
this restriction is the content of the following theorem.

\begin{theorem}\label{thm.main.jump1b}
  Let $E, F\to M$ be two hermitian vector bundles, $m < -1$, and
  $s' \in (1/2, -m -1/2)$. Then, for any $P \in \Psi^{m}(M; E, F)$,
  there exists a unique $P_{0} \in \Psi^{m+1}(\Gamma; E, F)$ with
  the following properties:
  \begin{enumerate}[\rm (1)]
    \item For any $h \in L_{\comp}^{2}(\Gamma; E)$, we have
    $\maS_{P}h := P(h \delta_{\Gamma}) \in H_{\loc}^{s'}(M; F)$, and
    hence the traces of $\maS_{P}h := P(h \delta_{\Gamma})$ at the two sides
    of $\Gamma$ are defined and they satisfy
    \begin{equation*}
      [\maS_{P}h]_{+} \seq [\maS_{P}h]_{-} \seq [\maS_{P}h]\vert_{\Gamma}
      \seq P_{0}h \in H_{\loc}^{s'-1/2}(\Gamma; F)\,.
    \end{equation*}

    \item
    For any $x \in \Gamma := \pa \Omega$ and $\xi' \in T_{x}^{*}\Gamma$,
    let $\xi \in T_{x}^{*}M$ be a lift of $\xi'$.
    The principal symbol of $P_{0}$ is then given by
    \begin{equation*}
      \sigma_{m+1}(P_{0}; \xi') \seq \frac1{2\pi}\int_{\RR} \sigma_{m}(P; \xi
      + t \bsnu^{\sharp}_{x}) \, dt\,.
    \end{equation*}

    \item The distribution kernel of the operator $P_{0}$ satisfies
    $k_{P_{0}}(x', y') = k_{P}(x', y')$ for all $x' \neq y'$ in $\Gamma$, and hence
    $(\phi P \psi)_{0} = \phi P_{0} \psi$, for all $\phi, \psi \in C_{0}^{\infty}(M)$.
  \end{enumerate}
\end{theorem}

The operator $P_{0}$ will be called the {\it restriction at $\Gamma$ operator}
associated to $P$.

\begin{proof}
  Let us notice that the relations $[\maS_{P}h]_{+} \seq [\maS_{P}h]_{-}
  \seq [\maS_{P}h]\vert_{\Gamma} \in H_{\loc}^{s'-1/2}(\Gamma; F)$ have already been
  proved (see Lemma \ref{lemma.enough.reg}). Also, the last equality in (iii)
  is an immediate consequence of the equality of kernels (because $k_{\phi P \psi}
  = \phi k_{P} \psi)$.

  Let us assume that the distribution kernel $k_{P}$ of $P$ is
  \emph{compactly supported} (in $M \times M$) and prove our theorem in this case.
  We may also assume that $E$ and $F$ are trivial, one dimensional.
  Since we have assumed that the support $\supp k_{P} \subset M \times M$
  of the distribution kernel of $P$ is compactly supported, its two projections
  $K_{1} := p_{1}\supp k_{P} \subset M$ and $K_{2} := p_{2}\supp k_{P} \subset M$
  are also compact. Hence $K := K_{1} \cup K_{2} \cup \supp h$ is also compact.

  For each $x \in \Gamma$, we choose local coordinates $y$ in a neighborhood $V_{x}$
  of $x$ that straighten out the boundary to the hyperplane by mapping it to $\{x_{n} = 0 \}
  \subset \RR^{n}$. We can choose these coordinates such that they map
  $\exp(t \bsnu)$ to $(y', t) \in \RR^{n-1} \times (-\epsilon, \epsilon)$.
  Let us cover $\Gamma \cap K$ with finitely many such neighborhoods $V_{j} := V_{x_{j}}$,
  which is possible since $K$ is compact. Let us then choose a smooth partition of unity
  $\phi_{0}, \phi_{1}, \ldots, \phi_{N}$ on $M$ subordinated
  to  $\{M \smallsetminus K, V_{1}, \ldots, V_{N}\}$.
  We can assume that $\phi_{0}$ vanishes in a neighborhood of
  $\Gamma \cap K$. By refining the covering $\{V_{j}\}$ of $\Gamma$,
  we can assume that the support of each $\phi_{i} P \phi_{j}$, $1 \le i , j \le N$,
  is completely contained in a set of the form $V_{x}$. Then we use Theorem
  \ref{thm.main.jump-2} for each of the operators $\phi_{i} P \phi_{j}$ on the
  coordinate neighborhood $V_{x}$ to obtain the limit operator
  $P_{0ij} \in \Psi^{m+1}(\Gamma)$.
  We define $P_{0} := \sum_{i, j=1}^{N} P_{0ij}$. Then, for each of these operators,
  we have
  \begin{equation}\label{eq.lemma.jump1b.aux}
    \begin{gathered}
      \phi_{i} \big[ \maS_{P}(\phi_{j} h) \big]\vert_{\Gamma}
      \seq P_{0ij} h\,,\\
      \sigma_{m+1}(P_{0ij}; \xi')
      \seq \frac{\phi_{i}}{2\pi} \left( \int_{\RR}
      \sigma_{m}(P; \xi
      + t \bsnu^{\sharp}_{x}) \, dt \, \right) \phi_{j}\,, \quad \mbox{and}\\
      k_{P_{0ij}}(x', y') \seq k_{\phi_{i} P \phi_{j}}(x', y')
      \seq \phi_{i}(x') k_{P}(x', y') \phi_{j}(y')
      \,.
    \end{gathered}
  \end{equation}
  by Theorem \ref{thm.main.jump-2}. Adding up all the corresponding relations
  for $i, j = 1, \ldots, N$, and noticing that $\sum_{i=1}^{N} \phi_{i} = 1$ on
  $\Gamma \cap K$ (recall that $\phi_{0}$ vanishes in a neighborhood of
  $\Gamma \cap K$), we obtain (ii) and
  \begin{equation*}
      k_{P_{0}}(x', y') \ede \sum_{i, j=1}^{N} k_{P_{0ij}}(x', y') \seq
      \sum_{i, j=1}^{N} \phi_{i}(x') k_{P}(x', y') \phi_{j}(y') \seq
      k_{P}(x', y')
  \end{equation*}
  for all $x', y' \in \Gamma$, $x' \neq y'$. We have thus proved also (iii).
  To complete (i), let $h \in L_{\comp}^{2}(\Gamma; E)$.
  Then $\maS_{P}(\phi_{j} h) := P\big[ (\phi_{j} h) \delta_\Gamma \big ]
  \in H_{\loc}^{s'}(M; E)$, by Lemma \ref{lemma.enough.reg},
  because $s' \in (1/2, -m - 1/2)$. Therefore, $\maS_{P} h\in H_{\loc}^{s'}(M; E)$
  as well, by linearity. This gives
  \begin{equation*}
      \big[ \maS_{P} h \big]\vert_{\Gamma} \seq \sum_{i, j=1}^{N}
      \phi_{i} \big[ \maS_{P}(\phi_{j} h) \big]\vert_{\Gamma} \seq
      \sum_{i, j=1}^{N} P_{0ij} h \, =: \, P_{0}h \,,
  \end{equation*}
  where the second equality is from Equation \eqref{eq.lemma.jump1b.aux}
  (a consequence of Theorem \ref{thm.main.jump-2}). This gives
  the last equality of (i) and hence completes the proof of (i) as well.
  (We have already noticed that the first two equalities in (i) are the
  standard properties of Sobolev spaces discussed in Lemma \ref{lemma.enough.reg}.)

  We have thus proved our theorem under the additional hypothesis that
  $k_{P}$ is compactly supported. The general case follows immediately from this
  one by using the results already proved for operators of
  the form $\phi P \psi$, where $\phi, \psi : M \to \CC$ are smooth and compactly
  supported. (Operators of this form will have compactly supported distribution
  kernels.)

  Let us prove (i), for example. Let $\psi \in \CIc(M)$ be equal to 1 on
  the support of $h$. Let $x \in \Gamma$ arbitrary and $U$
  a relatively compact neighborhood of $x$ in $U$. Let $\phi \in \CIc(M)$
  be equal to $1$ on $U$. We first define $P_{0}h \vert_{\Gamma \cap U}
  := (\phi P \psi)_{0} h \vert_{\Gamma \cap U} $. This definition is independent
  of $\phi$ and $\psi$ by (iii) for compactly supported distribution kernels already proved.
  We then have
  \begin{equation*}
    [\maS_{P}h] \vert_{\Gamma \cap U}
    \seq [\maS_{\phi P \psi }h]\vert_{\Gamma \cap U}\\
    \seq (\phi P \psi)_{0}h \vert_{\Gamma \cap U}
    \, =:\, P_{0}h\vert_{\Gamma \cap U}\,.
  \end{equation*}
  Since $x$ was arbitrary, we obtain that $[\maS_{P}h] \vert_{\Gamma}$
  and $P_{0}h$ coincide in the neighborhood of every point, and hence they are equal.
  The proofs of (ii) and (iii) in general (for arbitrary support of $k_{P}$) are completely
  similar (even simpler).
\end{proof}

The following theorem gives some additional properties of the operator
$P_{0}$ of the previous theorem under the additional assumption that the
inclusion $\Gamma \to M$ is proper and that $P$ is propertly supported.

\begin{theorem}\label{thm.prop.jump1b}
  Let $E, F\to M$ be two hermitian vector bundles, $m < -1$,
  $s' \in (1/2, -m -1/2)$, $P \in \Psi^{m}(M; E, F)$, as in Theorem
  \ref{thm.main.jump1b}. Let us assume also that the inclusion $\Gamma \subset M$
  is proper and that $P$ is propertly supported. Then the operator
  $P_{0} \in \Psi^{m+1}(\Gamma; E, F)$ associated to $P$ by Theorem
  \ref{thm.main.jump1b} has the following additional properties:
  \begin{enumerate}[(i)]
    \item $P_{0}$ is also properly supported and,
    for any $h \in L_{\loc}^{2}(\Gamma; E)$, $\maS_{P}h := P(h \delta_{\Gamma})$
    and $P_{0}h$ are defined and
    \begin{equation*}
      [\maS_{P}h]_{+} \seq [\maS_{P}h]_{-} \seq [\maS_{P}h]\vert_{\Gamma}
      \seq P_{0}h \in H_{\loc}^{s'-1/2}(\Gamma; F)\,.
    \end{equation*}

    \item If, moreover, $\Gamma$ has an $\epsilon$-normal tubular neighborhood, then,
    for all $s \in \RR$, all $h \in H_{\loc}^{s}(\Gamma; E)$, and all $t \in
    (-\epsilon, \epsilon)$ there exist $P_{t} \in \Psi^{m+1}(\Gamma; E)$ such that,
    using the notation and identification of Equation \eqref{eq.p.transport}, we have
    $[\maS_{P}h]_{t} := [P (h \delta_{\Gamma})]_{t} = P_{t}h$ and
    \begin{equation*}
      [\maS_{P}h]_{\pm} \ede \big[ P(h \delta_{\Gamma})]_{\pm}
      \seq \lim_{t \to \pm 0} P_{t}h
      \seq P_{0} h \in H_{\loc}^{s -m -1}(\Gamma; F)\,.
    \end{equation*}
  \end{enumerate}
\end{theorem}

\begin{proof}
  The distributions $h \delta_{\Gamma}$ are defined
  using Lemma \ref{lemma.def.hdelta2}, because the inclusion $\Gamma \to M$ is proper.
  Moreover, $\maS_{P}h := P(h \delta_{\Gamma}) \in H_{\loc}^{s'}(M; F)$ is
  defined because we have assumed that $P$ is properly supported. This gives the first
  two equalities of (i). To complete the proof of (i), we shall prove the last
  equality using the point (i) of Theorem \ref{thm.main.jump1b} as follows.
  It is enough to prove that
  \begin{equation}\label{eq.prop.jump1b.aux}
    \phi [\maS_{P}h]\vert_{\Gamma}
    \seq \phi P_{0}h \in H_{\loc}^{s'-1/2}(\Gamma; F)\,.
  \end{equation}
  for all $\phi$ smooth with compact support. The operator $P_{0}$ will also be
  properly supported because its kernel is the restriction of that of $P$ and
  $\Gamma \to M$ is proper. Because $P$ and $P_{0}$ are properly supported, for any
  such given $\phi$, and $\psi \in \CIc(M)$ with large support, we have
  \begin{equation*}
    \phi [\maS_{P}h]\vert_{\Gamma} \seq \phi [\maS_{P}(\psi h)]\vert_{\Gamma}
    \seq \phi P_{0} \psi h \seq \phi P_{0} h \in H_{\loc}^{s'-1/2}(\Gamma; F)\,.
  \end{equation*}
  This proves Equation \eqref{eq.prop.jump1b.aux} and hence the equality
  $[\maS_{P}h]\vert_{\Gamma} \seq \phi P_{0}h$.

  The point (ii) is the same as that of (i), but replacing the relation
  $\phi_{i} \big[ \maS_{P}(\phi_{j} h) \big]\vert_{\Gamma}
  \seq P_{0ij} h$ with $\phi_{i} \big[ \maS_{P}(\phi_{j} h) \big]_{\pm}
  \seq P_{0ij} h$ in Equation \eqref{eq.lemma.jump1b.aux}.
\end{proof}

\subsection{Lateral limits on manifolds for operators of order $m = -1$}
We now turn to the case of operators of order $-1$.
Recall from Notation \ref{not.bsnu} that $\bsnu$ is a fixed
vector field on $M$ that is the outer unit normal vector
to $\Gamma := \pa \Omega$. Also, recall that $\sharp : TM \to T^{*}M$
is the isomorphism defined by the metric.

\begin{notation}\label{not.ord-1}
  For $P \in \Psi^{m}(M; E, F)$, we let
  $\sigma_{m}(P) \in \CI(T^*M \smallsetminus \{0\}; \Hom(E; F))$ denote
  its principal symbol, and we shall write $\sigma_{m}(P;  \xi) \in
  \Hom(E_{x}, F_{x})$ for its value at $\xi \in T_{x}^*M \smallsetminus \{0\}$.
  If $m = -1$, we then let   $\JC_{+}(P; x), \JC_{-}(P; x) \in \Hom(E_{x}, F_{x})$
  to be defined by
  \begin{equation*}  
    \JC_{+}(P; x) \ede \sigma_{-1}(P; -\bsnu_{x}^{\sharp})\ \mbox{ and }\
    \JC_{-}(P; x) \ede -\sigma_{-1}(P; \bsnu_{x}^{\sharp}) \,.
  \end{equation*}
  We also let, for all $0\neq \xi' \in T^{*}\Gamma$,
  $\xi \in T_{x}^{*}M$ be such that it projects onto $\xi'$ and is
  orthogonal to $\bsnu_{x}^{\sharp}$ and
  \begin{equation*}  
    b_{0}(\xi') \ede \frac1{4\pi}\, \int_{\RR}\big[
      \sigma_{-1}(P; \xi + \tau \bsnu_{x}^{\sharp}) +
      \sigma_{-1}(P; \xi - \tau \bsnu_{x}^{\sharp}) \big] \, d\tau
      \in \Hom(E_{x}, F_{x})\,.
  \end{equation*}
\end{notation}

The choice of sign in the definition of $\JC_{+}$ is due to the fact
that $e_{n} = -\bsnu$ in the Euclidean case (see Section \ref{sec.sec2}).
Recall that $\Psi_{cl}^{m}(M; E, F)$ denotes the set of order $m$
classical pseudodifferential operators on $M$
acting from sections of a smooth vector bundle $E \to M$ to sections of a vector bundle
$F \to M$.

\begin{theorem}\label{thm.main.jump1}
  Let $P \in \Psi_{\cl}^{-1}(M; E, F)$ and assume that $\Gamma := \pa \Omega $ has an
  $\epsilon$-normal tubular neighborhood. Let $\JC_{+}(P)$, $\JC_{-}(P)$, and $b_{0}$
  be as in Notation \ref{not.ord-1} and assume that
  $\JC_{+}(P) = \JC_{-}(P)$. Then, for $t \in (-\epsilon, \epsilon)$,
  there exist pseudodifferential operators $P_{t} \in \Psi^{0}(\Gamma; E, F)$, such that,
  using the notation and identification of Equation \eqref{eq.p.transport}, we have
  $[\maS_{P}h]_{t} := [P (h \delta_{\Gamma})]_{t} = P_{t}h$ for $t \neq 0$ and,
  if we let $P_{0\pm} \ede \pm \frac{\imath}2 \JC_{+}(P) + P_{0}$, then,
  \begin{enumerate}[\rm (1)]
    \item for all $s \in \RR$ and all $h \in H_{\comp}^{s}(\Gamma; E)$, we have
    \begin{equation*}
      \big[ \maS_{P}h \big]_{\pm} \ede \big[ P(h \delta_{\Gamma}) \big]_{\pm}
      \ede \lim_{t \to \pm 0} P_{t} h \seq P_{0\pm} h \in H_{\loc}^{s}(\Gamma; F)\,.
    \end{equation*}

    \item $\sigma_{0}(P_{0}) = b_{0}$;

    \item $k_{P_{0}}(x', y') = k_{P}(x', y')$ for all $x' \neq y'$ in $\Gamma$;
  \end{enumerate}
\end{theorem}

\begin{proof}
  For the most part, the proofs of (i), (ii), and (iii)
  is word-for-word the same as the one of Theorem \ref{thm.main.jump1b},
  whose notations we use here as well, but using Theorem \ref{thm.main.jump0}
  instead of Theorem \ref{thm.main.jump-2} (which justifies the assumption that
  $P$ be classical).
  In particular, we begin again with the case when $P$ has compactly supported
  distribution kernel. For instance, the first two relations in the
  crucial Equation \eqref{eq.lemma.jump1b.aux} are replaced with
  \begin{equation*}  
    \begin{gathered}
      \phi_{i} \big[ \maS_{P}(\phi_{j} h) \big]_{\pm}
      \seq \big( P_{0ij} \big)_{\pm} h \qquad \mbox{and}\\
      \sigma_{m+1}({P_{0ij}}; \xi')
      \seq \frac{\phi_{i}}{4\pi} \left( \int_{\RR}
      \sigma_{m}({P}; \xi + t \bsnu^{\sharp}_{x})
      + \sigma_{m}({P}; \xi - t \bsnu^{\sharp}_{x}) \, dt \, \right) \phi_{j} \,,
    \end{gathered}
  \end{equation*}
  where $\xi \perp \bsnu_{x}$ projects onto $\xi'$. (The last relation of that equation does
  not change.) The only thing that we need to add to complete the proofs of (1), (2), and (3)
  is to notice that $\JC_{\pm}(P) = \sum_{ij=1}^{N} \JC_{\pm}(\phi_{i}P \phi_{j})$.
\end{proof}

If in the previous theorem $P$ is the Laplacian,
$P = \Delta := -d^{*}d$ or the classical Stokes
operator $ \bsXi_{0, 0}$, then $b_{0}= 0$ and hence $P_{0}$
is of order $-1$, but that is not
true in general. For instance, as we will see below,
it is not true if $P =  \bsXi_{V, V_{0}}$, unless
$V_{0}$ vanishes identically on $\Gamma = \pa \Omega$.
The following corollary extends the corresponding statements in Corollary 15.4.6 in
\cite{KMNW-2025} from the case of manifolds with cylindrical ends to that
of arbitrary manifolds.

\begin{corollary}\label{cor.adjoints}
  Let $P \in \Psi^{m}(M; E, F)$ be as in Theorem \ref{thm.main.jump1b}
  (i.e. $m < -1$) or as in Theorem \ref{thm.main.jump1} (i.e. $m = -1$
  and $P$ is classical). Then $(P^{*})_{0} = (P_{0})^{*}$ and, when
  $P$ is classical, $\JC_{+}(P^{*}) = \JC_{+}(P)^{*}$.
\end{corollary}

\begin{proof}
  The proof is similar to that of the corresponding statements in the case of
  manifolds with cylindrical ends (Corollary 15.4.6 in \cite{KMNW-2025}).
  Let $x', y' \in \Gamma$, $x' \neq y'$. Theorems \ref{thm.main.jump1b} and
  \ref{thm.main.jump1} then yield the second and the last of the following
  sequence of relations:
  \begin{equation*}
    k_{P_{0}^{*}}(x', y') \seq k_{P_{0}}(y', x')^{*} \seq
    k_{P}(y', x')^{*} \seq k_{P^{*}}(x', y') \seq k_{(P^{*})_{0}}(x', y')\,.
  \end{equation*}
  Both operators $P_{0}^{*}$ and $(P^{*})_{0}$ are determined by the values
  of their distribution kernels outside the diagonal. Since these distribution
  kernels of $(P^{*})_{0}$ and $(P_{0})^{*}$ coincide, we have
  $P_{0}^{*} = (P^{*})_{0}$, as claimed. The last statement follows from
  \begin{equation*}
    \JC_{+}(P^{*}) \seq \sigma_{-1}(P^{*}; -\bsnu^{\sharp})
    \seq \sigma_{-1}(P; -\bsnu^{\sharp})^{*} \seq \JC_{+}(P)^{*}\,.
  \end{equation*}
  The proof is now complete.
\end{proof}

We can therefore write $P_{0}^{*} = (P^{*})_{0} = (P_{0})^{*}$
without danger of confusion.

\subsection{Mapping properties}

The equality of traces is missing in the last theorem because we first need to
recall some mapping properties of the potential operator
$\maS_{P}(h):= P(h \delta_{\Gamma})$. Our main reference for mapping properties
of the layer potentials is \cite{H-W}, where symbols of rational type are discussed in
detail and where references to the original results can be found. A symbol is of
\emph{rational type} if in every fiber it is a quotient of polynomial functions.
In particular, a symbol of rational type is classical.

\begin{theorem}\label{thm.mapping.lp}
  Let $P \in \Psi^{m}(M; E, F)$ have symbol of \emph{rational type}
  (a quotient of polynomial functions).
  Let $\Omega_{+} := \Omega$, $\Omega_{-} := M \smallsetminus \overline \Omega$,
  and $\Gamma = \pa \Omega_{\pm}$, as before. Then, for any $s \in \RR$, any
  compact set $K \subset \Gamma$, any relatively compact open subset $L \subset M$,
  and any $h \in H^{s}(M; E)$ with support in $K$,
  there exists $C_{s, K, L} \ge 0$ such that
  \begin{equation*}
    \|\maS_{P} h \vert_{\Omega_{\pm}}\|_{H^{s - m - \frac12}(\Omega_{\pm} \cap L; F)} \ede
    \|P (h \delta_{\Gamma})
    \vert_{\Omega_{\pm}}\|_{H^{s - m - \frac12}(\Omega_{\pm}; F)}
    \le C_{s, K, L} \|h\|_{H^{s}(\Gamma; E)}\,.
  \end{equation*}
\end{theorem}

\begin{proof}
  If the distribution kernel $k_{P}$ of $P$ has compact
  support, the result follows from Theorem 9.4.7 on page 584 of \cite{H-W}.
  Let $\phi \in \CIc(M)$ be equal to $1$ on $K \cup L$. Then the result
  is true for $\phi P \phi$, because it has a compactly supported distribution
  kernel. This gives immediately the desired result.
\end{proof}

\section{The deformation and Stokes operators and Green formulas}
\label{sec.Def}

We now recall the definitions and the properties of some needed differential operators,
including the deformation operator $\Def$,
the Stokes operator $\bsXi_{0, 0}$ of Equation \eqref{eq.def.bsXi} (corresponding to
$V$ and $V_{0}$ vanishing in that definition of $\bsXi_{V, V_{0}}$).
To establish some of the basic properties of these operators, we will need various
Green-type formulas that we study using the following ``abstract integration by
parts'' approach. The results of this section
are known (although not very easy to find in the literature, see \cite{KMNW-2025} for
references and the missing proofs).

\subsection{A general integration by parts formula}
Let $E, F \to M$ be Hermitian vector bundles and let $P : \CI(M; E) \to \CI(M; F)$
be a {\it first order} differential operator. Let $\Omega \subset M$ be an open subset
with smooth boundary $\Gamma := \pa \Omega$ such that $\Omega$ is on one side of
$\Gamma$, as before. Let then $\pa_{\bsnu}^{P} : \CI(M;E)\to \CI(\Gamma ;F)$ be defined 
by the following \emph{abstract integration by parts formula}
\begin{equation}\label{eq.def.panu}
    (Pu, v)_{\Omega} \seq (u, P^{*}v)_{\Omega}
    + (\pa_{\bsnu}^{P} u, v)_{\Gamma}\,,\ \ u\in \CIc(M; E),\, v\in \CIc(M; F)\,.
\end{equation}
The next proposition is Proposition 9.1 from Chapter 2 of
\cite{Taylor1} (see also Proposition A.3.14 from \cite{KMNW-2025});
it states that there exists an operator $\pa_{\bsnu}^{P}$
with these properties. Here $P^{*}$ is the \emph{formal} adjoint of $P$
or any extension of it. (Recall that the formal adjoint is defined using only
smooth, compactly supported functions.) Also, $\bsnu$ is the \emph{outer unit normal
vector} to $\Gamma := \pa \Omega$, as before. Recall that this vector was extended
to a globally defined smooth vector field on $M$.

\begin{proposition} \label{prop.bdry.op}
  Let $P : \CIc(M; E) \to \CIc(M; F)$ be a first order differential operator let and
  $\sigma_{1}(P) : T^{*}M \to \Hom(E; F)$ denote its principal symbol,
  as usual. Then
  \begin{equation*}
    \pa_{\bsnu}^{P} \seq - \imath \sigma_{1}(P; \bsnu^{\sharp}) \in \Hom(E; F)\,.
  \end{equation*}
  In particular,
    $(Pu, w)_{\Omega} = (u, P^{*}w)_{\Omega}
    - \imath (\sigma_{1}(P; \bsnu^{\sharp})u, w)_{\Gamma}.$
\end{proposition}

\subsection{Diferential operators}
This formula will be used for a number of differential operators that we
introduce next. One of the most basic ones is the {\it Levi-Civita connection}
\begin{equation*}
  \nabla^{LC} : \CI(M; TM) \to \CI(M; T^{*}M \otimes TM)\,,
\end{equation*}
which is the unique torsion-free, metric preserving connection on $TM$.
One should not confuse $\nabla^{LC}X \in \CI(M; T^{*}M \otimes TM)$
with the \emph{gradient} $\nabla f := (df)^{\sharp} \in \CI(M; TM)$. Its extension
to other tensor bundles will also be denoted by $\nabla^{LC}$.
We shall need the \emph{deformation operator}
$\Def : \CI(M; TM) \to \CI(M; T^{*}M \otimes T^{*}M)$,
$\Def(X) \ede \frac12 \maL_{X}g_{M}$, where $\maL_{X}$ denotes
the Lie derivative in the direction of $X$. A more useful equivalent definition of
$\Def$ is
\begin{equation}\label{eq.def.Def}
    \Def(X)(Y, Z) \ede \langle \Def(X), Y \otimes Z \rangle
    \seq \frac12 \big [ (\nabla_{Y}^{LC} X) \cdot Z
    + (\nabla_{Z}^{LC} X) \cdot Y \big]\,,
\end{equation}
where $X, Y$, and $Z$ are smooth vector fields on $M$ and $X \cdot Y = g_{M}(X, Y)$
is the scalar product induced by the metric $g_{M}$ on $M$.

Recall the isomorphism $\sharp : TM \to T^{*}M$ induced by the metric $g_{M}$ on
$M$. Its inverse will be denoted by the same symbol.
The vector field $\bsnu$ defines maps
\begin{equation*}
  \bsnu \otimes \sharp\,,\ \sharp \otimes \bsnu : T^{*}M \otimes T^{*}M \to TM\,.
\end{equation*}
(For instance, the first map is explicitly given by
$(\bsnu \otimes \sharp)(\xi \otimes \eta)
= \xi(\bsnu) \eta^{\sharp}$.) By  $\langle \ \cdot \ ,\, \bsnu \otimes 1\,\rangle
: T^{*}M \otimes T^{*}M \to \CC \otimes T^{*}M =T^{*}M$ we shall denote the contraction
with $\bsnu$ on the first variable. Then
$\Dnu  : \CI(M; TM) \to \CI(M; TM)$ is given by
\begin{equation}
  \label{eq.def.Dnu}
    \Dnu  X \ede \frac12 (\bsnu \otimes \sharp +
    \sharp \otimes \bsnu) \, \Def(X) \seq \big\langle\, \Def(X)\, ,
    \, \bsnu \otimes 1\, \big\rangle^{\sharp}
    \,.
\end{equation}
(The last equation is sometimes written
$\Dnu X := (\Def(X) \bsnu \otimes 1)^{\sharp}$ \cite{D-M, M-T, Varnhorn}.)
We can finally define the operator $\bop : \CI(M; TM \oplus \CC) \to \CI(M; TM)$
\begin{equation}\label{def.conormal.Tderiv}
    %
    \bop \cvector{\bsu}{p} \ede -2 \Dnu(\bsu) + p \bsnu \,, \quad
    \mbox{where } \bsu \in \CI(M; TM) \mbox{ and } p \in \CI(M)\,.
\end{equation}
The operator $\bop$ (and hence also $\Dnu$)
will play an important operator in the study of the Stokes equations.
We shall consider the operator
\begin{equation}\label{def.conormal.Tderiv2}
  \tbop  U \ede
  \cvector{-2 \Dnu(\bsu) + p \bsnu}{0} \seq
  \left(
    \begin{array}{cc}
      -2 \Dnu & \bsnu \\
      0 & 0
    \end{array}
  \right)
  U\,.
\end{equation}

Let $V , V_{0} : M \to [0, \infty)$. Recall from the Equation \eqref{eq.def.bsXi}
in the introduction that the \emph{deformation Laplacian} is the second order
differential operator $\bsL \ede 2\Defstar\Def$. The operator
  $\bsL_{V} \ede 2\Defstar \Def + V$
will be called the \emph{perturbed deformation Laplacian.}
Also, recall that the \emph{generalized Stokes operator} is
the operator
\begin{equation*} 
  \bsXi \ede \bsXi_{V, V_{0}}  \ede \left(\begin{array}{ccc}  \bsL_{V} & \nabla \\
  \nabla^* &  -V_0
  \end{array}
  \right)  \in \End(\CI(M; TM \oplus \CC))\,.
\end{equation*}
(By $\End(V)$, we denote the space of \emph{endomorphisms} (suitable
linear maps $V \to V$) of a module $V$ over some ring that is clear
from the context.)

We now study some of the properties of these operators.
A direct calculation gives right away the following result.

\begin{lemma}\label{lemma.bop.star}
  We have
  \begin{equation*}
    \bopstar (\bsu) \seq
    \left (
      \begin{array}{c}
      - 2\Dnu^{*}\bsu \\
      \bsnu \cdot \bsu
      \end{array}
    \right )
    \seq \left ( \begin{array}{c}
      - 2\Dnu^{*} \\
      \bsnu^{\sharp}
    \end{array}
    \right )\, \bsu \seq
    \left(
    \begin{array}{cc}
      -2 \Dnu^{*} & 0 \\
      \bsnu^{\sharp} & 0
    \end{array}
  \right)
  U
  \seq \tilde{\boldsymbol T}_{\bsnu}^{*} U
  \,.
  \end{equation*}
\end{lemma}

We shall need the following notation. Let $V$ be a vector space and $v \in V$ and
$w \in V^{*}$. We let then $v \otimes w \in \End(V)$ be the endomorphism
(i.e., linear map $V \to V$) defined by
\begin{equation*}
  (v \otimes w)x \ede w(x) v \,.
\end{equation*}
In particular, if $V$ is hermitian with isomorphism $\sharp : V \to V^{*}$
induced by the metric,
then $(v \otimes w^{\sharp}) x = (w\cdot x) v$ and
$(v \otimes w^{\sharp})^{*} = w \otimes v^{\sharp}$.
We let $T^{*\otimes 2}M \ede T^{*}M \otimes T^{*}M$.
We now recall for completeness some well known formulas, some of which will be
used in what follows, see Section A.3 of \cite{KMNW-2025} for references and proofs.

\begin{proposition}\label{prop.formulas.new}
  Let $X$, $Y$, and $Z$ be smooth vector fields on $M$. Then
  \begin{enumerate}[\rm (1)]
    \item \label{prop.fn.1}
    $\sigma_{1}(\Def;  \xi)X \seq \frac{\imath}2 \big[  \xi \otimes X^{\sharp}
    + X^{\sharp} \otimes  \xi] \in S^{2}T^{*}M \subset T^{*\otimes 2}M$
    \item \label{prop.fn.2}
    $\sigma_{1}(\Defstar;  \xi) \big( Y^{\sharp} \otimes Z^{\sharp} \big)
    \seq -\frac{\imath}2\big [  \xi(Y)Z +  \xi(Z)Y \big]\,.$
    \item \label{prop.fn.3}
    $\pa_{\bsnu}^{\Defstar}\Def = -\Dnu $;
    \item \label{prop.fn.4}
    $\sigma_{1}(\Dnu ;  \xi) = \frac{\imath}2 \big [  \xi(\bsnu)
    +  \xi^{\sharp} \otimes \bsnu^{\sharp} \big]$;
    \item \label{prop.fn.5}
    $\sigma_{1}(\Dnu ^{*};  \xi) = - \frac{\imath}2
    \big [  \xi(\bsnu) + \bsnu \otimes  \xi \big]$; and
    \item \label{prop.fn.6}
    $\sigma_{2}(\Defstar\Def;  \xi) = \frac12(| \xi|^{2}
    +  \xi^{\sharp} \otimes  \xi)$.
  \end{enumerate}
\end{proposition}

\subsection{Green formulas on $\Omega$}\label{sec.Green}
We now recall some Green-type formulas on an open set $\Omega =: \Omega_{+}
\subset M$ with smooth boundary. Recall that $\Omega_{-} := M \smallsetminus \overline \Omega$
and that we assume that both $\Omega = \Omega_{+}$ and $\Omega_{-}$ have boundary $\Gamma$.

To state our Green-type formulas, we will use the following notation:
\begin{equation}\label{eq.def.UWB}
  \begin{gathered}
    U \ede \cvector{\bsu}{p} \seq (\bsu \ \ p)^{\top }\,, \qquad
    W \ede \cvector{\bsw}{q} \seq (\bsw \ \ q)^{\top }\,,\\ \qquad
    \mathfrak v \ede (V\bsu, \bsw)_{\Omega} - (V_{0}p, q)_{\Omega}\,, \quad \mbox{and}\\
    B_{\Omega}(U, W) \ede 2(\Def \bsu, \Def \bsw)_{\Omega}
    + (\nabla^{*} \bsu, q)_{\Omega}
    + (p, \nabla^{*} \bsw)_{\Omega} + \mathfrak v \,,
  \end{gathered}
\end{equation}
where $\bsu$ and $\bsw$ are suitable sections of
$TM$ and $p$ and $q$ are suitable scalar functions.
(As suggested by the notation, the inner products in the last formula
are defined by integration on $\Omega$.)

In the following, $\textbf{1}_{\Omega}$ will denote the characteristic function of the
set $\Omega $ (that is, $\textbf{1}_{A}(x) = 1$ if $x \in A$ and $\textbf{1}_{A}(x) = 0$ if
$x \notin A$). We then have the following representation (or Green-type) formulas 
(see \cite{KMNW-2025, M-W, Varnhorn}).

\begin{proposition} \label{prop.Green}
  Let $\bsXi = \bsXi_{V, V_{0}}$ be our modified Stokes operator \eqref{eq.def.bsXi}
  and $\textbf{1}_{\Omega}$ be the characteristic function of $\Omega.$ Let
  $U := (\, \bsu \ \ \ p \, )^{\top}$ and $W := (\, \bsw \ \ \ q \, )^{\top}$
  be as in Equation \eqref{eq.def.UWB} with $\bsu, \bsw \in H^{2}(\Omega; TM)$
  and $p, q \in H^{1}(\Omega)$. Then
  \begin{enumerate}[\rm (1)]
    \item
      $\big ( \bsXi U, W \big )_{\Omega} \seq B_{\Omega} (U, W) + (\bop U, \bsw)_{\Gamma}
      \seq B_{\Omega} (U, W).$
      \item
      $ \big ( \bsXi U, W \big )_{\Omega} - \big (U, \bsXi W \big )_{\Omega}
      \seq (\bop  U, \bsw)_{\Gamma}  - (\bsu , \bop  W)_{\Gamma}
      \,.$
    \item $\bsXi \big( \textbf{1}_{\Omega} U \big) = \textbf{1}_{\Omega}
      \big(  \bsXi U  \big) - (\tbop  U) \delta_{\Gamma}
      + \tbopstar (U \delta_{\Gamma}).$
  \end{enumerate}
\end{proposition}

We shall need the following definition from \cite{Grosse-Kohr-Nistor-23}.

\begin{definition} \label{def-L2-cont}
  Let $M$ be a manifold. If $M$ is connected, we say that a differential operator
  $T : \CI(M; E) \to \CI(M; F)$ satisfies
  the \textit{$L^2$-unique continuation property} if, given $u \in L^2(M; E)$
  that vanishes in an open subset of $M$ and satisfies $T u \seq 0$,
  then $u \seq 0$ \emph{everywhere} on $M$. For general $M$, we say
  that $T$ satisfies the \textit{ $L^2$-unique continuation property} if
  it satisfies this property on any connected component of $M$.
\end{definition}

This concept allows us to obtain the following corollary.
Recall that a \emph{Killing vector field} $X$ is a vector field
that preserves the metric, equivalently, $\Def X = 0$.

\begin{corollary}\label{cor.e.est}
  Let $V, V_{0} \ge 0$ and
  $U = \cvector{\bsu}{p} \in H^{2}(\Omega; TM) \oplus H^{1}(\Omega)$
  satisfy $\bsXi U = 0$ in $\Omega$ and $(\bop  U, \bsu)_{\Gamma} = 0$. Then we have the following properties:
  \begin{enumerate}[\rm (1)]
    \item $\Def \bsu \seq 0$, $V \bsu \seq 0$, $\nabla^{*} \bsu \seq 0$,
    $V_{0}p \seq 0$, and $\nabla p \seq 0$
    in $\Omega$.\smallskip

    \hspace*{-1.4 cm} Let $\Omega_{0}$ be a connected component of
    $\Omega$.\smallskip

    \item
    If, furthermore, $V_{0} \not \equiv 0$ in $\Omega_{0}$, then
    $p = 0$ on $\Omega_{0}$.

    \item Similarly, if one of the following three conditions is satisfied:
    \begin{enumerate}[\rm $(i)$]
      \item $\Omega_{0}$ has no non-zero Killing vector fields;
      \item $V \not \equiv 0$ on $\Omega_{0}$; or
      \item $\pa \Omega_{0} \neq \emptyset$ and $\bsu = 0$ on $\pa \Omega_{0}$;
    \end{enumerate}
    then $\bsu = 0$ in $\Omega_{0}$.
  \end{enumerate}
  The result remains true if $\Omega = M$ {\rm(}we just drop all terms involving
  $\pa \Omega${\rm)}.
\end{corollary}

\begin{proof}
  Let $\mathfrak w := (p, \nabla^{*}\bsu)_{\Omega}- (\nabla^{*}\bsu, p)_{\Omega}$.
  We notice that $\overline{(p, \nabla^{*}\bsu)}_{\Omega} = (\nabla^{*}\bsu, p)_{\Omega}$,
  and hence the real part $\operatorname{Re}(\mathfrak w)$ of $\mathfrak w$ vanishes.
  Let us take
  \begin{equation*}
    W \ede \cvector{\bsw}{q} \seq \cvector{\bsu}{-p} \, =: \, U'
  \end{equation*}
  in the formula $\big ( \bsXi U, W \big )_{\Omega} \seq B_{\Omega} (U, W)
  + (\bop  U, \bsw)_{\Gamma} $ of Proposition \ref{prop.Green}. Together with
  the definition of $B_{\Omega}$ in Equation \eqref{eq.def.UWB} and with
  $\operatorname{Re}\big[(p, \nabla^{*}\bsu)_{\Omega}
  - (\nabla^{*}\bsu, p)_{\Omega}\big] =: \operatorname{Re}(\mathfrak w) =0$,
  this gives
  \begin{align*}
    0 & \seq \operatorname{Re} \big[\big ( \bsXi U, U' \big )_{\Omega}
    - (\bop U, \bsu)_{\Gamma} \big]
    \seq \operatorname{Re} \big[  B_{\Omega} (U, U')\big]\\
    &
    \seq \operatorname{Re} \big[ 2(\Def \bsu, \Def \bsu)_{\Omega}
    - (\nabla^{*}\bsu, p)_{\Omega}
    + (p, \nabla^{*}\bsu)_{\Omega} + (V \bsu, \bsu)_{\Omega}
    + (V_{0}p, p)_{\Omega} \big]
    \\
    &
    \seq 2(\Def \bsu, \Def \bsu)_{\Omega} + (V \bsu, \bsu)_{\Omega}
    + (V_{0}p, p)_{\Omega}\,.
  \end{align*}
  Because $V, V_{0} \ge 0$, all three terms in the last sum are non-negative,
  so each of them equals zero. Therefore $\Def \bsu = 0$, $V \bsu = 0$,
  and $V_{0}p = 0$ in $\Omega$. We also have
  \begin{equation*}
    0 \seq \bsXi U \seq \cvector{2\Defstar \Def \bsu + V \bsu + \nabla p}
    {\nabla^{*} \bsu - V_{0}p}
    \seq \cvector{\nabla p}{\nabla^{*} \bsu}\,,
  \end{equation*}
  and hence we obtain (i). The condition $\nabla p = 0$ just proved implies that $p$
  is locally constant. Since, moreover, $V_{0}p = 0$,
  this constant is zero on the connected components of $\Omega$ on
  which $V_{0} \not \equiv 0$, and this proves (ii). (Notice that this is exactly
  the $L^{2}$-unique continuation property of $\nabla$.) Similarly, (iii)
  follows from the fact that $\Def$ satisfies the $L^{2}$-unique
  continuation property (see \cite{Grosse-Kohr-Nistor-23}).
\end{proof}

In particular, this corollary gives that
$\bsu = 0$ on $\supp(V) \cap \Omega$ and $p = 0$ on
$\supp(V_{0}) \cap \Omega$.

\subsection{The principal symbol of $\bsXi$}
It turns out that $\bsXi$ is elliptic, but not in the usual sense.
To explain this, we need to recall a few basic definitions related to \ADN\ elliptic
operators.

\begin{definition}\label{def.ADN}
  Let $M$ be a smooth manifold and $s_i, t_j \in \RR$, $i, j \in \{0, 1\}$.
  We set $\bss = (s_{0}, s_{1})$, $\bst = (t_{0}, t_{1})$,
  $E_{0} = TM$, and $E_{1} = \CC$. Then
  \begin{equation*}
    \Psi_{\cl}^{[\bss+\bst]}(M; TM \oplus \CC) \ede
    \{ T=[T_{ij}] \mid\, T_{ij} \in \Psi_{\cl}^{s_i + t_j}(M; E_j, E_i),\
    i, j \in \{0, 1\} \}\,.
  \end{equation*}
  An operator $T = [T_{ij}] \in \Psi_{\cl}^{[\bss+\bst]}(M; TM \oplus \CC)$
  is said to be of \emph{\ADN-order $\le [\bss+\bst]$}. For $T=[T_{ij}] \in
  \Psi_{\cl}^{[\mathbf{s+t}]}(M; TM \oplus \CC)$ let $\Symb_{\bss, \bst}(T) \ede
  [\sigma _{s_i+t_j}(T_{ij})]$ be its $(\bss, \bst)$--principal symbol.
  The operator $T$ is said to be {\it $(\bss, \bst)$--\ADN\ elliptic} if its
  $(\bss, \bst)$--principal symbol matrix $\Symb_{\bss, \bst}(T)$ is invertible outside the
  zero section.
\end{definition}

Let $\End(V)$ denote the space of \emph{endomorphisms} (in a suitable sense) of a module over
some ring that is clear from the context. In this spirit, $\End(TM)$ denotes the space of
linear maps $l : TM \to TM$ that are vector bundle morphisms. It is itself a vector bundle
over $M$. Recall that $\sharp : T^{*}M \to TM$ is the isomorphism defined by the Riemannian
metric $g_{M}$ of $M$ (the ``musical isomorphism''). The following result is also known, and
is proved using, for instance, the formula for $\sigma_{2}(\Defstar\Def)$ in Proposition
\ref{prop.formulas.new}. See Proposition 15.3.29 of \cite{KMNW-2025} for a proof and
further references. See also \cite{H-W, KohrNistor-Stokes, M-T, Wl-Ro-La}.

\begin{proposition} \label{prop.ADN.elliptic}
  Let $\bss = \bst =(1,0)$. Then the generalized Stokes operator
  $\bsXi := \bsXi_{V, V_{0}}$ \eqref{eq.def.bsXi} belongs to
  $\Psi_{\cl}^{[\bss + \bst]}(M; TM \oplus \CC)$ and is
  $(\bss, \bst)$--\ADN\ elliptic {\rm (}Definition \ref{def.ADN}{\rm )}. Its
  $(\bss, \bst)$--principal symbol of $\bsXi$ is
  \begin{equation*} 
    \Symb_{\bss, \bst}\left(\bsXi \right)( \xi)=
    \left(
      \begin{array}{cc}
        | \xi|^2 +  \xi^{\sharp}  \otimes  \xi  & \imath  \xi^{\sharp}\\
        -\imath  \xi & -V_0
      \end{array}
    \right) \in \End(TM \oplus \CC)\,,
  \end{equation*}
  which is invertible for $\xi \neq 0$ with inverse
  \begin{equation*} 
    \left(
      \begin{array}{cc}
        \frac{1}{| \xi |^2} - \frac{V_0+1}{2V_0+1}\frac{1}{| \xi |^4}
        \xi^{\sharp} \otimes  \xi
        & \frac{\imath}{(2V_0+1)| \xi |^2} \xi^{\sharp}\\
        - \frac{\imath}{(2V_0+1)| \xi |^2} \xi & -\frac{2}{2V_0+1}
      \end{array}
    \right) \in \End(TM \oplus \CC)\,.
  \end{equation*}
\end{proposition}

\section{Pseudoinverses, layer potentials, and jump relations}

In this section, we extend the construction of the single and double layer potentials for our
generalized Stokes operator to non-compact manifolds, assuming only the existence of
the Moore-Penrose pseudoinverse $\psdinv$ of $\bsXi$. The Moore-Penrose pseudoinverse
exists in the case of compact manifolds, as we will prove in the next section.
It exists also in the case of manifolds with straight cylindrical ends (see \cite{KMNW-2025}
and \cite{KohrNistor-Stokes}). A novelty of our approach in thus that we do not require the
existence of a (true) inverse of $\bsXi$ in order to define and study the layer potentials,
as is done classically, see \cite{H-W, M-T, M-T1, KMNW-2025, M-W, Varnhorn}.

\subsection{The Moore-Penrose pseudoinverse of $\bsXi$ and its principal symbol}
Let $\mathcal N$ be the kernel of $\bsXi : H_{\loc}^{1} (M; TM) \oplus L_{\loc}^{2}(M)
\to H_{\loc}^{-1} (M; TM) \oplus L_{\loc}^{2}(M)$. Then $\mathcal N$ will consist of {\it smooth}
sections, by elliptic regularity. Assume that $\mathcal N \subset L^{2}(M; TM \oplus \CC)$ and that
it is finite dimensional and let $p_{\maN}$ be the $L^{2}$--orthogonal projection onto $\maN$.
Let us assume that $\bsXi$ has the following invertibility property on the orthogonal
complement of $\maN$. Let $\bss =\bst =(1,0)$, see Definition \ref{def.ADN}.
We assume that there exists a pseudodifferential operator
\begin{equation}\label{eq.def.psdinv}
  \psdinv \in \Psi_{\cl}^{[- \bss - \bst]}(M; TM \oplus \CC)
  \ede \left(
      \begin{array}{cc}
        \Psi^{-2}(M; TM)   & \Psi^{-1}(M; TM, \CC)\\
        \Psi^{-1}(M; \CC, TM) &  \Psi^{0}(M)
      \end{array}
    \right)
\end{equation}
such that
\begin{equation}\label{eq.rel.pnu}
  \begin{gathered}
    \psdinv (1 - p_{\maN}) \seq (1 - p_{\maN}) \psdinv \seq \psdinv \ \mbox{ and }\\
    \bsXi \psdinv \seq \psdinv \bsXi =  1 - p_{\maN}\,.
  \end{gathered}
\end{equation}
These relations uniquely determine $\psdinv$ and the resulting operator will be
called the {\it Moore-Penrose pseudoinverse} of $\bsXi := \bsXi_{V, V_{0}}$.

We also obtain the following result:

\begin{proposition}\label{prop.form.inverse}
  Let $\bss =\bst =(1,0)$ and
  $\psdinv \, =:\, \left(
    \begin{array}{cc}
      \maA  & \maB  \\
     \maC  & \maD
    \end{array}
    \right) \in \Psi_{\cl}^{[- \bss - \bst]}(M; TM \oplus \CC)$
  be the Moore-Penrose pseudoinverse of $\bsXi$,
  as in Equation
  \eqref{eq.rel.pnu}. Then

  \begin{enumerate}[\rm (1)]
    \item $\maA \in \Psi_{\cl}^{-2}(M; TM)$,
    $\maC = \maB^{*} \in \Psi_{\cl}^{-1}(M; TM, \CC)$,
    and $\maD \in \Psi_{\cl}^{0}(M)$.

    \item We have
    $\sigma_{-2}(\maA )(x, \xi) \seq \frac{1}{| \xi |^2} -
      \frac{V_0+1}{2V_0+1}\frac{1}{| \xi |^4}
       \xi^{\sharp} \otimes  \xi \,,$
      $\sigma_{-1}(\maB )(x, \xi) \seq \sigma_{-1}(\maC )(x, \xi)^{*} \seq
      \frac{\imath}{(2V_0+1)| \xi |^2} \xi^{\sharp} \,,$
      and $\sigma_0(\maD )(x, \xi) \seq -\frac{2}{2V_0+1}.$
    \end{enumerate}
\end{proposition}

\begin{proof}
  Let $\bss=\bst =(1,0)$ be as in the statement. The point (1) is a
  consequence of the property that $\bsXi \in \Psi_{\cl}^{[\bss + \bst]}(M; TM \oplus \CC)$
  and the definition of $\Psi_{\cl}^{[- \bss - \bst]}(M; TM \oplus \CC)$,
  except the relation $\maB^{*} = \maC$, which is a consequence of the fact that
  $\bsXi$ is symmetric and its Moore-Penrose pseudoinverse is unique, and hence also
  symmetric. The multiplicativity of the principal symbol $\Symb$ gives that
  \begin{equation*}
    \Symb_{\bss, \bst}(\bsXi)\Symb_{-\bst, -\bss}(\psdinv)
    \seq \Symb_{\bf 0, 0}(1) \seq 1\,.
  \end{equation*}
  Therefore, the $({- \bss}, {- \bst})$-principal symbol
  of the Moore-Penrose pseudo-inverse $\psdinv$ of $\bsXi$ is the inverse of
  the $(\bss, \bst )$-principal symbol of $\bsXi$,
  which is given by Proposition \ref{prop.ADN.elliptic}. Thus, the
  principal symbols of the operator $\maA$, $\maB$, $\maC$, and $\maD$
  (the entries of $\psdinv$)
  are as stated (and as given by Proposition \ref{prop.ADN.elliptic}).
\end{proof}

It will be convenient to simplify the notation for our symbols as follows.

\begin{remark}\label{rem.form}
  Let $f := \frac{V_{0} + 1}{2V_{0} + 1}$ and $g := \frac1{2V_{0} + 1}$.
  Then we have
  \begin{equation*}
    \begin{array}{ll}
      \sigma_{-2}(\maA )(x, \xi) \seq \frac{1}{| \xi |^2} -
      \frac{f}{| \xi |^4}
      \xi^{\sharp} \otimes  \xi \,, \qquad
      & \sigma_{-1}(\maB )(x, \xi)  \seq
      \frac{\imath g}{| \xi |^2} \xi^{\sharp} \,,\\[3mm]
      \sigma_{-1}(\maC )(x, \xi) \seq - \frac{\imath g}{|\xi|^2}  \xi \,,
      \ \mbox{ and }\ \qquad
      & \sigma_0(\maD )(x, \xi) \seq - 2g \,.
    \end{array}
\end{equation*}
\end{remark}

\subsection{Definition of layer potential operators}
\label{ssec.5.2}
The assumption about the existence of the
Moore-Penrose pseudo-inverse $\psdinv$ of $\bsXi$ allows us now to extend the classical
methods to define the single and double layer potential operators for the Stokes operator
(see also \cite{D-M, Ladyzhenskaya, M-T, Varnhorn}). Nevertheless, some care needs to
be exercised. We let $\Gamma := \pa \Omega$,
as usual in this paper. For the following definition,
recall the Stokes operator $\bsXi = \bsXi_{V, V_{0}}$ of Equation
\eqref{eq.def.bsXi}. Also, recall the distribution
$\bsh \delta_{\Gamma}$ of Lemma \ref{lemma.def.hdelta2} and
the operator $\bopstar$ of Lemma \ref{lemma.bop.star}.

\begin{definition} \label{def.lp}
  Let $\bsh\in L^2(\Gamma ; TM)$. The {\it single-layer
  potential} $\maS_{\rm{ST}}(\bsh)$, the {\it single-layer velocity
  potential} $\maV_{\rm{ST}}(\bsh)$, and the {\it single-layer pressure
  potential} $\maP_{\rm{ST}}(\bsh)$ for $\Xi$
  are given by:
  \begin{equation*} 
    \maS_{\rm{ST}}(\bsh) \ede
    \left(
      \begin{array}{c}
        \maV_{\rm{ST}}(\bsh)\\
        \maP_{\rm{ST}}(\bsh)
      \end{array}
      \right)
      \ede \psdinv \, \left[\cvector{\bsh}{0}
      \ \delta_{\Gamma}\right]\,.
  \end{equation*}
  Similarly, the {\it double-layer potential} $\maD_{\rm{ST}}(\bsh)$,
  the {\it double-layer velocity
  potential} $\mathcal{W}_{\rm{ST}}(\bsh)$, and the {\it double-layer pressure
  potential} $\mathcal{Q}_{\rm{ST}}(\bsh)$ for $\bsXi$ are given by:
  \begin{equation*} 
    \maD_{\rm{ST}}(\bsh) \ede
    \left(
      \begin{array}{c}
        \maW_{\rm{ST}}(\bsh)\\
        \maQ_{\rm{ST}}(\bsh)
      \end{array}
      \right)
      \ede \psdinv \, \big[ \bopstar
      \left(
        \bsh \delta_{\Gamma}
      \right) \big]\,.
  \end{equation*}
\end{definition}

These definitions can be made more explicit as follows.

\begin{remark}\label{rem.components}
  We have
  \begin{equation*} 
    \maW_{\rm{ST}}(\bsh) \seq
    \left(
      \begin{array}{cc}
        \maA  & \maB
      \end{array}
    \right)
    \left( \begin{array}{c}
        - 2\Dnu^{*} \\
        \bsnu^{\sharp}
      \end{array}
    \right)
    (\bsh \delta_{\Gamma})
    \seq (- 2 \maA  \Dnu^{*} + \maB \bsnu^{\sharp})\,
    (\bsh \delta_{\Gamma})
  \end{equation*}
  and
  \begin{equation*} 
    \maQ_{\rm{ST}}(\bsh)
    \seq \left(
      \begin{array}{cc}
        \maC  & \maD
      \end{array}
      \right)
      \left( \begin{array}{c}
        - 2\Dnu^{*} \\
        \bsnu^{\sharp}
      \end{array}
      \right)
      (\bsh \delta_{\Gamma}) \\
       \ \seq (- 2 \maC  \Dnu^{*} + \maD \bsnu^{\sharp})\,
      (\bsh \delta_{\Gamma})\,.
\end{equation*}
  Similarly, $\maS_{\rm{ST}}(\bsh) \seq \big(\, \maA (\bsh \delta_{\Gamma}) \ \ \
  \maC(\bsh \delta_{\Gamma})\, \big)^{\top }$.
\end{remark}

Theorem \ref{thm.mapping.lp} gives the following result.

\begin{proposition}\label{prop.Hkestimates}
  We continue to assume that $M$ is compact. Let $\Omega_{+} := \Omega$ and $\Omega_{-}
  := M \smallsetminus \overline{\Omega}$, as before. Let $\bsh \in H^{k + 1/2}(\Gamma; TM)$,
  $k \in \ZZ_{+}$. Then
  \begin{equation*}
    \maV_{\rm{ST}}(\bsh)\vert_{\Omega_{\pm}}
    \in H^{k+2}(\Omega_{\pm}; TM) \ \mbox{ and }\
    \maP_{\rm{ST}}(\bsh)\vert_{\Omega_{\pm}} \in H^{k+1}(\Omega_{\pm})\,.
  \end{equation*}
  Similarly, let $\bsh \in H^{3/2}(\Gamma; TM)$. Then
  \begin{equation*}
    \maW_{\rm{ST}}(\bsh)\vert_{\Omega_{\pm}} \in H^{k+2}(\Omega_{\pm}; TM)\ \mbox{ and }\
    \maQ_{\rm{ST}}(\bsh)\vert_{\Omega_{\pm}} \in H^{k+1}(\Omega_{\pm})\,.
  \end{equation*}
\end{proposition}

\begin{proof}
  Indeed, this follows from Proposition \ref{thm.mapping.lp}
  and Remark \ref{rem.components} because
  $\maA$ has order $-2$, $\maC$ and $\bsP := - 2 \maA  \Dnu^{*} + \maB \bsnu^{\sharp}$
  have order $-1$, and $- 2 \maC  \Dnu^{*} + \maD \bsnu^{\sharp}$ has
  order zero. Moreover, all four of them have rational type symbols.
\end{proof}

The following result is a consequence of the definition of the single and double
layer potentials. Notice the additional conditions are needed since $\bsXi$ is not
necessarily invertible. Let $\bsh \delta_{\Gamma}$ the distribution in
$\CIc(M; TM)'$ of Lemma \ref{lemma.def.hdelta2}, as above.
We shall also write
\begin{equation*}
  ( \bsh \delta_{\Gamma}\  \ 0)^\top \ede \cvector{\bsh \delta_{\Gamma}}{0}
  \seq \cvector{\bsh}{0} \ \delta_{\Gamma} \,.
\end{equation*}

\begin{proposition}\label{prop.compatibility}
  Let $\bsh \in L^{2}(\Gamma; TM)$, let $\maN$ be the kernel of $\bsXi$, and
  let $p_{\maN} \in \Psi^{-\infty}(M; TM \oplus \CC)$ be the
  $L^{2}$-projection onto $\maN$.
  \begin{enumerate}[\rm (1)]
    \item We have $\bsXi \maS_{\rm{ST}}(\bsh) = ( \bsh \delta_{\Gamma}\  \ 0)^\top
    - p_{\maN}(( \bsh \delta_{\Gamma}\ \ 0)^\top)$ and
    $\bsXi \maD_{\rm{ST}}(\bsh) = \bopstar (\bsh \delta_{\Gamma})
    - p_{\maN} \bopstar (\bsh \delta_{\Gamma})$.

    \item Assume that $\int_{\Gamma} \bsh \cdot \bsu \,dS_{\Gamma} = 0$ for all
    $(\bsu, p) \in \maN$. Then
    $\bsXi \maS_{\rm{ST}}(\bsh) = ( \bsh \delta_{\Gamma}\ \ 0)^\top$, and hence it
    vanishes on $M \smallsetminus \Gamma$.

    \item Analogously, assume that $\int_{\Gamma} \bsh \cdot \bop(\bsu, p) \,dS_{\Gamma} = 0$
    for all $(\bsu, p) \in \maN$. Then
    $\bsXi \maD_{\rm{ST}}(\bsh) = \bopstar \big(\bsh \delta_{\Gamma}\big)$,
    and hence it vanishes on $M \smallsetminus \Gamma$.
  \end{enumerate}
\end{proposition}


\begin{proof}
  Since the space $\maN$ consists of smooth sections (by elliptic regularity), $p_{\maN}$
  is a regularizing pseudodifferential operator (i.e., one of order $-\infty$),
  as stated. Equation \eqref{eq.rel.pnu} then extends to distributions and gives
  \begin{equation*}
    \bsXi \maS_{\rm{ST}}(\bsh) \seq \bsXi \psdinv (( \bsh \delta_{\Gamma}\ \ 0)^\top)
    \seq ( \bsh \delta_{\Gamma}\ \ 0)^\top - p_{\maN}(( \bsh \delta_{\Gamma}\ \ 0)^\top)\,,
  \end{equation*}
  which is the first relation in (1). The proof for the double
  layer potential operators is completely similar. Indeed,
  \begin{equation*}
    \bsXi \maD_{\rm{ST}}(\bsh) \seq \bsXi \psdinv \bopstar (\bsh \delta_{\Gamma})
    \seq \bopstar (\bsh \delta_{\Gamma}) - p_{\maN}\bopstar (\bsh \delta_{\Gamma})\,.
  \end{equation*}
  This completes the proof of (1).

  To check the points (2) and (3), let us choose an orthonormal basis $\{\phi_{1},
  \ldots, \phi_{k}\}$ of $\maN$, which we recall was assumed to be finite dimensional. Then $p_{\maN}(v)
  = \sum_{j=1}^{k} (v, \phi_{k})\phi_{k}$. Moreover, we have $\bsXi \psdinv v = v$
  if, and only if $p_{\maN} v = 0$, by Equation \eqref{eq.rel.pnu}. We then compute
  (and we replace the inner product by the duality pairing of distributions) to obtain
  \begin{align*}
    p_{\maN}(v) \seq 0 & \ \Leftrightarrow \ \sum_{j=1}^{k} \langle v, \phi_{j} \rangle
    \phi_{j} \seq  0\\
    & \ \Leftrightarrow \  \forall\, j=1, \ldots, k\,, \quad  \langle v, \phi_{j} \rangle \seq 0\\
    & \ \Leftrightarrow \  \forall\, w \in \maN\,, \quad  \langle v, w \rangle \seq 0\,.
  \end{align*}
  Then, for $v := ( \bsh \delta_{\Gamma}\ \ 0)^\top$, this gives,
  \begin{align*}
    p_{\maN}{( \bsh \delta_{\Gamma}\ \ 0)^\top} \seq 0
    & \ \Leftrightarrow \  \forall\, w=(u,p) \in \maN\,, \quad
    \langle \bsh \delta_{\Gamma}, u \rangle \seq 0\\
    & \ \Leftrightarrow \  \forall\, w=(u,p) \in \maN\,, \quad
    \int_{\Gamma} \bsh \cdot \overline{u} \, dS_{\Gamma} \seq 0\,.
  \end{align*}
  (Note that $\maN$ is invariant under complex conjugation because our space
  $H^{s}(M; TM)$ is the complexification of its real variant.) This proves (2).

  To prove (3), we use the same approach, but for $v := \bopstar (\bsh \delta_{\Gamma})$
  this time, to obtain
  \begin{align*}
    p_{\maN}(\bopstar (\bsh \delta_{\Gamma})) \seq 0
    & \ \Leftrightarrow \  \forall\, w \in \maN\,, \quad
    \langle \bopstar (\bsh \delta_{\Gamma}), w \rangle \seq 0\\
    & \ \Leftrightarrow \  \forall\, w \in \maN\,, \quad
    \langle \bsh \delta_{\Gamma}, \bop w \rangle \seq 0\\
    & \ \Leftrightarrow \  \forall\, w \in \maN\,, \quad
    \int_{\Gamma} \bsh \cdot \overline{\bop w} \, dS_{\Gamma} \seq 0\,.
  \end{align*}
  This completes the proof.
\end{proof}

We shall need the following consequences of the representation formula
in Proposition \ref{prop.Green}, the first one of which we will call
\emph{Pompeiu's formula}.

\begin{proposition}\label{prop.rep.formula}
  Let $U = (\, \bsu \ \ \ p \,)^{\top } \in L^{2}(M; TM \oplus \CC)$, let $p_{\maN}$
  be the orthogonal projection onto the kernel $\maN$ of $\bsXi$, and let
  $\textbf{1}_{\Omega}$ be the characteristic function of $\Omega$.
  Then the following formulas hold.
  \begin{enumerate}[\rm (1)]
  \item
    $\textbf{1}_{\Omega} U
    \seq \psdinv \left(\textbf{1}_{\Omega}
    \big( \bsXi U  \big) \right) -
    \maS_{\rm{ST}} (\bop U\vert_{\Gamma})
    + \maD_{\rm{ST}}  (\bsu\vert_{\Gamma})  + p_{\maN}(\textbf{1}_{\Omega} U) \,.$
  \item If, moreover, $\bsXi U = 0$ in $\Omega $, then
    \begin{equation*}  
    \maD_{\rm{ST}} (\bsu\vert_{\Gamma})(x) - \maS_{\rm{ST}} (\bop U\vert_{\Gamma})(x)
    + p_{\maN}(\textbf{1}_{\Omega} U)(x)  \seq
    \begin{cases}
      \ U(x) & \mbox{ if } x \in \Omega\\
      \ \ \, 0    & \mbox{ if } x \in M \smallsetminus \overline{\Omega} \,.
    \end{cases}
    \end{equation*}
  \end{enumerate}
\end{proposition}

\begin{proof}
  Equation \eqref{eq.rel.pnu} and the last relation of Proposition \ref{prop.Green} give
  \begin{align}\label{eq.Pompeiu}
    \textbf{1}_{\Omega} U  - p_{\maN}(\textbf{1}_{\Omega} U)
    & \seq \psdinv \left(\textbf{1}_{\Omega} \big( \bsXi U  \big) -
    \left( \begin{array}{c} \bop U \\ 0 \end{array} \right) \delta_{\Gamma}
    + \bopstar \, (\bsu \delta_{\Gamma})\right)\\
    & \seq \psdinv \left(\textbf{1}_{\Omega}  \big( \bsXi U  \big) \right) -
    \maS_{\rm{ST}} (\bop U)
    + \maD_{\rm{ST}}  (\bsu)\,. \nonumber
  \end{align}
  The second part follows immediately from the last equation (Pompeiu's formula, Equation
  \ref{eq.Pompeiu})
  and the definitions of the single and double layer potentials,
  Definition \ref{def.lp}.
\end{proof}

\subsection{Jump relations}
\label{ssec.Jump}
The jump relations work in general (also for non-compact manifolds).
We will now establish some needed jump relations for the potential
operator $\maS_{\bsP} = \maW_{\rm{ST}}$ (Definition \ref{def.lp})
associated to the pseudodifferential operator $\bsP := - 2 \maA  \Dnu^{*} + \maB \nu^{\sharp}$
using the results of Section \ref{sec.Def}. This calculation is motivated by
Remark \ref{rem.components}. We follow the approach in \cite{KMNW-2025}, where
these results were proved for manifolds with cylindrical ends.

\begin{proposition}\label{prop.local.stokes}
  Let $\bsP := - 2 \maA  \Dnu^{*} + \maB \bsnu^{\sharp}$. We denote
  $f = (V_{0} + 1)/(2V_{0} + 1)$ and $g = 1/(2V_{0} + 1)$ (as before). Then
  \begin{equation*}
    \sigma_{-1}(\bsP;  \xi) \seq
    \frac{\imath}{|\xi|^2} \Big ( \xi(\bsnu)   + \bsnu \otimes \xi
    - \frac{2f \xi(\bsnu)}{| \xi|^{2}} \xi^{\sharp} \otimes \xi
    + g  \xi^{\sharp} \otimes \bsnu^{\sharp} \Big )\,.
  \end{equation*}
  Consequently, $\JC_{+} = \JC_{-} = -\imath$ for this operator.
\end{proposition}

\begin{proof}
  The calculations are local, so we may assume that $\Omega = \RR_{+}^{n}$.
  In particular, $\bsnu = - e_{n}$ and $\bsnu^{\sharp} = -e_{n}^{\sharp}$.
  We decompose $\xi = \xi' + t \bsnu$. Using the formulas of Proposition
  \ref{prop.formulas.new}\eqref{prop.fn.5} and Remark \ref{rem.form}, we obtain
  \begin{align*}
    \sigma_{-1}(\bsP ;  \xi) & \seq -2 \sigma_{-2}(\maA;  \xi)\sigma_{1}(\Dnu^{*};  \xi)
    + \sigma_{-1}(\maB;  \xi)\bsnu^{\sharp}\\
    & \seq -2\left(\frac1{| \xi|^{2}} - \frac{f}{| \xi|^{4}}  \xi^{\sharp} \otimes  \xi
    \right)
    \left(\frac{\imath}2 \right)(  t  + e_{n} \otimes  \xi)
    - \frac{\imath g}{| \xi|^{2}}  \xi^{\sharp} \otimes e_{n}^{\sharp}\\
    & \seq -\frac{\imath}{| \xi|^2} \big(   t  + e_{n} \otimes  \xi
    - \frac{2f   t }{| \xi|^{2}}  \xi^{\sharp} \otimes  \xi
    + g  \xi^{\sharp} \otimes e_{n}^{\sharp} \big)\,.
  \end{align*}
  The coefficients $\JC_{+}$ and $\JC_{-}$ are obtained by expanding the
  formula of the last equation according in terms of the highest powers of $t$,
  using $ \xi =  $,
  with $ \xi' = ( \xi_{1}, \ldots,  \xi_{n-1})$, to obtain
  \begin{align*}
    \lim_{t \to \pm \infty} t \sigma_{-1}(\bsP ;  \xi) &  \seq -
    \lim_{t \to \pm \infty}\frac{\imath t}{| \xi|^2} \big(t + e_{n} \otimes  \xi
    - \frac{2ft}{| \xi|^{2}} \xi^{\sharp} \otimes \xi
    + g \xi^{\sharp} \otimes \xi \big)\\
    &  \seq - \lim_{t \to \pm \infty}\frac{\imath t^{2}}{|\xi' + t e_{n}^{\sharp}|^2}
    \big(1 + e_{n} \otimes e_{n}^{\sharp}
    - \frac{2f t^{2}}{| \xi|^{2}} e_{n} \otimes e_{n}^{\sharp}
    + g e_{n} \otimes e_{n}^{\sharp} \big)\\
    &  \seq -
    \imath
    \big(1 + e_{n} \otimes e_{n}^{\sharp}
    - 2f  e_{n} \otimes e_{n}^{\sharp}
    + g e_{n} \otimes e_{n}^{\sharp} \big)\\
    &  \seq -
    \imath
    \big[1 + (1 - 2f + g) e_{n} \otimes e_{n}^{\sharp}
    \big] \seq - \imath \,,
  \end{align*}
  because $1 - 2f + g = 0$.
\end{proof}

We shall need the following calculation using residues (the details
can be found in Lemma 16.3.2 of \cite{KMNW-2025}).

\begin{lemma}\label{lemma.Mirela} Let $a > 0$. Then
  \begin{equation*}
  \displaystyle\int_{\RR}\frac{x^2 dx}{(a^{2} + x^{2})^{2}}
  \seq \displaystyle\frac{\pi}{2a}\,,\quad
  \displaystyle\int_{\RR}\frac{dx}{(a^{2} + x^{2})^{2}}
  \seq \displaystyle\frac{\pi}{2a^3}\,, \quad
   \mbox{and} \quad \displaystyle\int_{\RR}\frac{dx}{a^{2} + x^{2}}
  \seq \displaystyle\frac{\pi}{a}\,.
  \end{equation*}
\end{lemma}

For the rest of the paper, we let
$\bsP := - 2 \maA  \Dnu^{*} + \maB \bsnu^{\sharp}$
be the pseudodifferential operator defining the
vector part $\maW_{\rm{ST}}$ of the double layer potential (Definition
\ref{def.lp}). Theorem \ref{thm.main.jump1} then yields the
``restriction at $\Gamma$ operator''
\begin{equation}\label{eq.def.bsk}
  \bsK \ede \bsP _{0} \ede \big(- 2 \maA  \Dnu^{*} + \maB \bsnu^{\sharp})_{0}
\end{equation}
which is an order zero, classical pseudodifferential operator on $\Gamma :=
\pa \Omega$. Here is our first ``jump relation,'' which extends the classical
one on Euclidean spaces.

\begin{theorem}\label{thm.K1}
  Let $\bsK := \bsP_{0}$ be as in Equation \eqref{eq.def.bsk}. Then
  \begin{equation*}
    \maW_{\rm{ST}}(\bsh)_{\pm}
    \ede \big[ \bsP  (\bsh \delta_{\Gamma})\big]_{\pm}
    \seq \left [\pm \frac12   +
    \bsK \right ] \bsh\,,
  \end{equation*}
  where $\sigma_{0}(\bsK;  \xi') = \frac{\imath V_{0}}{2(2V_{0} + 1)| \xi'|}
  \big( \bsnu \otimes  \xi' - \xi^{'\sharp} \otimes \bsnu^{\sharp} \big )$.
  In particular, the two operators $\pm \frac12   + \bsK$
  are elliptic for $V_{0} \ge 0$ and have self-adjoint principal
  symbols.
\end{theorem}

\begin{proof}
  Let $f = (V_{0} + 1)/(2V_{0} + 1)$ and $g = 1/(2V_{0} + 1)$, as in
  Proposition \ref{prop.local.stokes}. As in that proposition, we
  use local coordinates such that $\bsnu = -e_{n}$. Using Theorem \ref{thm.main.jump1}
  and Proposition \ref{prop.local.stokes}, we see that it is enough to
  identify $\sigma_{0}(\bsK; \xi') := \sigma_{0}(\bsP_{0}; \xi')$.
  To that end, we separate the terms that are \emph{even} in $t$ in the expansion of
  $\sigma_{-1}(\bsP; \xi)$ in terms of powers of $t$. For instance, the even part of
  \begin{equation*}
    \xi^{\sharp} \otimes  \xi \seq \xi^{'\sharp} \otimes  \xi' + t (e_{n} \otimes \xi'
    + \xi^{'\sharp} \otimes e_{n}^{\sharp}) +  t ^{2} e_{n} \otimes e_{n}^{\sharp}
  \end{equation*}
  is $\xi^{'\sharp} \otimes  \xi' +  t ^{2} e_{n} \otimes e_{n}^{\sharp}$,
  whereas its odd part is $t (e_{n} \otimes \xi' + \xi^{'\sharp} \otimes e_{n}^{\sharp})$.
  This gives
  \begin{multline*}
    b(\xi', t) \ede
    \sigma_{-1}(\bsP ;  \xi) + \sigma_{-1}(\bsP ;  \xi', -  t ) \\
    \seq -\frac{2 \imath}{| \xi|^2} \Big[ e_{n} \otimes  \xi' +
    g \xi^{'\sharp} \otimes e_{n}^{\sharp} -
    \frac{2f   t ^{2}}{| \xi|^{2}}(e_{n} \otimes  \xi'
    +  \xi^{'\sharp} \otimes e_{n}^{\sharp}) \Big]\\
    \seq -2 \imath \Big[ \left( \frac1{| \xi|^2} - \frac{2f
     t ^{2}}{| \xi|^{4}} \right)
    e_{n} \otimes  \xi' +
    \left( \frac{g}{| \xi|^2} - \frac{2f   t ^{2}}{| \xi|^{4}} \right)
    \xi^{'\sharp} \otimes e_{n}^{\sharp} \Big]
    \,.
  \end{multline*}
  Lemma \ref{lemma.Mirela} gives
  \begin{equation*}
    \int_{\RR} \frac{1}{| \xi|^2}\, d  t  \seq \frac{\pi}{| \xi'|}
    \ \mbox{ and }\
    \int_{\RR} \frac{  t ^{2}}{| \xi|^4}\, d  t  \seq \frac{\pi}{2| \xi'|}
    \,,\quad  \xi' \neq 0\,,
  \end{equation*}
  We next use Theorem \ref{thm.main.jump1} and these relations to obtain
  \begin{multline*}
    \sigma_{0}(\bsK;  \xi') \seq \frac1{4 \pi}
    \int_{\RR} b( \xi',   t )\, d  t \\
    \seq - \frac{\imath}{2\pi} \int_{\RR}
    \Big[ \left( \frac1{| \xi|^2} - \frac{2f   t ^{2}}{| \xi|^{4}} \right)
    e_{n} \otimes  \xi' +
    \left( \frac{g}{| \xi|^2} - \frac{2f   t ^{2}}{| \xi|^{4}} \right)
     \xi^{'\sharp} \otimes e_{n}^{\sharp} \Big]\, d  t \\
    \seq \frac{\imath V_{0}}{2(2V_{0} + 1)| \xi'|}
    \big( \xi^{'\sharp} \otimes e_{n}^{\sharp} - e_{n} \otimes  \xi'\big )\,.
  \end{multline*}
  This explicit formula gives that $\sigma_{0}(\bsK)^{*} = \sigma_{0}(\bsK)$.
  An elementary calculation gives that the eigenvalues of $\sigma_{0}(\bsK;  \xi')$
  are $\lambda = \pm \frac{V_{0}}{2(2V_{0} + 1)}$. Since they satisfy $|\lambda|<1/4$,
  we obtain that $\pm \frac 12   + \bsK$ is elliptic.
\end{proof}

\begin{remark}
  Theorem \ref{thm.K1} gives right away that $\bsK = \bsP _{0}$ is a pseudodifferential
  operator of order $-1$ if, and only if, $V_{0} = 0$.
\end{remark}

Using Theorem \ref{thm.main.jump1}, let us define $\bsS := \maA_{0}$ and $\maC_{0}$
to be the ``restriction at $\Gamma$ operators'' associated to the pseudodifferential
operators $\maA$ and $\maC$ of Proposition \ref{prop.form.inverse}
(as two of the matrix components of $\psdinv$). These are the operators appearing
in the definition of the single layer potential $\maS_{\rm{ST}}$ (see
Remark \ref{rem.components}). The notationi $\bsS := \maA_{0}$ is the
customary one in the theory of layer potentials. Recall that
$f := \frac{V_{0}+1}{2V_{0}+1}$ and $g = \frac1{2V_{0} + 1}$.
We obtain the following relations, also called ``jump relations''.

\begin{theorem}\label{thm.jump.rel}
  Let $\bsh \in L^{2}(\Gamma; TM)$, where $\Gamma := \pa \Omega$,
  as before.
  \begin{enumerate}[\rm (1)]
    \item $\maV_{\rm{ST}}(\bsh)_{\pm} \seq \bsS \bsh := \maA_{0} \bsh$ and
      $\sigma_{-1}(\bsS; \xi') \seq
      \frac1{4|\xi'|}(2 - f \bsnu \otimes \bsnu^{\sharp} - f \eta^{\sharp} \otimes \eta)
      \,,$
    where $\eta := |\xi'|^{-1}\xi'$.
    Consequently, $\bsS$ is elliptic with self-adjoint symbol.

    \item
      $\big[\maP_{\rm{ST}}(\bsh)\big]_{\pm} = \big(\mp \frac g2
      \bsnu^{\sharp} + \maC_{0} \big) \bsh$,
      where $\sigma_{0}(\maC_{0};\xi') = -\frac{g \imath}{2|\xi'|} \xi'.$

    \item $[\bop \maS_{\rm{ST}}(\bsh)]_{\pm} \seq \left(\mp \frac12
      + {\bsK}^{*}\right)\bsh\,,$ where $\bsK = \bsP_{0}
      := (-2 \maA \Dnu^{*} + \maB \bsnu^{\sharp})$, as in Theorem \ref{thm.K1}.
  \end{enumerate}
\end{theorem}

\begin{proof}
  Recall that the linear map
  $\xi^{\sharp} \otimes \xi \in \End(T_{x}M)$ is
  defined by $(\xi^{\sharp} \otimes \xi) (v) := \xi(v) \xi^{\sharp}$.
  This gives, $\sigma_{-2}(\maA, \xi) = \frac1{|\xi|^{4}}
  \big(|\xi|^{2} - f \xi^{\sharp} \otimes \xi\big)$.
  For $\xi \in T^{*}M$, let us write, as before, $\xi = \xi' + t \bsnu^{\sharp}$,
  with $\xi'(\bsnu) = 0$ and we use the projection $T_{x}^{*}M \to T_{x}^{*}\Gamma$,
  when $x \in \Gamma$. To
  prove the first equality, we use Proposition \ref{thm.main.jump1b}
  (see also Theorem \ref{thm.main.jump-2}) and then Proposition \ref{prop.form.inverse}
  (see also Remark \ref{rem.form}) to obtain
  \begin{align*}
    \sigma_{-1}(\bsS; \xi') & \seq \frac1{2\pi} \int_{\RR}
    \sigma_{-2}(\maA; \xi) \, dt\\
    & \seq \frac1{2\pi} \int_{\RR} \frac1{|\xi|^{4}}
    \big(|\xi|^{2} - f \xi^{\sharp} \otimes \xi\big) \, dt\\
    & \seq \frac1{2\pi} \int_{\RR} \frac1{|\xi|^{4}}
    \big[|\xi|^{2} - f (\xi^{\prime \sharp} \otimes \xi' + t \bsnu \otimes \xi' +
    t \xi^{\prime \sharp} \otimes \bsnu^{\sharp}
    + t^{2} \bsnu \otimes \bsnu^{\sharp}) \big] \, dt\\
    & \seq \frac1{2\pi} \int_{\RR} \frac1{|\xi|^{4}}
    \big[|\xi|^{2} - f (\xi^{\prime \sharp} \otimes \xi' +
    t^{2} \bsnu \otimes \bsnu^{\sharp}) \big] \, dt\\
    & \seq \frac1{2\pi} \int_{\RR} \left(\frac1{|\xi'|^{2} + t^{2}}
    - \frac f{(|\xi'|^{2} + t^{2})^{2}} \xi^{\prime \sharp} \otimes \xi'
    - \frac {ft^{2}}{(|\xi'|^{2} + t^{2})^{2}}
     \bsnu \otimes \bsnu^{\sharp}\right) \, dt\\
    & \seq \frac{1}{2|\xi'|}
    - \frac {f }{4|\xi'|^{3}} \xi^{\prime \sharp} \otimes \xi'
    - \frac {f}{4|\xi'|}
    \bsnu \otimes \bsnu^{\sharp}\,.
  \end{align*}
  This proves (1).

  For the second relation, we use the relation
  $\sigma_{-1}(\maC; \xi) = -\frac{g \imath}{|\xi|^{2}} \xi$ (see
  Proposition \ref{prop.form.inverse}) and then Theorem \ref{thm.main.jump1}
  (see also Theorem \ref{thm.main.jump0}). We also write $\xi = \xi' + t \bsnu^{\sharp}$
  with $\xi' \perp \bsnu$, as in the proof of the previous point.
  Then we notice that the even part of $\sigma_{-1}(\maC; \xi)$
  (in $\tau$) is $-\frac{g \imath}{|\xi|^{2}} \xi'$. Therefore
  \begin{equation*}
    \sigma_{0}(\maC_{0}; \xi') \seq -\frac1{2\pi}
    \int_{\RR}\frac{g \imath}{|\xi|^{2}} \xi' \, d t
    \seq -\frac{g \imath}{2\pi}
    \left(\int_{\RR}\frac{1}{|\xi|^{2}} \, d t \right) \xi'
    \seq -\frac{g \imath}{2 |\xi'|} \xi'\,.
  \end{equation*}
  The ``jump part'' is also obtained from the principal symbol of
  $\maC$, namely, it is $\mp \frac{\imath}2 \sigma_{-1}(\maC; \bsnu^{\sharp})
  = \mp \frac{g}{2}\bsnu^{\sharp}$.

  Let us now prove the third relation. We have $\bsXi^{*} = \bsXi$, and hence
  $\bsXi^{(-1)*} = \bsXi^{(-1)}$. Theorem \ref{thm.main.jump1} and Proposition
  \ref{prop.local.stokes} then give $\JC_{+} (\bop\bsXi^{(-1)}) =
  \JC_{+} (\bsXi^{(-1)}\bopstar)^{*} = \overline{(-\imath)} = \imath$.
  Moreover, $(\bop\bsXi^{(-1)})_{0} = (\bsXi^{(-1)}\bopstar)_{0}^{*} = \bsK^{*}$.
  This then gives the following relation:
  \begin{multline*}
    [\bop \maS_{\rm{ST}}(\bsh)]_{\pm}
    \seq [\bop \bsXi^{(-1)}(\bsh \delta_{\Gamma})]_{\pm}
    \seq \left( \pm \frac{\imath}2 \JC_{+} (\bop \bsXi^{(-1)})
      + (\bop \bsXi^{(-1)})_{0} \right) \bsh \\
     \seq \left( \pm \frac{\imath}2 \JC_{+} (\bsXi^{(-1)}\bopstar)^{*}
      + (\bsXi^{(-1)}\bopstar\big)_{0}^{*} \right) \bsh
    \seq \left( \mp \frac12 + \bsK^{*} \right) \bsh\,.
  \end{multline*}
  This completes the proof.
\end{proof}

For the (usual) Stokes operator $\boldsymbol\Xi_{0,0}$,
some of the ``jump relations'' proved in this section
can be found in \cite{D-M, M-W}, \cite[Lemma 3.1]{K-L-W},
\cite[(6.1), (6.2)]{K-W}, or \cite[Lemma 1.3]{Varnhorn}.
Here is an immediate consequence of the ellipticity relations proved
above.

\begin{theorem} \label{thm.Fredholm_first}
  Let $\Omega \subset M$ be a domain with smooth boundary $\Gamma := \pa \Omega
  = \pa \Omega_{-} \neq \emptyset$. We assume that $\Gamma$ is compact.
  Let $\bsK \ede \bsP _{0} \ede
  \big(- 2 \maA  \Dnu^{*} + \maB \bsnu^{\sharp})_{0} \in
  \Psi^{0}(\Gamma ; TM)$ be as in Theorem \ref{thm.K1}.
  Also, let $\bsS := \maA_{0} \in \Psi^{-1}(\Gamma ; TM)$ be as in Theorem \ref{thm.jump.rel}.
  Then $\frac12 + \bsK$ and $\bsS$ are Fredholm of index zero on $L^2(\Gamma; TM)$
  and their Moore-Penrose pseudo-inverses satisfy
  $(\frac12 + \bsK)^{(-1)} \in \Psi_{\cl}^{0}(\Gamma; TM)$ and
  $\bsS^{(-1)} \in \Psi_{\cl}^{1}(\Gamma; TM)$. Moreover,
  $\bsS$ is bounded, self-adjoint on $L^{2}(\Gamma; TM)$.
\end{theorem}

\begin{proof}
  We know from Theorem \ref{thm.K1}
  that $\frac12 + \bsK$ is elliptic with self-adjoint principal
  symbol. Because $\Gamma$ is compact, $\frac12 + \bsK$ is then Fredholm of index zero
  by classical results \cite{Hormander3, H-W}. (Indeed, the operator
  $T :=\left(\frac12   + \bsK\right)-\left(\frac12   + \bsK^*\right)$
  belongs to $\Psi ^{-1}(\Gamma ;TM)$, and hence is compact on the space $L^2(\Gamma ;TM)$.
  Thus the operator $\frac12   + \bsK = \left(\frac12   + \bsK^*\right) + T$ has
  the same index as the operator $\frac12   + \bsK^*$, because $T$ is compact.
  However, the index of $\frac12   + \bsK $ is the opposite index of
  $\frac12   + \bsK^*$, by definition.
  Consequently, $\frac12   + \bsK$ is a Fredholm operator of index zero
  on $L^2(\Gamma ;TM)$, as asserted.)
  The fact that $(\frac12 + \bsK)^{(-1)} \in \Psi_{\cl}^{0}(\Gamma; TM)$
  is also a classical result on pseudodifferential operators
  \cite{BealsSpInv, Hormander3} (a proof can also be found in \cite{KMNW-2025, KNW-22}).
  This proves our result for $\frac12 + \bsK$. The proof for $\bsS$ is completely
  similar, but using Theorem \ref{thm.jump.rel}(1), which states that $\bsS$
  has a self-adjoint principal symbol, instead of Theorem \ref{thm.K1}.

  Finally, since $\maA$ is self-adjoint and the distribution kernel $k_{\bsS}$
  of $\bsS$ is the restriction of the distribution kernel $k_{\maA}$
  of $\maA$ (Proposition \ref{thm.main.jump1b}), we obtain that $\bsS$ is
  also self-adjoint.
\end{proof}

\section{Fredholmness and invertibility of layer potential operators}
\label{sec.sec9}

We now derive consequences on the kernel, image and the Fredholm property of our
generalized Stokes operators $\bsXi$ in the case $M$ closed. More precisely, we assume
throughout this section that $M$ is a smooth, compact, boundaryless manifold
(i.e., a {\it closed manifold}), that $M$ is connected and that
$V, V_{0} \in \CI(M)$ are non-negative.

\subsection{Fredholmness of the generalized Stokes operator $\bsXi_{V, V_{0}}$}
The following result relies heavily on the results and  methods of
\cite{KNW-22} and \cite{KohrNistor-Stokes}.
Let $\phi : A \to \CC$. We shall write $\phi \not \equiv 0$ on $A$ if $\phi$
\emph{does not vanish identically} on $A$, that is, there exists $a \in A$ such that
$\phi(a) \neq 0$. By contrast, when we write ``$\phi \neq 0$ on $A$,'' we mean that
``$\phi(a) \neq 0$ \emph{for all} $a \in A$.'' Similarly, when we write
``$\phi = 0$ on $A$,'' we mean that ``$\phi(a) = 0$ \emph{for all} $a \in A$,''
(that is, the negation of the statement ``$\phi \not \equiv 0$'').

Let us assume that $\bsXi$ is a Fredholm operator and let $\maN$ denote its
kernel. Let
\begin{equation*}
  \widetilde{\bsXi} : \big (H^{2}(M; TM) \oplus H^{1}(M) \big) \cap \maN^{\perp}
  \to \big(L^{2}(M; TM) \oplus H^{1}(M)\big) \cap \maN^{\perp}
\end{equation*}
be the induced operator (here the orthogonal is in distribution sense and
the operator is well-defined since $\bsXi$ is symmetric, so, if $\xi \in
H^{2}(M; TM) \oplus H^{1}(M)$ and $\eta \in \maN$,
then $(\bsXi \xi, \eta) = (\xi , \bsXi \eta) = 0$). Consequently,
$\widetilde{\bsXi}$ is invertible. We then extend its inverse to an operator
$H^{-1}(M; TM) \oplus L^{2}(M) \to H^{1}(M; TM) \oplus L^{2}(M)$, denoted
$\psdinv$, the {\it Moore-Penrose pseudo-inverse} of $\bsXi$.
Indeed, let $p_{\maN}$ be the $L^{2}(M)$-orthogonal projection onto $\maN$. Then
the Equation \eqref{eq.rel.pnu} is still satisfied:
\begin{equation*}
  \begin{gathered}
    \psdinv (1 - p_{\maN}) \seq (1 - p_{\maN}) \psdinv \seq \psdinv \ \mbox{ and }\\
    \bsXi \psdinv \seq \psdinv \bsXi =  1 - p_{\maN}\,.
  \end{gathered}
\end{equation*}
Recall that we are assuming $M$ to be connected.

\begin{theorem}\label{thm.form}
  Let us assume that $V, V_{0}$ are smooth and non-negative
  and that $M$ is a smooth, compact manifold without boundary
  (i.e., a closed manifold).
  Then $\bsXi := \bsXi_{V, V_{0}}: H^{1}(M; TM) \oplus L^{2}(M) \to H^{-1}(M; TM) \oplus L^{2}(M)$
  is a self-adjoint Fredholm operator.
  Let $\maN \subset \CI(M; TM \oplus \CC)$ be defined by
  \begin{enumerate}[\rm (1)]
    \item $\maN := \{(\bsu, p) \mid \Def \bsu = 0 \,, \ \nabla p = 0\}$
    if $V = 0$ and $V_{0}=0$ on $M$;
    \item $\maN := \{(\bsu, 0) \mid \Def \bsu = 0\}$,
    if $V =0$ and $V_{0} \not \equiv 0$  on $M$;
    \item $\maN := \{(0, p) \mid \nabla p = 0\}$,
    if $V_{0}= 0$ on $M$ and either $V \not \equiv 0$ on $M$ or $M$ does not have
    non-zero Killing vector fields;
    \item $\maN := \{0\}$, if $V_{0} \not \equiv 0$ on $M$ and either $V \not \equiv 0$
    on $M$ or $M$ does not have non-zero Killing vector fields.
  \end{enumerate}
  The kernel of $\bsXi_{V, V_{0}}$ is given by
  $\ker \bsXi_{V, V_{0}} = \maN$. Moreover, $\bsXi$ has a (unique) Moore-Penrose
  pseudoinverse $\psdinv \in \Psi_{\cl}^{-\bss - \bst}(M; TM \oplus \CC).$
\end{theorem}

Recall that the condition $\Def \bsu = 0$ means that $\bsu$ is a
\emph{Killing vector field} (i.e. it preserves the metric). For a generic
manifold $M$, this space is reduced to 0. The condition $\nabla p = 0$
simply means that $p$ is locally constant (thus this part of the kernel
is at most one-dimensional if $M$ is connected).

\begin{proof}
  The generalized Stokes operator $\bsXi := \bsXi_{V, V_{0}}$
  of Equation \eqref{eq.def.bsXi} is $(\bss , \bst )$
  Douglis-Nirenberg elliptic, by Proposition \ref{prop.ADN.elliptic}.
  It follows that $\bsXi$ is Fredholm as an operator
  \begin{equation*}
    \bsXi : H^{1}(M; TM) \oplus L^{2}(M) \to H^{-1}(M; TM) \oplus L^{2}(M)\,.
  \end{equation*}
  This is a consequence of the usual properties of pseudodifferential operators
  on compact manifolds. See, for example, Theorem 15.4.17 in \cite{KMNW-2025}.
  (See also \cite{KNW-22}).
  Since $\bsXi$ is formally self-adjoint and elliptic (in \ADN-sense), it is (trully)
  self-adjoint (this follows using elliptic regularity, see,
  for instance, \cite{Grosse-Kohr-Nistor-23, KNW-22} or Theorem 15.3.30 in \cite{KMNW-2025}).
  Therefore $\bsXi$ is of index zero. It remains to determine its
  kernel.

  Let $U = \cvector{\bsu}{p} = (\, \bsu\ \ \ p\,)^{\top} \in H^{1}(M; TM)
  \oplus L^{2}(M)$ be such that $\bsXi U = 0$. Then $U \in H^{2}(M; TM)
  \oplus H^{1}(M)$, by elliptic regularity, as above. Proposition \ref{prop.Green}(1)
  (for $\Omega = M$, which means that the inner products on the boundary are dropped)
  then gives that
  \begin{equation*}
    \Def \bsu \seq 0 \,, \quad V\bsu \seq 0\,, \quad \mbox{and}\quad
    V_{0}p \seq 0\,.
  \end{equation*}
  Then, the equation $\bsXi U = 0$ implies
  \begin{equation*}
    0 \seq 2 \Defstar\Def \bsu + V \bsu + \nabla p \seq \nabla p\,.
  \end{equation*}
  The relations proved and the fact that both $\Def$ and $\nabla$
  satisfy the $L^{2}$-unique continuation property gives that $\ker \bsXi_{V, V_{0}}
  \subset \maN$.

  The opposite inclusion $\maN \subset \ker \bsXi $ follows from the definition
  (taking also into account the fact that, if $\Def u = 0$, then its
  divergence $\nabla^{*}u = 0$). This gives the desired equality
  $\ker \bsXi = \maN$. Finally, the proof that
  $\psdinv \in \Psi_{\cl}^{-\bss - \bst}(M; TM \oplus \CC)$ is a classical
  result of Beals \cite{BealsSpInv} (see \cite{KMNW-2025} for a detailed
  proof). Because $\bsXi $ is a Fredholm operator, self-adjoint and with
  trivial kernel in $\maN^\perp$, it follows that $\bsXi :\maN^\perp\to \maN^\perp$
  is invertible, and this proves the existence of a unique Moore-Penrose pseudoinverse
  operator $\bsXi ^{(-1)}$ of $\bsXi :H^{1}(M; TM) \oplus L^{2}(M) \to H^{-1}(M; TM)
  \oplus L^{2}(M)$.
\end{proof}

The proof of \cite[Theorem 5.11]{KohrNistor-Stokes} also gives the last
point of the last theorem (the invertibility of $\bsXi$). \emph{For this
reason, we make from now on the following assumption.}

\begin{assumption}\label{assumpt.V}
  We assume that
  \begin{enumerate}[\rm (i)]
    \item $M$ is a smooth, compact, connected, boundaryless manifold;
    \item $V , V_{0} : M \to [0, \infty)$ are smooth;
    \item either $M$ does not have non-zero Killing vector fields or
    $V$ does not vanish identically on $M$ (i.e. $V \not \equiv 0$ on $M$).
  \end{enumerate}
\end{assumption}

This assumption will soon be replaced by the stronger assumption \ref{assumpt.VV0}.
Note that (i) and (ii) of \ref{assumpt.V} above had already been in place.
The assumption that $M$ is connected is just to
simplify some of our statements. The general case follows from this one easily.

Since $\bsXi$ is a self-adjoint Fredholm operator, its range will be the
orthogonal of its kernel (this is well known, it follows, for instance,
also from our Lemma \ref{lemma.aux.v}).
It follows from Assumption \ref{assumpt.V} and Theorem
\ref{thm.form} that the kernel $\maN$ of $\bsXi$ is contained in the space
$\{(0, c) \mid c \in \CC \}$ of {\it constant} scalar fields. Let us assume
that they are equal and make then explicit the compatibility conditions of
Proposition \ref{prop.compatibility}.

\begin{corollary}\label{cor.compatibility}
  Let us assume that the kernel $\maN$ of $\bsXi$ is $\maN = \{(0, c) \mid c \in \CC \}$,
  the space of {\it constant} scalar fields. Let $\bsh \in L^{2}(\Gamma; TM)$.
  \begin{enumerate}[\rm (1)]
    \item We have $\bsXi \maS_{\rm{ST}}(\bsh) = (\bsh \delta_{\Gamma} \ \ 0)^{\top}$, and
    hence $\bsXi \maS_{\rm{ST}}(\bsh)$ vanishes outside $\Gamma$.

    \item On the other hand,
    $\bsXi \maD_{\rm{ST}}(\bsh) = \bopstar (\bsh \delta_{\Gamma})
    - (0, \frac1{\operatorname{vol}(M)}\, \int_{\Gamma}\bsh \cdot \bsnu\, dS_{\Gamma}).$
  \end{enumerate}
\end{corollary}

\begin{proof}
  Let $\chi := \frac1{\operatorname{vol}(M)^{1/2}}\textbf{1}_{M}$, the constant function on $M$
  with $L^{2}$-norm 1. We identify $\chi$ with its image $(0, \chi) \in \maN
  \subset H^{2}(M; TM) \oplus H^{1}(M)$.
  (In general, we identify $\CI(M)$ with its image in $L^{2}(M; TM \oplus \CC)$.)
  The formula for the projection $p_{\maN}$ is
  \begin{equation*}
    p_{\maN}(\bsu, p) \seq \langle (\bsu, p), \chi  \rangle \chi \seq
    \frac1{\operatorname{vol}(M)} \int_{M} p \dvol\,
  \end{equation*}
  a formula that extends then to distributions in an obvious way, by
  replacing the integral over $M$ with the pairing with distributions.

  For the case of the single layer potential, we obtain
  \begin{equation*}
    p_{\maN} (\bsh \delta_{\Gamma}) \seq \langle (\bsh \delta_{\Gamma}, 0),
    (0, \chi)  \rangle \chi\, \seq 0\,.
  \end{equation*}
  Proposition \ref{prop.compatibility}(2)
  then gives $\bsXi \maS_{\rm{ST}}(\bsh) = \bsh \delta_{\Gamma}$, which
  obviously vanishes outside $\Gamma$.

  On the other hand, for the double layer potential, because
  $\bop (0, p) = p \bsnu$ (see Equation \eqref{def.conormal.Tderiv}), we obtain
  \begin{multline*}
    p_{\maN}\big( \bopstar (\bsh \delta_{\Gamma}) \big)
    \seq \langle \bopstar(\bsh \delta_{\Gamma}), (0, \chi)  \rangle \chi
    \seq \langle \bsh \delta_{\Gamma}, \bop  (0, \chi)  \rangle \chi\\
    \seq \langle \bsh \delta_{\Gamma}, \chi \bsnu  \rangle \chi
    \seq \frac1{\operatorname{vol}(M)}\, \int_{\Gamma} \bsh \cdot \bsnu \, dS_{\Gamma}
  \end{multline*}
  Proposition \ref{prop.compatibility}(1) then yields the desired formula.
\end{proof}

\subsection{Invertibility of layer potentials}
We now prove one of the main results of this paper on layer potential operators. In the
following theorem, we can split our generalized Stokes boundary value problem as a direct
sum according to the connected components of $\Omega$, so there is no loss of generality
to assume that $\Omega$ is connected. In other words, we assume that
$\Omega$ is a {\it domain.} (We have assumed that $M$ is
connected for the same reason.) The general case follows immediately from this
particular case.
Recall the basic operators $\bsP := - 2 \maA  \Dnu^{*} + \maB \bsnu^{\sharp}$ and
$\bsK \ede \bsP_{0}$. The first one appears in the definition of the double layer
potential operator $\maW_{\rm{ST}}$ and the second one was
introduced in Equation \eqref{eq.def.bsk} and studied in Theorem \ref{thm.K1}.
(The correspondence $\bsP  \mapsto \bsP _{0}$ is the basic correspondence studied,
for example, in Theorems \ref{thm.main.jump0} and \ref{thm.main.jump1}.)
Recall that by the statement ``$\phi \not \equiv 0$ on $A$'' we mean that
there exists $a$ in the domain of $\phi$ such that $\phi(a) \not = 0$. To
negate this statement, we shall simply write ``$\phi = 0$ in $A$.''

We shall need the following simple lemma.

\begin{lemma}\label{lemma.aux.v}

  Let $\maH$ be a Hilbert space, $V \subset \maH$ be a closed subspace
  and $T : \maH \to \maH$ be a bounded operator. Then the following result holds:
  \begin{equation*} 
    \Big(T(V^{\perp})\Big)^{\perp}
    = \{ \bsh  \in \maH \mid T \bsh \in V\}\,.
  \end{equation*}
  In particular, for $V =\{0\}$, $\Big(T \maH\Big)^{\perp}
  = \ker T^{*}$.
\end{lemma}

\begin{proof}
  Let $(\cdot , \cdot )$ be the inner product of the space $\maH$. We have
  \begin{align*}
    v \in \Big(T(V^{\perp})\Big)^{\perp}
    & \ \Leftrightarrow \
    0  \seq (T\eta, v)\,, \quad \forall \eta \in V^{\perp}\\
    & \ \Leftrightarrow\
    0  \seq (\eta, T^{*}v)\,, \quad \forall \eta \in V^{\perp}\\
    & \ \Leftrightarrow\
    T^{*}v \in (V^{\perp})^{\perp} \seq V\,,
  \end{align*}
  where, for the last equality, we have used the assumption that $V$ is
  closed. The last statement is obtained by taking $V = \{0\}$.
\end{proof}

For some of our results, we shall need the following stronger assumption
(including Assumption \ref{assumpt.V}).

\begin{assumption}\label{assumpt.VV0}
  We assume that
  \begin{enumerate}[\rm (i)]
    \item $M$ is a smooth, compact, connected, boundaryless manifold;
    \item $V , V_{0} : M \to [0, \infty)$ are smooth;
    \item $\Omega_{-} := M \smallsetminus \overline{\Omega}$ has the same
    boundary $\Gamma \neq \emptyset$ as $\Omega$;
    \item Either $V \not \equiv 0$ on $M$ or $M$ does not have non-zero Killing vector fields;
    \item Given a connected component $\Omega_{0}$ of
    $\Omega_{-} := M \smallsetminus \overline{\Omega}$, either $\Omega_{0}$ does not have non-zero
    Killing vector fields or $V \not\equiv 0$ on $\Omega_{0}$;
    \item $V_{0} \not\equiv 0$ on every connected component $\Omega_{0}$ of $\Omega_{-}$.
  \end{enumerate}
\end{assumption}

(Again, some of the points above had been already assumed. Also, notice
that our assumptions on $\Omega$ imply that $\Omega_{-}\neq \emptyset$.)
We are ready now to prove the following theorem.

\begin{theorem}\label{thm.K2}
  Let Assumption \ref{assumpt.VV0} hold.
  Let $\bsK \ede \big(- 2 \maA  \Dnu^{*} + \maB \bsnu^{\sharp})_{0}$
  be as in Theorem \ref{thm.K1}.
  \begin{enumerate}[\rm (1)]
    \item If $V_{0} = 0$ on $\Omega$, then $\big(\frac12   + \bsK \big)L^{2}(\Gamma; TM)
    \subset \{\bsnu\}^{\perp}$.

    \item If $V_{0} = 0$ on $\Omega$, we have an isomorphism $\frac12   + \bsK :
    \{\bsnu\}^{\perp} \to \{\bsnu\}^{\perp}$.

    \item If $V_{0} \not \equiv 0$ on $\Omega$ (that $V_{0} \not \equiv 0$ on all connected
    components of $M \smallsetminus \Gamma$), then $\frac12   + \bsK$ is
    invertible on $L^2(\Gamma; TM)$.
  \end{enumerate}
\end{theorem}

\begin{proof}
  Recall that $\frac12 + \bsK$ is Fredholm on $L^{2}(\Gamma; TM)$ by Theorem
  \ref{thm.Fredholm_first}, because $\Gamma := \pa \Omega$ is compact.
  Also, recall that $\maN = 0$, where $\maN$ is the kernel of $\bsXi := \bsXi_{V, V_{0}}$,
  as usual. It will be convenient to split our proof into {\it four} steps.
  \smallskip

  \noindent {\bf Step 1 (Proof of (1): Necessary condition on the image).}\
  Let $\bsh \in H^{3/2}(\Gamma; TM)$. We begin by considering the
  \emph{double layer potential} $U := \maD_{\rm{ST}}(\bsh)$ of $\bsh$, which satisfies
  \begin{equation*}
    U \ede \maD_{\rm{ST}}(\bsh)
    \in H^{2}(\Omega; TM) \oplus H^{1}(\Omega)\,,
  \end{equation*}
  by Proposition \ref{prop.Hkestimates}. Let us write $U \seq (\, \bsu\ \ \ p \,)^{\top}$.
  The assumptions on $V$ and $V_{0}$ and Theorem \ref{thm.form} imply that the kernel
  $\maN$ of $\bsXi$ vanishes. Proposition \ref{prop.compatibility} implies then
  that $\bsXi U = 0$ in $M\setminus \Gamma $.

  Let $W = (\, \bsw \ \ \ q \,) := (\, 0 \ \ \ 1 \,)$. The assumption $V_{0} = 0$ on
  $\Omega$ gives $\bsXi W = 0$ on $\Omega$. Hence $(\bsXi U, W)_{\Omega} = 0$ and
  $(U, \bsXi W)_{\Omega} = 0$. Proposition \ref{prop.Green}(2) on $\Omega$
  gives
  \begin{equation*}
    (\bop  U, \bsw)_{\Gamma}  - (\bsu , \bop  W)_{\Gamma}  =
    \big ( \bsXi U, W \big )_{\Omega} - \big (U, \bsXi W \big )_{\Omega}
    \seq 0 \,.
  \end{equation*}
  Therefore, using that $\bsw = 0$, $\bop W = - 2\Dnu\bsw + q \bsnu = \bsnu$, and
  $\bsu_{+} = \big( \frac12  + \bsK\big ) \bsh$ (by Theorem \ref{thm.K1}), we obtain
  \begin{multline}\label{eq.aux.wc}
    0
    \seq (\bop  U, \bsw)_{\Gamma}  - (\bsu , \bop  W)_{\Gamma}
    \seq 0 -\int_{\Gamma} \bsu_{+} \cdot \bsnu \,dS_{\Gamma} \seq
    - \Big ( \Big(\frac12 + \bsK\Big) \bsh, \bsnu \Big )_{\Gamma} \,,
  \end{multline}
  where the inner product in the last expression is the $L^{2}$--inner product on $\Gamma$.
  This shows that
  \begin{equation*}
    \Big(\frac12 + \bsK\Big) L^{2}(\Gamma; TM) \subset \{\bsnu\}^{\perp}\,
  \end{equation*} by the density of the space $H^{3/2}(\Gamma; TM)$
  in $L^{2}(\Gamma; TM)$. This proves (1).
  \smallskip

  \noindent {\bf Step 2 (Reformulation of (2) using Lemma \ref{lemma.aux.v}).}\
  We now turn to the proof of (2).
  Using (1) we see that, in order to prove (2), it is hence enough to prove that $\{\bsnu\}^{\perp}
  \subset \big(\frac12 + \bsK \big) \{\bsnu\}^{\perp}$. To that end, we will use the adjoint of
  $\frac12 + \bsK$ and Lemma \ref{lemma.aux.v}. We claim now that, in fact, in order to complete
  the proof of the desired opposite inclusion (and hence to complete the proof of (2)), it suffices
  to show that
  \begin{equation}\label{eq.aux.incl}
    \{ \bsh \in L^{2}(\Gamma; TM) \mid \Big(\frac12   + \bsK^{*} \Big)\bsh
    \in \CC \bsnu\} \subset \CC\bsnu\,.
  \end{equation}
  Indeed, the last relation will give
  \begin{equation*}
    \{\bsnu\}^{\perp} \subset \{ \bsh \mid \Big(\frac12 + \bsK^{*} \Big)\bsh
    \in \CC \bsnu\}^{\perp} = \Big (\frac12 + \bsK \Big )\{\bsnu\}^{\perp}\,,
  \end{equation*}
  where we have used Lemma \ref{lemma.aux.v} for $V = \CC \bsnu$.

  \noindent {\bf Step 3 (Proof of Equation \eqref{eq.aux.incl}).}
  Let us prove Equation \eqref{eq.aux.incl}, which will complete the proof of (2),
  as noticed above. To this end, let in this step $\bsh \in L^{2}(\Gamma; TM)$ be such that
  \begin{equation}\label{eq.assumpt.h-first}
    \Big(\frac12 + \bsK^{*}\Big)\bsh \seq \lambda \bsnu \,,
  \end{equation}
  for some $\lambda \in \CC$. Since $\frac12   + \bsK^{*}$ is elliptic and $\Gamma$
  is smooth and compact, we have that $\bsh \in H^{s}(\Gamma; TM)$
  for all $s \in \RR$, by elliptic regularity.

  We now consider the \emph{single layer potential} $U := \maS_{\rm{ST}}(\bsh)$ associated
  to our fixed $\bsh$ satisfying Equation \eqref{eq.assumpt.h-first}. Then
  Proposition \ref{prop.Hkestimates} gives that the restrictions of $U$ to
  $\Omega_{+} := \Omega$ and to $\Omega_{-}$ satisfy
  \begin{equation}\label{eq.needed.est-first}
    U \ede (\, \bsu \ \ \ p\, )^{\top} \ede \maS_{\rm{ST}}(\bsh)
    \in H^{2}(\Omega_{\pm}; TM) \oplus H^{1}(\Omega_{\pm})\,.
  \end{equation}
  We need both restrictions, because we will study $U$ on both domains.
  \smallskip

  \noindent {\it We first study $U := \maS_{\rm{ST}}(\bsh)$ on $\Omega_{-}$,}
  for our fixed $\bsh$ satisfying Equation \eqref{eq.assumpt.h-first}.
  Theorem \ref{thm.jump.rel} gives that
  $[\bop U]_{-} = (\frac12   + \bsK^{*})\bsh = \lambda \bsnu $.
  (Recall that $[\bop U]_{-}$ is the trace on $\Gamma$ of $\bop U$ from the domain
  $\Omega_{-}$). Theorem \ref{thm.jump.rel} gives
  the ``no-jump relation'' $\bsu_{+} = \bsu_{-}$ at $\Gamma := \partial \Omega$
  (interior and exterior traces).
  We then notice that $\bsXi U = 0$ in $M \smallsetminus \Gamma$, by Proposition
  \ref{prop.compatibility}. The equation $\bsXi U = 0$ in $M \smallsetminus \Gamma$ and
  $V_{0} = 0$ in $\Omega$
  imply that $0 = \nabla^{*}\bsu - V_{0} p = \nabla^{*}\bsu$ on $\Omega$. Therefore,
  \begin{multline}\label{eq.aux.ip}
    \big(\bop U, \bsu \big)_{\Gamma}  \seq
    \int_{\Gamma} [\bop U]_{-} \cdot \bsu \,dS_{\Gamma}\seq \lambda \int_{\Gamma}
    \bsnu \cdot \bsu_{-} \,dS_{\Gamma}\\ \seq \lambda \int_{\Gamma}
    \bsnu \cdot \bsu_{+} \,dS_{\Gamma} \seq \lambda \int_{\Omega_{+}}
    \nabla^{*}u \dvol \seq 0\,.
  \end{multline}

  We have already noticed that $\bsXi U = 0$ on $M \smallsetminus \Gamma$.
  Equation \eqref{eq.needed.est-first} and \eqref{eq.aux.ip} show that the assumptions
  of Corollary \ref{cor.e.est} are satisfied on $\Omega_{-}$ (note that $\bsnu$
  changes sign, but this does not affect our assumptions).
  Because for every connected component $\Omega_{0}$ of
  $\Omega_{-}$ we have $V_{0} \not \equiv 0$ on $\Omega_{0}$
  and either $V$ does not vanish identically on $\Omega_{0}$ or there
  are no non-trivial Killing vector fields on $\Omega_{0}$, Corollary \ref{cor.e.est}
  then gives
  \begin{equation}\label{eq.aux.Omega-}
    \bsu = 0 \quad \mbox{and} \quad p= 0 \quad \mbox{in }\ \Omega_{-}\,.
  \end{equation}

  \noindent {\it Let us now study $U$ on $\Omega_{+} := \Omega$.}\
  Let $U = (\, \bsu \ \ \ p\, )^{\top} := \maS_{\rm{ST}}(\bsh)$, with $\bsh$ as above
  (and thus satisfying Equation \eqref{eq.assumpt.h-first}).
  Recall that $\bsXi U = 0$ in $M \smallsetminus \Gamma$.
  The fact that $\bsu_{+} = \bsu_{-} = 0$ at $\Gamma$ allows us to
  use again Corollary \ref{cor.e.est} to conclude that
  \begin{equation}\label{eq.aux.Omega+}
    \bsu = 0 \quad \mbox{and} \quad p= \mbox{constant in }\ \Omega_{+} := \Omega\,.
  \end{equation}
  (Anticipating the proof of (3), in case $V_{0}$ does not vanish
  identically on $\Omega$, we even obtain that
  $p = 0$ in $\Omega_{+} := \Omega$.)
  \smallskip

  \noindent {\it Let us now prove the needed properties of $\bsh$.}\
  We have already proved that $\bsu = 0$ in $M \smallsetminus \Gamma$
  (see Equations \eqref{eq.aux.Omega-} and \eqref{eq.aux.Omega+}), and hence
  $\Dnu \bsu = 0$ in $M \smallsetminus \Gamma$.
  The definition of $\bop$ and the properties of $U = (\, \bsu \ \ \ p\, )^{\top}
  := \maS_{\rm{ST}}(\bsh)$ then give
    $\bop \maS_{\rm{ST}}(\bsh) := -2 \Dnu \bsu + p \bsnu \seq p \bsnu$
  on $M \smallsetminus \Gamma$. The jump relation of Theorem \ref{thm.jump.rel}(3) then gives
  \begin{equation}\label{eq.p.limits-first}
    \bsh \seq [\bop \maS_{\rm{ST}}(\bsh)]_{-}
    - [\bop \maS_{\rm{ST}}(\bsh)]_{+} \seq (p_{-} - p_{+}) \nu \seq - p_{+} \nu \,.
  \end{equation}
  That is, $\bsh$ is a constant multiple of $\bsnu$. This proves Equation \eqref{eq.aux.incl}.
  As explained in Step 2, this gives that $\{\bsnu\}^{\perp} \subset \big(\frac12   + \bsK \big)
  \{\bsnu\}^{\perp}$, which completes the proof of point (2) of our theorem.
  \smallskip

  \noindent {\bf Step 4 (Proof of (3): $V_{0} \not\equiv 0$ on $\Omega_{+}$).}\
  Let us assume now that $V_{0}$ does not vanish identically
  on any component of $M \smallsetminus \Gamma$.
  To prove the point (3), it suffices to show that
  $\big(\frac12   + \bsK \big) L^{2}(M; TM) = L^{2}(M; TM)$.
  As is well-known (see, for instance, Lemma \ref{lemma.aux.v}), it enough to show
  that $\ker \big(\frac12   + \bsK^{*} \big)
  = 0$. Let now $\bsh \in L^{2}(M; TM)$ be such that $(\frac12   + \bsK^{*})\bsh = 0$ and
  $U := \maS_{\rm{ST}}(\bsh)$. All the assumptions for the previous step are
  valid, so all its conclusions remain valid. Moreover, the assumption $V_{0} \not\equiv 0$
  on $\Omega_{+}$ implies that $p = 0$ on $\Omega_{+}$ (as noticed
  above). Equation \eqref{eq.p.limits-first}
  then gives $\bsh = 0$. The proof is now complete.
\end{proof}


Now we turn to the analogous result for the single layer potential.
To that end, recall the upper-left corner operator $\maA$ appearing in the
matrix formula for $\psdinv$, Proposition \ref{prop.form.inverse} (see
also Theorem \ref{thm.form}). Consequently, the operator $\maA$ appears
also in the definition of our layer potentials,
see Remark \ref{rem.components} (but see also Definition \ref{def.lp}).
The associated limit operator is $\bsS := \maA_{0}$ \emph{the boundary single layer potential
operator} (see Theorem \ref{thm.main.jump1b} and Proposition \ref{prop.form.inverse}).
Its invertibility can be treated as in Theorem \ref{thm.K2} just proved.

\begin{theorem}\label{thm.S}
  Let Assumption \ref{assumpt.VV0} hold. Let $\bsS := \maA_{0}$
  be the boundary single layer potential operator, as usual. Then
  $\bsS^{*} = \bsS$.
  \begin{enumerate}[\rm (1)]
    \item Assume that $V_{0} = 0$ in $\Omega$. Then $\ker \bsS = \CC \bsnu$.

    \item On the other hand, if $V_{0} \not \equiv 0$ in $\Omega$
    (and hence $V_{0} \not \equiv 0$ in \emph{all}
    connected components of $M \smallsetminus \Gamma$), then $\bsS :
    L^2(\Gamma; TM) \to H^{1}(\Gamma; TM)$ is invertible.
  \end{enumerate}
\end{theorem}


\begin{proof}
  We have already proved that $\bsS^{*} = \bsS$.
  Recall that $\bsS$ is Fredholm on $L^{2}(\Gamma; TM)$
  by Theorem \ref{thm.Fredholm_first},
  because $\Gamma := \pa \Omega$ is compact.
  It is convenient to split our proof into {\it three} steps.
  The first two steps are devoted to the
  proof of (1). Recalling our assumptions, we have that
  $V_{0} \not \equiv 0$ on $M$ and either $V \not \equiv 0$ or $M$
  does not have non-zero Killing vector fields.
  Proposition \ref{prop.form.inverse} then gives
  that $\bsXi$ is invertible, and hence all the layer potentials are
  defined and no compatibility relations are needed for them to be
  solutions of the generalized Stokes operator $\bsXi := \bsXi_{V, V_{0}}$.
  \medskip

  \noindent {\bf Step 1 ($\bsnu \in \ker \bsS$ if $V_{0} = 0$ on
  $\Omega)$.}\
  We now assume that $V_{0} = 0$ on $\Omega_{+} := \Omega$, unless explicitly otherwise
  stated. Let $\bsh \in H^{1/2}(\Gamma; TM)$ be arbitrary. We begin by considering
  the \emph{single layer potential} $U$ with the density $\bsh$, which satisfies
  \begin{equation*}
    U \seq (\, \bsu\ \ \ p \,)^{\top} \ede \maS_{\rm{ST}}(\bsh)
    \in H^{2}(\Omega; TM) \oplus H^{1}(\Omega)\,,
  \end{equation*}
  by Proposition \ref{prop.Hkestimates}.
  We have that $\bsXi U = 0$ on $M \smallsetminus \Gamma$, again by
  Proposition \ref{prop.compatibility}.
  Let $W = (\, \bsw \ \ \ q \,)^\top := (\, 0 \ \ \ 1 \,)^\top $, so that $\bsXi W = 0$
  in $\Omega_{+}$ (recall that $V_{0} = 0$ in $\Omega_{+}$).
  We next use the Proposition \ref{prop.Green}(2) on $\Omega_{+}$.
  In that identity, $(\bsXi U, W)_\Omega = (U, \bsXi W)_\Omega = 0$, so the left hand side vanishes.
  Therefore, using that $\bsw = 0$, that $\bop W = - 2\Dnu\bsw + q \bsnu = \bsnu$, and that
  $\bsu_{+} = \bsS \bsh$ (by Theorem \ref{thm.jump.rel}), we obtain
  \begin{equation*} 
    0\seq (\bop  U, \bsw)_{\Gamma}  - (\bsu , \bop  W)_{\Gamma}  \seq 0
    - \int_{\Gamma} \bsu_{+} \cdot \bsnu \,dS_{\Gamma} \seq
    - \big ( \bsS \bsh, \bsnu \big )_{\Gamma} \seq
    - \big ( \bsh, \bsS\bsnu \big )_{\Gamma}\,.
  \end{equation*}
  By the density of $H^{1/2}(\Gamma; TM)$ in $L^{2}(\Gamma; TM)$,
  this shows that $\bsnu \in \ker \bsS$.
  \smallskip

  \noindent {\bf Step 2 ($\ker \bsS \subset \CC \bsnu$  if $V_{0} = 0$ on
  $\Omega$).}\ Let
  $\bsh \in L^{2}(\Gamma ; TM)$ be such that
  $\bsS\bsh = 0$. We continue to consider the \emph{single layer potential}
  $U = (\, \bsu \ \ \ p\, )^{\top} := \maS_{\rm{ST}}(\bsh),$ as in the previous step,
  except that now $\bsS\bsh = 0$.
  Since $\bsS$ is elliptic and $\Gamma$
  is smooth and compact, we have that $\bsh \in H^{s}(\Gamma; TM)$
  for all $s \in \RR$, by elliptic regularity.
  Proposition \ref{prop.Hkestimates} then gives that
  the restrictions of $U$ to $\Omega_{+} := \Omega$ and to $\Omega_{-}$
  satisfy
  \begin{equation}\label{eq.needed.est}
    U\vert_{\Omega_{\pm}} \ede (\, \bsu \ \ \ p\, )^{\top}\vert_{\Omega_{\pm}}
    \ede \maS_{\rm{ST}}(\bsh) \vert_{\Omega_{\pm}}
    \in H^{2}(\Omega_{\pm}; TM) \oplus H^{1}(\Omega_{\pm})\,.
  \end{equation}
  We know that $\bsXi U = 0$ in $\Omega_{\pm}$ by Proposition
  \ref{prop.compatibility}. Theorem \ref{thm.jump.rel} gives that
  $\bsu_{+} = \bsu_{-} =\bsS\bsh = 0$.
  Therefore $\bsu = 0$ in $\Omega_{-}$ and in $\Omega_{+}$, by Corollary \ref{cor.e.est}.
  The same corollary gives that $p$ is constant on each connected component
  of $M \smallsetminus \Gamma$ and that this constant is $0$ in $\Omega_{-}$, because
  $V_{0}$ is not identically equal to zero on any connected component of
  $\Omega_{-}$. We also obtain that $\Dnu \bsu = 0$ on $M \smallsetminus \Gamma$.
  The definitions of $\bop$ and $U$ then give
    $\bop U := -2 \Dnu \bsu + p \bsnu \seq p \bsnu$
  on $M \smallsetminus \Gamma$ and
  \begin{equation}\label{eq.p.limits1}
    \bsh \seq [\bop \maS_{\rm{ST}}(\bsh)]_{-}
    - [\bop \maS_{\rm{ST}}(\bsh)]_{+} \seq (p_{-} - p_{+})\bsnu \seq -p_{+}\bsnu \,.
  \end{equation}
  That is, $\ker \bsS \subset \CC \bsnu$.
  This completes the determination of $\ker \bsS$ and hence the proof of our
  theorem if $V_{0}$ vanishes identically on $\Omega_{+} := \Omega$.
  \smallskip

  \noindent {\bf Step 3 ($V_{0} \not\equiv 0$ on $\Omega_{+}$).}\
  Let $V_{0} \not \equiv 0$ in $\Omega$ and $\bsh \in \ker \bsS$,
  then the same arguments as in the last step give furthermore that $p_{+} = 0$,
  by Corollary \ref{cor.e.est}, since $V_{0} \not \equiv 0$ in $\Omega$.
  The proof is now complete.
\end{proof}

In both of the above theorems, the assumptions that $M$ and $\Omega$ are
connected do not really decrease the generality, in the sense that the study of
the Stokes equations can be reduced to this case (even if some properties of
the corresponding layer potentials might change). Note, however, that {\it we
are not assuming $\Omega_{-} := M \smallsetminus \overline{\Omega}$ to be
connected.}

Note also that the assumptions of Theorem \ref{thm.S} can be weakened by only requiring
that on $M$, either there are no non-zero Killing vector fields or $V\not \equiv 0$.

\subsection{The double layer operator in the case of connected boundary $\Gamma $}

The case $\maN \neq 0$ is much more complicated. We content our selves
with a particular case.

\begin{theorem}\label{thm.K2bis}
  Let $\Omega \subset M$, $\Gamma := \pa \Omega = \pa \Omega_{-} \neq \emptyset$,
  $\bsK \ede \bsP _{0}$, and $V$,
  be as in Theorem \ref{thm.K1}. (In particular,
  on every connected component of $\Omega_{-}$,
  either $V \not \equiv 0$ or there are no Killing vector fields.)
  We further assume that $\Gamma$ is connected and that $V_{0} = 0$.
  Then we have an isomorphism $\frac12   + \bsK : \{\bsnu\}^{\perp} \to \{\bsnu\}^{\perp}$.
\end{theorem}

The assumption on $V$ is needed to ensure that $\maN := \ker \bsXi$
is contained in the space of constant scalar fields.

\begin{proof}
  Recall that $\frac12 + \bsK$ is Fredholm on $L^{2}(\Gamma; TM)$ by Theorem
  \ref{thm.Fredholm_first}, because $\Gamma := \pa \Omega$ is compact.
  Let $\bsh \in H^{3/2}(\Gamma; TM)$. We begin by considering the
  \emph{double layer potential} $U := \maD_{\rm{ST}}(\bsh)$ of $\bsh$, which satisfies
  \begin{equation*}
    U \seq (\, \bsu\ \ \ p \,)^{\top} \ede \maD_{\rm{ST}}(\bsh)
    \in H^{2}(\Omega; TM) \oplus H^{1}(\Omega)\,,
  \end{equation*}
  by Proposition \ref{prop.Hkestimates}. The assumptions on $V$ and $V_{0}$ and
  Theorem \ref{thm.form} imply that the kernel
  $\maN$ of $\bsXi$ is the space of constants: $\maN = \{(0, c) \mid c \in \CC\}$.
  Corollary \ref{cor.compatibility} implies then that $\bsXi U = (0, c)$ on
  $M \smallsetminus \Gamma$, where
  $c \operatorname{vol}(M) = -(\bsh, \bsnu)_{\Gamma}$.

  Let $W = (\, \bsw \ \ \ q \,) := (\, 0 \ \ \ 1 \,)$. The assumption $V_{0} = 0$ on
  $\Omega$ gives $\bsXi W = 0$ on $M$. We have $(\bsXi U, W)_{\Omega} = c \operatorname{vol}(\Omega)$
  and $(U, \bsXi W)_{\Omega} = 0$. Proposition \ref{prop.Green}(2) on $\Omega_{+} = \Omega$
  gives
  \begin{equation*}
    (\bop  U, \bsw)_{\Gamma} - (\bsu , \bop  W)_{\Gamma} \seq
    \big ( \bsXi U, W \big )_{\Omega} - \big (U, \bsXi W \big )_{\Omega}
    \seq -\frac{\operatorname{vol}(\Omega)}{\operatorname{vol}(M)} (\bsh, \bsnu)_{\Gamma} \,.
  \end{equation*}

  Therefore, using that $\bsw = 0$,
  $\bop W = - 2\Dnu\bsw + q \bsnu = \bsnu$, and
  $\bsu_{+} = \big( \frac12  + \bsK\big ) \bsh$
  (by Theorem \ref{thm.K1}), we obtain
  \begin{equation*}
    \frac{\operatorname{vol}(\Omega)}{\operatorname{vol}(M)} (\bsh, \bsnu)_{\Gamma}
    \seq (\bsu , \bop  W)_{\Gamma} - (\bop  U, \bsw)_{\Gamma}  \seq \int_{\Gamma}
    \bsu_{+} \cdot \bsnu \,dS_{\Gamma} \seq
    \Big ( \Big(\frac12 + \bsK\Big) \bsh, \bsnu \Big )_{\Gamma}\,,
  \end{equation*}
  where the inner product in the last expression is the $L^{2}$--inner product on $\Gamma$.
  This shows that
  \begin{equation}\label{eq.aux2.incl1}
    \big(\frac12 \II + \bsK\big) \{\bsnu\}^{\perp} \subset \{\bsnu\}^{\perp}\,
  \end{equation} by the density of the space $H^{3/2}(\Gamma; TM)$
  in $L^{2}(\Gamma; TM)$.

  To complete the proof of our theorem, it is hence enough to prove
  the opposite inclusion to that of Equation \eqref{eq.aux2.incl1}, that is, that
  $\{\bsnu\}^{\perp} \subset \big(\frac12 \II + \bsK \big) \{\bsnu\}^{\perp}$.
  As in the proof of Theorem \ref{thm.K2}, in order to complete the proof of this
  desired opposite inclusion (and hence to complete the proof of our theorem),
  it suffices by Lemma \ref{lemma.aux.v} to show that
  \begin{equation}\label{eq.aux2.incl}
    \{ \bsh \in L^{2}(\Gamma; TM) \mid \big(\frac12   + \bsK^{*} \big)\bsh
    \in \CC \bsnu\} \subset \CC\bsnu\,.
  \end{equation}

  Let us prove Equation \eqref{eq.aux2.incl}, which will complete the proof of our theorem,
  as already noticed above. To this end, let $\bsh \in L^{2}(\Gamma ; TM)$
  be such that
  \begin{equation}\label{eq.assumpt.h}
    (\frac12   + \bsK^{*})\bsh \seq \lambda \bsnu \,,
  \end{equation}
  for some $\lambda \in \CC$. Since $\frac12   + \bsK^{*}$ is elliptic and $\Gamma$
  is smooth and compact, we have that $\bsh \in H^{s}(\Gamma; TM)$
  for all $s \in \RR$, by elliptic regularity.

  We now consider the \emph{single layer potential} $U := \maS_{\rm{ST}}(\bsh)$ associated
  to our fixed $\bsh$ satisfying Equation \eqref{eq.assumpt.h}. We first notice
  that $\bsXi U = 0$ in $M \smallsetminus \Gamma$ by Corollary \ref{cor.compatibility}.
  (or by Proposition \ref{prop.compatibility}).
  (Recall that the assumption on $V$ implies that $\maN := \ker \bsXi$ is contained in
  the space of constant scalar fields.)
  Then Proposition \ref{prop.Hkestimates} gives that the restrictions of $U$ to
  $\Omega_{+} := \Omega$ and to $\Omega_{-}$ satisfy
  \begin{equation}\label{eq.needed.est-2}
    U \ede (\, \bsu \ \ \ p\, )^{\top} \ede \maS_{\rm{ST}}(\bsh)
    \in H^{2}(\Omega_{\pm}; TM) \oplus H^{1}(\Omega_{\pm})\,.
  \end{equation}
  We need both restrictions, because we will study $U$ on both domains.
  \smallskip

  \noindent {\it We first study $U := \maS_{\rm{ST}}(\bsh)$ on $\Omega_{-}$,}
  for our fixed $\bsh$ satisfying Equation \eqref{eq.assumpt.h}.
  We have already noticed that $\bsXi U = \bsXi (\, \bsu \ \ \ p\, )^{\top} = 0$
  in $M \smallsetminus \Gamma$. Theorem \ref{thm.jump.rel} gives that
  $[\bop U]_{-} = (\frac12   + \bsK^{*})\bsh = \lambda \bsnu $.
  (Recall that $[\bop U]_{-}$ is the trace on $\Gamma$ of $\bop U$ from the domain
  $\Omega_{-}$). Theorem \ref{thm.jump.rel} gives also
  the ``no-jump relation'' $\bsu_{+} = \bsu_{-} = 0$ at $\Gamma := \partial \Omega$
  (interior and exterior traces).
  The equation $\bsXi U = 0$ in $M \smallsetminus \Gamma$ and $V_{0} = 0$ in $\Omega_{+}$
  imply that $0 = \nabla^{*}\bsu - V_{0} p = \nabla^{*}\bsu$ on $\Omega_{+}$. Therefore,
  \begin{multline}\label{eq.aux2.ip}
    \big(\bop U, \bsu \big)_{\Gamma}  \seq
    \int_{\Gamma} [\bop U]_{-} \cdot \bsu \,dS_{\Gamma}\seq \lambda \int_{\Gamma}
    \bsnu \cdot \bsu_{-} \,dS_{\Gamma}\\ \seq \lambda \int_{\Gamma}
    \bsnu \cdot \bsu_{+} \,dS_{\Gamma} \seq \lambda \int_{\Omega_{+}}
    \nabla^{*}u \dvol \seq 0\,.
  \end{multline}

  We have already noticed that $\bsXi U = 0$ on $M \smallsetminus \Gamma$.
  Equations \eqref{eq.needed.est-2} and \eqref{eq.aux2.ip} show that the assumptions
  of Corollary \ref{cor.e.est} are satisfied on $\Omega_{}$.
  Because every connected component $\Omega_{0}$ of
  $\Omega_{-}$ is such that either $V$ does not vanish identically on $\Omega_{0}$ or there
  are no non-trivial Killing vector fields on $\Omega_{0}$, Corollary \ref{cor.e.est}(3)
  then gives $\bsu = 0$ in $\Omega_{-}$. The same corollary gives that $p$ is constant
  in all connected components of $\Omega_{-}$.\smallskip

  \noindent {\it Let us now study $U$ on $\Omega_{+}$.}\
  Let $U = (\, \bsu \ \ \ p\, )^{\top} := \maS_{\rm{ST}}(\bsh)$, with $\bsh$ as above
  (thus satisfying Equation \eqref{eq.assumpt.h}).
  Recall that $\bsXi U = 0$ in $M \smallsetminus \Gamma$.
  Corollary \ref{cor.e.est} gives then that $p$ is constant on
  $\Omega$. (In case $V_{0}$ does not vanish identically on $\Omega$, we even obtain that
  $p = 0$ on $\Omega_{+} := \Omega$.) The fact that $\bsu_{+} = 0$ at $\Gamma$ allows us to
  use again Corollary
  \ref{cor.e.est} to conclude that $\bsu = 0$ in $\Omega = \Omega_{+}$. Recalling that we have
  already proved that $\bsu = 0$ in $\Omega_{-}$, we see that
  $\bsu = 0$ in $M \smallsetminus \Gamma$ and hence
  $\Dnu \bsu = 0$ in $M \smallsetminus \Gamma$.
  \smallskip

  \noindent {\it We are ready now to prove the needed properties of $\bsh$.}\
  The definition of $\bop$ and the properties of $U = (\, \bsu \ \ \ p\, )^{\top}
  := \maS_{\rm{ST}}(\bsh)$ then give
    $\bop U := -2 \Dnu \bsu + p \bsnu \seq p \bsnu$
  on $M \smallsetminus \Gamma$. The jump relation of Theorem \ref{thm.jump.rel}(3) then gives
  \begin{equation}\label{eq.p.limits}
    \bsh \seq [\bop \maS_{\rm{ST}}(\bsh)]_{-}
    - [\bop \maS_{\rm{ST}}(\bsh)]_{+} \seq (p_{-} - p_{+}) \nu \,.
  \end{equation}
  Because $\Gamma$ is connected, it follows that $\bsh$ is a constant multiple of $\bsnu$.
  This proves Equation \eqref{eq.aux2.incl}. As explained above, this gives
  that $\{\bsnu\}^{\perp} \subset \big(\frac12 + \bsK \big) \{\bsnu\}^{\perp}$,
  which completes the proof of of our theorem.
\end{proof}

\section{Applications: well-posedness of the Dirichlet problem
for the generalized Stokes system}
\label{sec.sec10}

We now present two applications of the results of the previous section.
First, we present the standard application to the well-posedness of the Stokes
system on a smooth manifold with boundary $\Omega$. The we use the well-posedness
result obtain the behavior of the co-normal derivative $\bop \maD_{\rm{ST}}(\bsh)$
at the boundary, in particular, of its ``no-jump'' property.

\subsection{Well-posedness for our generalized Stokes system}
In this section, we start with a given smooth, compact Riemannian manifold with
boundary $\overline{\Omega}$ (with interior denoted $\Omega$) and two smooth potential
functions $V_{0}, V :\overline{\Omega} \to [0, \infty)$. We allow both $V$ and $V_{0}$
to vanish identically on $\overline{\Omega}$. In fact, this is one of the cases of the
greatest interest. We use the method of Mitrea and Taylor, see, for instance, \cite{M-T1}.
We thus choose a smooth, compact, connected boundaryless manifold $M$ containing
$\overline{\Omega}$ and such that $\Omega$ is on one side of its boundary
$\Gamma := \pa \Omega$ (equivalently, $\Omega_{-}:= M \smallsetminus \overline{\Omega}$
also has boundary $\pa \Omega_{-} = \Gamma$). This ensures that the pair $(M, \Omega)$ satisfies
the assumptions of the previous section. A canonical choice for $M$ could be the ``double''
of $\overline{\Omega}$, obtained by gluing two copies of $\overline{\Omega}$
along their boundaries. We also extend $V$ and $V_{0}$ to smooth functions on $M$
and we assume that we can choose these functions to still be $\ge 0$ on $M$. This is certainly
possible if $V$ and $V_0$ are identically zero in $\Omega $. In order
to be able to apply the results of the previous section, we shall choose these extensions
to satisfy Assumption \ref{assumpt.VV0}. To this end, it is enough to take them not to
vanish identically on any connected component of $\Omega_{-}$.

The invertibility of the layer potential operator $\frac12 + \bsK$ established in the previous
section and the mapping properties of pseudodifferential operators on Sobolev spaces yield as
usual the following result which is classical for the usual Stokes operator (i.e. when
$V = 0$ and $V_{0} = 0$, see \cite[Proposition 10.5.1, Theorem 10.6.2]{M-W} in the case of a
bounded Lipschitz domain in $\mathbb R^n$, $n\geq 2$, and \cite[Theorem 5.1]{D-M} in the case
of a $C^1$ domain on a compact manifold).

\begin{theorem}\label{thm.main.WP}
  Let $\overline{\Omega}$ be a smooth, compact, connected manifold
  with boundary and let $V : \overline{\Omega} \to [0, \infty)$ be a
  smooth function that we assume to extend to a smooth, non-negative
  function on the {\it double} of $\overline{\Omega}$. We consider the
  generalized Stokes operator $\bsXi := \bsXi_{V, 0}$ {\rm (}thus $V_{0} = 0$
  in $\Omega${\rm )}. Let $\Gamma
  := \pa \overline{\Omega}$, as before. Then, for every $m \in \ZZ_{+}$ and
  for any $\bsf \in H^{m+1/2}(\Gamma; \overline{\Omega})$, such that
  $(\bsf, \bsnu)_{\Gamma } = 0$, there exists a solution
  $U = (\bsu \ \ p)^{\top} \in H^{m+1}(\Omega; T\overline{\Omega})
  \oplus H^{m}(\Omega)$ of the \emph{Dirichlet problem}
  \begin{equation}\label{eq.Dirichlet.p}
    \bsXi U \ede \bsXi_{V, 0} U \seq 0 \ \mbox { in } \Omega \
    \mbox{ and }\ \bsu \vert_{\Gamma} \seq \bsf \ \mbox { on } \pa \Omega \,.
  \end{equation}
  Any two solutions $U_{1}$ and $U_{2}$ of this Dirichlet problem differ by a
  constant scalar field: $U_{2} - U_{1} = (0 \ \ \ c)^{\top}$, $c \in \CC$. Also,
  any of the following two formulas provides a solution of this Dirichlet problem:
  \begin{equation*}
    U_{1} \ede \maD_{\rm{ST}}\Big(\Big(\frac{1}{2} + \bsK\Big)^{(-1)}\bsf\Big)
    \quad \mbox{or} \quad \mbox
    U_{2} \ede \maS_{\rm{ST}}\left(\bsS^{(-1)}\bsf\right)\,.
  \end{equation*}
  Moreover, there exists a constant $C_{m} \ge 0$ such that all solutions
  $U = (\bsu \ \ p)^{\top}$ satisfy
  \begin{equation*}
    \|\bsu\|_{H^{m+1}(\Omega; T\overline{\Omega})} +
    \Big \|p - \int_{\Omega} p \, \dvol \Big \|_{H^{m}(\Omega)} \le C_{m}
    \|\bsf\|_{H^{m+1/2}(\Gamma; T\overline{\Omega})}\,.
  \end{equation*}
\end{theorem}

\begin{proof}
  Let $M$ be any compact, connected, smooth manifold containing $\overline{\Omega}$
  such that $\overline{\Omega}$ is on one side of its boundary $\Gamma := \pa
  \overline{\Omega}$. For instance, we could take $M$ to be the double of $\overline{\Omega}$,
  as explained above. Our assumptions on $V$ and $V_{0}$ allow us to extend $V$ and $V_{0}$
  to $M$ so that they remain non-negative and do not vanish identically on any connected
  component of $\Omega_{-} := M \smallsetminus \overline \Omega$. (However, we take
  $V_{0} = 0$ in $\Omega$.) Then Assumption
  \ref{assumpt.VV0} will be satisfied and hence Theorem \ref{thm.K2} implies that the
  operator $\frac12   + \bsK$ is invertible on $\{\bsnu\}^{\perp}$. Moreover, its
  Moore-Penrose pseudoinverse $(\frac12   + \bsK)^{(-1)} \in \Psi^{0}(\Gamma; TM)$
  by Theorem \ref{thm.Fredholm_first}. Hence, $\bsh := (\frac12   + \bsK)^{(-1)}
  \boldsymbol f \in H^{m+1/2}(\Gamma ;TM)$ and $U := \maD_{\rm{ST}}(\bsh)$ satisfies the
  required properties by Proposition \ref{prop.compatibility}(3) (which assures that
  $\bsXi U \seq 0$ in $\Omega $) and Proposition \ref{prop.Hkestimates} (which, together
  with the elliptic regularity of the operator $\bsXi $, assures that
  $U\in H^{m+1}(\Omega; TM) \oplus H^{m}(\Omega)$).

  It remains to prove that $U$ is unique up to a constant with these properties. To this end,
  let $U_{1} = (\bsu_{1} \ \ p_{1})^{\top}$ be another solution of our Dirichlet problem
  \eqref{eq.Dirichlet.p}.
  Because $\bsu - \bsu_{1}= 0$ on $\Gamma $, Corollary \ref{cor.e.est} gives that $\bsu =
  \bsu_{1}$ and that $p-p_{1}$ is a constant. Both $U_{1}$ and $U_{2}$ in the statement are
  well defined and satisfy the boundary conditions, by the jump relations, Theorems \ref{thm.K1}
  and \ref{thm.jump.rel}. They also satisfy $\bsXi U_{1} = \bsXi U_{2} = 0$ in $\Omega$,
  by Proposition \ref{prop.compatibility} (recall that the kernel of $\bsXi$ is $\maN = 0$ in the
  case considered here, by Assumption \ref{assumpt.VV0}). The unicity of the solution up to a
  constant scalar field shows that we may take in the last equations $U$ to be any of
  the solutions $U_{1}$ or $U_{2}$ just discussed. The estimate then follows from
  Proposition \ref{prop.Hkestimates}.
\end{proof}

Let us consider now the case $V_{0} \not\equiv 0$ on $\Omega$, which turns out to
be easier. It is the case that is needed for our applications to manifolds with
cylindrical ends.

\begin{theorem}\label{thm.main.WP2}
  Let $\overline{\Omega}$ be a smooth, compact, connected manifold
  with boundary and let $V, V_{0} : \overline{\Omega} \to [0, \infty)$ be two
  smooth functions that we assume to extend to smooth, non-negative
  functions on the {\it double} of $\overline{\Omega}$. Let also $V_{0} \not \equiv 0$
  on $\Omega$ and consider the generalized Stokes operator $\bsXi := \bsXi_{V, V_{0}}$.
  Let $\Gamma := \pa \overline{\Omega}$, as before. Then,
  for every $m \in \ZZ_{+}$ and for any $\bsf \in H^{m+1/2}(\Gamma; \overline{\Omega})$,
  there exists a {\it unique} solution $U = (\bsu \ \ p)^{\top} \in H^{m+1}(\Omega; T\overline{\Omega})
  \oplus H^{m}(\Omega)$ of the \emph{Dirichlet problem}
  \begin{equation}\label{eq.Dirichlet.p2}
    \bsXi U \ede \bsXi_{V, V_{0}} U \seq 0 \ \mbox { in } \Omega \
    \mbox{ and }\ \bsu \vert_{\Gamma} \seq \bsf \ \mbox { on } \pa \Omega
  \end{equation}
  and
  \begin{equation*}
    U \ede \maD_{\rm{ST}}\Big(\Big(\frac{1}{2} + \bsK\Big)^{-1}\bsf\Big)
    \seq \maS_{\rm{ST}}\left(\bsS^{-1}\bsf\right)\,.
  \end{equation*}
  Moreover, there exists a constant $C_{m} \ge 0$ such that
  \begin{equation*}
    \|\bsu\|_{H^{m+1}(\Omega; T\overline{\Omega})} +
    \| p \|_{H^{m}(\Omega)} \le C_{m}
    \|\bsf\|_{H^{m+1/2}(\Gamma; T\overline{\Omega})}\,.
  \end{equation*}
\end{theorem}

The proof of this theorem is (almost) the same as that of Theorem
\ref{thm.main.WP}, only simpler, since the operators $\frac{1}{2} + \bsK$
and $\bsS$ are invertible.

\subsection{The boundary behavior of $\bop \maD_{\rm{ST}}$}
Let the assumptions of Theorem \ref{thm.main.WP} hold and let $H_{\bsnu}^{3/2}(\Gamma ;TM)
:= \{\bsf \in H^{3/2}(\Gamma ;TM) \mid (\bsf, \bsnu)_{\Gamma } = 0\}$. Then we define the
{\it Dirichlet-to-Neumann operator}
$\mathcal N_{\rm{ST}}:H_{\bsnu}^{3/2}(\Gamma ;TM) \to H^{1/2}(\Gamma ;TM)$ as follows
(see \cite[p.37]{Taylor2} for the case of the Laplace operator on a closed manifold).
For $\bsf \in H_{\bsnu}^{3/2}(\Gamma ;TM)$ arbitrary, let
$U_{0} := (\bsu \ \ p)^{\top}\in H^2(\Omega ;TM) \oplus H^1(\Omega)$ be
the unique solution of the Dirichlet problem
\begin{align} \label{eq.Dirichlet.0}
  \bsXi_{V, V_{0}} U_{0} \seq 0 \ \mbox{ in } \ \Omega  \ \
  \bsu \vert_{\Gamma} \seq \bsf \ \mbox{ on } \ \pa \Omega \,,
\end{align}
satisfying $\int_{\Omega } p \dvol = 0$, see Theorem \ref{thm.main.WP}. Then we let
\begin{equation}
\label{Dirichlet-oper}
  \mathcal N_{\rm{ST}}\bsf \ede [\bop U_0]_{+} \ \mbox{ on } \Gamma \,,
\end{equation}
the limit being evaluated from $\Omega $.

The next result shows that there is no jump for the conormal derivative
$\bop \maD_{\rm{ST}} (\bsf)$ across $\Gamma :=\pa \Omega $
(see \cite[Proposition 11.4]{Taylor2} in the case of the Laplace operator).

\begin{theorem} \label{jump-conormal-dl}
  We assume that the conditions of Assumption \ref{assumpt.VV0} are satisfied.
  Let $\bsf \in  H^{3/2}(\Gamma ;TM)$, $(\bsf, \bsnu)_{\Gamma } = 0$.
  Then there is no jump for $\bop \maD_{\rm{ST}} (\bsf)$ across $\Gamma $. More
  precisely,
  \begin{equation*}
    [\bop \maD_{\rm{ST}} (\bsf)]_{+} \seq \bop [\maD_{\rm{ST}} (\bsf)]_{-}
    \seq \Big(\frac12  + {\bsK}^{*}\Big)
    \mathcal N_{\rm{ST}}\bsf\,.
  \end{equation*}
\end{theorem}


\begin{proof}
  The kernel of $\bsXi$ vanishes because of Assumption \ref{assumpt.VV0}.
  There is therefore no need for the projection in Proposition
  \ref{prop.rep.formula}. Let $U_0 := (\bsu \ \ p)^{\top}$
  be the unique solution in $H^2(\Omega ;TM)\oplus H^1(\Omega)$ of the
  Dirichlet problem \eqref{eq.Dirichlet.0}, satisfying $\int_{\Omega } p \dvol = 0$,
  where $\bsf\in H^{3/2}(\Gamma ;TM)$
  is orthogonal to $\bsnu$, as in the statement of the Theorem.
  Because $\bsXi U_0=0$ in $\Omega$, $\bsu = \bsf$ on $\Gamma$,
  and $\mathcal N_{\rm{ST}}\bsf \ede
  [\bop U_0]_{+}$ on $\Gamma $, the second relation of Proposition
  \ref{prop.rep.formula} applied to $U_0$ implies that
  \begin{align} \label{cons.rep.invert}
    \maD_{\rm{ST}} \bsf(x) - \maS_{\rm{ST}} (\maN_{\rm{ST}} \bsf )(x) \seq
    \begin{cases}
      \ U_0(x) & \mbox{ if } x \in \Omega\\
      \ \ \, 0    & \mbox{ if } x \in M \smallsetminus \overline{\Omega} \,.
    \end{cases}
  \end{align}

  Now considering the vector part of Equation \eqref{cons.rep.invert},
  taking the limit of \eqref{cons.rep.invert} on $\Gamma $ from $\Omega $, and using Theorem
  \ref{thm.K1} and Theorem \ref{thm.jump.rel}, we obtain
  \[
  \Big(\frac{1}{2}   + \bsK\Big)\bsf - \bsS(\maN_{\rm{ST}} \bsf ) =
  \bsu \vert_{\Gamma} =\bsf\,,
  \]
  and hence
  \begin{equation*}
    \bsS(\maN_{\rm{ST}} \bsf ) \seq \Big(-\frac{1}{2}   + \bsK\Big)\bsf\,.
  \end{equation*}
  This identity and the ellipticity of the
  operators $\bsS$ and $-\frac{1}{2} + \bsK$ (see also Theorem \ref{thm.K1} and Theorem
  \ref{thm.jump.rel}) imply that $\maN_{\rm{ST}}$ is elliptic as well.

  Next we apply the operator $\bop $ to both sides of identity \eqref{cons.rep.invert}.
  Taking the limit at the boundary first from $\Omega $, we obtain
  \begin{align}\label{conormal-dl-plus}
    [\bop \maD_{\rm{ST}} (\bsf)]_{+} - [\bop \maS_{\rm{ST}} (\maN_{\rm{ST}}\bsf)]_{+}
    = [\bop U_0]_{+}
    = \maN_{\rm{ST}}\bsf\,.
  \end{align}
  Then, by taking the limit at the boundary from $\Omega_{-}$ gives
  \begin{align}\label{conormal-dl-minus}
    [\bop \maD_{\rm{ST}} (\bsf)]_{-} - [\bop \maS_{\rm{ST}}(\maN_{\rm{ST}}\bsf)]_{-} = 0\,.
  \end{align}
  Since both limits $[\bop \maS_{\rm{ST}}(\maN_{\rm{ST}}\bsf)]_\pm $ exist,
  by Theorem \ref{thm.jump.rel}, the formulas \eqref{conormal-dl-plus} and
  \eqref{conormal-dl-minus} show that the limits $[\bop \maD_{\rm{ST}} (\bsf)]_\pm $
  exist as well and that they are given by
  \begin{align*}
  &[\bop \maD_{\rm{ST}} (\bsf)]_{+} \seq \maN_{\rm{ST}}\bsf + \Big(-\frac12 + {\bsK}^{*}\Big)
  \maN_{\rm{ST}}\bsf\,,\\
  &[\bop \maD_{\rm{ST}} (\bsf)]_{-} \seq \Big(\frac12 + {\bsK}^{*}\Big)
  \maN_{\rm{ST}}\bsf \,,
  \end{align*}
  and hence
  \[
  [\bop \maD_{\rm{ST}} (\bsf)]_\pm \seq \Big(\frac12 + {\bsK}^{*}\Big)
  \maN_{\rm{ST}}\bsf\,.
  \]
  This completes the proof.
\end{proof}

As in the case of our well-posedness Theorem \ref{thm.main.WP}, we can extend this
result (with a proof that is almost word-for-word the same)
to the case when $V_{0} \not\equiv 0$ in $\Omega$, and the proof of the
result in this case is actually easier.

\begin{theorem} \label{jump-conormal-dl2}
  We assume that the conditions of Assumption \ref{assumpt.VV0} are satisfied
  and that $V_{0} \not\equiv 0$ in $\Omega$.
  Let $\bsf \in  H^{3/2}(\Gamma ;TM)$.
  Then there is no jump for $\bop \maD_{\rm{ST}} (\bsf)$ across $\Gamma $. More
  precisely,
  \begin{equation*}
    [\bop \maD_{\rm{ST}} (\bsf)]_{+} \seq \bop [\maD_{\rm{ST}} (\bsf)]_{-}
    \seq \Big(\frac12  + {\bsK}^{*}\Big)
    \mathcal N_{\rm{ST}}\bsf\,.
  \end{equation*}
  We also have 
    $\bsS\maN_{\rm{ST}} \seq -\frac{1}{2}   + \bsK.$
\end{theorem} 

\def\cprime{$'$}

\end{document}